\documentclass[12pt]{amsart}  

\usepackage{etex}
\usepackage[english]{babel}
\usepackage[cp1250]{inputenc}
\usepackage{amssymb,amsthm,amsfonts,amsmath}
\usepackage{mathrsfs}
\usepackage{verbatim}
\usepackage[shortlabels]{enumitem}
\usepackage{titletoc}
\usepackage[title]{appendix}

\usepackage{url}
\usepackage{amsopn}
\usepackage{graphicx}
\usepackage[usenames, dvipsnames]{xcolor}
\DeclareGraphicsExtensions{.pdf,.png,.jpg}
\allowdisplaybreaks

\usepackage[baseline]{euflag}

\definecolor{darkblue}{rgb}{0.0, 0.0, 0.55}
\usepackage[colorlinks,linkcolor=BrickRed,citecolor=OliveGreen,urlcolor=darkblue,hypertexnames=true]{hyperref}
\usepackage[a4paper,margin=2.5cm]{geometry}
\renewcommand{\qedsymbol}{\rule[.12ex]{1.2ex}{1.2ex}}
\def\ds{\displaystyle}

\DeclareMathOperator{\Span}{span}

\DeclareMathOperator{\Rank}{rank}
\DeclareMathOperator{\Pos}{Pos}
\DeclareMathOperator{\Sq}{Sos}

\DeclareMathOperator{\Vol}{Vol}
\DeclareMathOperator{\conv}{conv}
\DeclareMathOperator{\codim}{codim}
\DeclareMathOperator{\GL}{GL}
\DeclareMathOperator{\SO}{SO}

\DeclareMathOperator{\sq}{sq}

\DeclareMathOperator{\Tr}{tr}

\DeclareMathOperator{\Jac}{Jac}
\DeclareMathOperator{\grad}{grad}

\usepackage{xspace}

\newcommand{\crp}{cross-positive\xspace}
\newcommand{\ccrp}{completely cross-positive\xspace}
\newcommand{\expp}{exponentially-positive\xspace}
\newcommand{\cexpp}{completely exponentially-positive\xspace}
\newcommand{\exotic}{proper\xspace}

\newcommand{\sa}[1]{\mathbb S_{#1}(\RR)}

\newcommand{\psd}[1]{M_{#1}(\RR)_{\succeq0}}

\newcommand{\beq}{\begin{equation}}
\newcommand{\eeq}{\end{equation}}
\newcommand{\ben}{\begin{enumerate}}
\newcommand{\een}{\end{enumerate}}

\numberwithin{equation}{section}

\newcommand{\x}{{\tt x}}
\newcommand{\z}{{\tt z}}
\newcommand{\w}{{\tt w}}
\newcommand{\vv}{{\tt v}}
\newcommand{\PHI}{\underline{\phi}}
\newcommand{\PSI}{\underline{\psi}}
\newcommand{\ur}{{\underline r}}

\newcommand{\RR}{\mathbb R}

\newcommand{\NN}{\mathbb N}
\newcommand{\ZZ}{\mathbb Z}
\newcommand{\CC}{\mathbb C}

\newcommand{\cQ}{\mathcal Q}
\newcommand{\sym}{\mathbb S}

\newcommand{\y}{{\tt y}}
\newcommand{\e}{e}

\newcommand{\cU}{\mathcal U}
\newcommand{\cH}{\mathcal H}
\newcommand{\cK}{\mathcal K}

\newcommand{\PP}{\mathbb P}
\newcommand{\dd}{{\rm d}}

\newcommand{\cM}{\mathcal M}

\newtheorem{theorem}{Theorem}[section]
\newtheorem{corollary}[theorem]{Corollary}
\newtheorem{cor}[theorem]{Corollary}
\newtheorem{thm}[theorem]{Theorem}
\newtheorem{lemma}[theorem]{Lemma}
\newtheorem{lem}[theorem]{Lemma}
\newtheorem{proposition}[theorem]{Proposition}
\newtheorem{prop}[theorem]{Proposition}

\newtheorem{algorithm}[theorem]{Algorithm}

\theoremstyle{definition}

\newtheorem{remark}[theorem]{Remark}

\makeatletter
\newcommand{\pushright}[1]{\ifmeasuring@#1\else\omit\hfill$\displaystyle#1$\fi\ignorespaces}
\newcommand{\pushleft}[1]{\ifmeasuring@#1\else\omit$\displaystyle#1$\hfill\fi\ignorespaces}
\makeatother

\author[I. Klep]{Igor Klep${}^{1,Q}$}
\address{Igor Klep, Faculty of Mathematics and Physics,
 University of Ljubljana \& 
Faculty of Mathematics, Natural Sciences
and Information Technologies
University of Primorska, Koper
\&
 Institute of Mathematics, Physics and Mechanics, Ljubljana, Slovenia}
\email{igor.klep@fmf.uni-lj.si}
\thanks{${}^1$Supported by the 
Slovenian Research Agency program P1-0222 and grants 
J1-50002, J1-8132, J1-2453, N1-0217, J1-3004, J1-60011.}

\author[K. \v Sivic]{Klemen \v Sivic${}^2$}
\address{Klemen \v Sivic, Faculty of Mathematics and Physics, University of Ljubljana  \& Institute of Mathematics, Physics and Mechanics, Ljubljana, Slovenia}
\email{klemen.sivic@fmf.uni-lj.si}
\thanks{${}^2$Supported by the Slovenian Research Agency program P1-0222 and grants J1-8132, N1-0217, J1-2453, J1-3004, J1-60011.}

\author[A. Zalar]{Alja\v z Zalar${}^{3,Q}$}
\address{Alja\v z Zalar, 
Faculty of Computer and Information Science, University of Ljubljana  \& 
Faculty of Mathematics and Physics, University of Ljubljana  \&
Institute of Mathematics, Physics and Mechanics, Ljubljana, Slovenia}
\email{aljaz.zalar@fri.uni-lj.si}
\thanks{${}^3$Supported by the Slovenian Research Agency grants J1-50002, J1-2453, J1-8132, J1-3004, J1-60011 and P1-0288.}

\thanks{${}^Q$This work was performed within the project COMPUTE, funded within the QuantERA II Programme that has received funding from the EU's H2020 research and innovation programme under the GA No 101017733 {\normalsize\euflag}}


\date{\today}

\begin{document}

\numberwithin{equation}{section}

\dottedcontents{section}[3.8em]{}{2.3em}{.4pc} 
\dottedcontents{subsection}[6.1em]{}{3.2em}{.4pc}
\dottedcontents{subsubsection}[8.4em]{}{4.1em}{.4pc}

\begin{abstract}
A $\ast$-linear map $\Phi$ between matrix spaces is cross-positive if it is positive on orthogonal pairs $(U,V)$ of positive semidefinite matrices in the sense that $\langle U,V\rangle:=\Tr(UV)=0$ implies $\langle \Phi(U),V\rangle\geq0$, and is completely cross-positive if all its ampliations $I_n\otimes \Phi$ are cross-positive. (Completely) cross-positive maps arise in the theory of operator semigroups, where they are sometimes called exponentially-positive maps, and are also important in the theory of affine processes on symmetric cones in mathematical finance.

To each $\Phi$ as above a bihomogeneous form is associated by $p_\Phi(x,y)=y^T\Phi(xx^T)y$. Then $\Phi$ is cross-positive if and only if $p_\Phi$ is nonnegative on the variety of pairs of orthogonal vectors $\{(x,y)\mid x^Ty=0\}$. Moreover, $\Phi$ is shown to be completely cross-positive if and only if $p_\Phi$ is a sum of squares modulo the principal ideal $(x^Ty)$. These observations bring the study of cross-positive maps into the powerful setting of real algebraic geometry. Here this interplay is exploited to prove quantitative bounds on the fraction of cross-positive maps that are completely cross-positive. Detailed results about cross-positive maps $\Phi$ mapping between $3\times 3$ matrices are given. Finally, an algorithm to produce cross-positive maps that are not completely cross-positive is presented.
\end{abstract}

\title[Cross-positive linear maps, positive and sums of squares polynomials]{Cross-positive linear maps, positive polynomials\\[1mm] and sums of squares}

\maketitle

\tableofcontents

\section{Introduction}

Let $M_n(\RR)$ be the vector space of $n\times n$ real matrices equipped with the \textbf{involution $T$} which is the usual transposition of matrices.
We use $\psd n$ to denote the set of all positive semidefinite (symmetric) matrices. We let $I_n$ (resp.\ $0_n$) stand for the $n\times n$ identity (resp.\ zero) matrix.
A linear map $A: M_n(\RR)\to M_n(\RR)$ 
is $\ast$-linear if $A(U^T)=A(U)^T$ for all $U\in M_n(\RR)$. A
$\ast$-linear map $A$ is
 \textbf{positive} if it maps
positive semidefinite matrices into positive semidefinite matrices,
and is \textbf{completely positive} if the  ampliations 
$$
I_k\otimes A:
M_k(\RR)\otimes M_n(\RR)
\to M_k(\RR)\otimes M_n(\RR),\quad
U\otimes V\mapsto U\otimes A(V)
$$
are positive for every $k\in\NN$. Here $\otimes$ stands for the Kronecker tensor product of matrices.
Relaxing positivity of $A$ to the condition
\beq\label{eq:crp}
	\forall U,V\in\psd n: \langle U,V\rangle=0\quad\Rightarrow\quad\langle A(U),V\rangle\geq 0,
\eeq
where $\langle\textvisiblespace,\textvisiblespace\rangle$ denotes the standard scalar product
on $M_n(\RR)$, i.e., $\langle B,C\rangle:=\Tr(C^TB)$,
gives a definition of a \textbf{cross-positivity} of $A$ in which case $A$ is \textbf{\crp}.  Similarly, we call $A$ \textbf{\ccrp} if
\beq\label{eq:ccrp}
	\forall k\in\NN, \, \forall U,V\in\psd {nk}: \langle U,V\rangle=0\quad\Rightarrow\quad\langle (I_k\otimes A)(U),V\rangle\geq0.
\eeq
 \looseness=-1

In \cite{KOST} the authors construct, for the first time, a proper \crp map $A$, that is, a \crp map that is not \ccrp.
Such maps and the associated 
one-parameter semigroups 
(under composition) $\{\exp(tA) \colon t\ge 0\}$
of endomorphisms of a symmetric cone 
are an important ingredient in  the theory of affine processes on symmetric
cones. In the semigroup theory \crp (resp.\ \ccrp) maps are known as \expp (resp.\ \cexpp) maps
(see Section \ref{one-par-semigroups} for details).
Affine processes play a major
role in math finance \cite{CK-RMT16};
they are simple enough to be tractable from the point of view of theory and numerics, while at the same time sufficiently flexible from a modeling point of view. Affine processes on the cone of real positive semidefinite matrices were classified in \cite[Theorem 2.4]{CFMT11}, see also \cite[Theorem 2.19]{CK-RMT16} for the classification of affine processes on all symmetric cones. According to the classification, the linear drift of an affine process is given by a \crp map. The \crp map defining the drift is unique only modulo an integral with respect to a measure that describes jumps of the affine process. The operator defined by the integral is completely positive, so a drift defined by a \crp, but not completely, cross-positive map cannot be removed by a change of measure. See \cite{CK-RMT16} or \cite[Section 6]{KOST} for more details.

In this paper we investigate and quantify the gap between 
\crp maps and \ccrp maps, and provide an algorithm for providing further examples of proper \crp maps. In addition to matrix analysis our main tools include real algebraic geometry \cite{BCR98}, convexity \cite{Sch,BB05} and 
harmonic analysis \cite{Duo87}.

\subsection{Main results and readers' guide} 
In the preliminary Section \ref{sec:prelim}
we translate the properties of 
$\ast$-linear maps 
$A:M_n(\RR)\to M_n(\RR)$
to properties of biquadratic forms
\begin{equation}\label{eq:pA}
p_A:=\y^TA(\x\x^T)\y\in\RR[\x,\y],
\end{equation} 
where $\x=(\x_1,\ldots,\x_n),$ $\y=(\y_1,\ldots,\y_n)$ are tuples 
of commuting indeterminates. 
Then we explain that \ccrp maps are much tamer and easier to handle than
\crp maps, resembling the well-known relationship between positive and completely positive maps \cite{Cho75,Arv09,KMSZ19,tHM21}.

The main contribution of this article is three-fold.
First, 
we quantify the gap between \crp and \ccrp maps.
Roughly speaking, very few \crp maps are \ccrp.
More precisely, as shown in Corollary
\ref{verjetnost-intro},
	the probability $p_n$ that a 
	\crp map
		$M_n(\RR)\to M_n(\RR)$
	is \ccrp, is 
less than $(Cn)^{-\frac12\binom{n+1}{2}^2}$ for an absolute constant $C$,
so	\(\lim\limits_{n\to\infty}p_{n}=0.\)
Our proof roughly
follows Blekherman's outline 
in his papers characterizing the gap between positive and sum of squares polynomials
\cite{Ble06, BB05}.
A key new ingredient in the proof is a dimension-independent 
reverse H\"older inequality for bilinear biforms 
given in Section \ref{revHolder}.

Section \ref{sec:31}
considers the smallest nontrivial case, that is, the case
of \crp maps $A:M_3(\RR)\to M_3(\RR)$.
We give 
real algebraic geometry inspired
certificates (Nichtnegativstellens\"atze) 
for $A$ to be \crp;
see Theorem \ref{Nsatz-with-tr} for the case when $A$
satisfies some mild nonsingularity-type assumption,
and Corollary \ref{3x3e_1->0} for the singular case.\looseness=-1

Finally, 
in Section \ref{sec:algo},
as a side product of our analysis we provide a randomized polynomial-time algorithm 
based on semidefinite programming \cite{WSV00} for producing \exotic
\crp maps.

\section{Preliminaries}\label{sec:prelim}

\subsection{Cross-positivity in the language of operator semigroups}
\label{one-par-semigroups}
Consider a $\ast$-linear map $A:M_n(\RR)\to M_n(\RR)$. For each $t\in \RR$ the linear map $\exp(tA):M_n(\RR)\to M_n(\RR)$ is defined by $\exp(tA)=\sum_{i=0}^{\infty}\frac{1}{i!}(tA)^i$. The operator valued function $t\mapsto \exp(tA)$ is the solution of the differential equation $\dot{X}(t)=AX(t)$, which makes it important in analysis and applications to physics \cite{Lin76} and math finance \cite{CFMT11, CK-RMT16}.  The well-known formula $\exp((s+t)A)=\exp(sA)\circ \exp(tA)$ implies that the set 
$$\{\exp(tA) \colon t\ge 0\}$$ 
is a (one-parameter) semigroup under composition. The $\ast$-linear map $A$ is the \textbf{generator} of this one-parameter semigroup.
We call $A$ \textbf{\expp}, resp.~\textbf{\cexpp}, if 
$\exp(tA)$ is a positive, resp.~completely positive map
for all $t\geq 0$. In such a case the semigroup $\{\exp(tA)\colon t\ge 0\}$ is a \textbf{positive}, resp. \textbf{completely positive one-parameter semigroup}. Note the positivity of linear maps and their one-parameter semigroups is studied more generally over ordered vector spaces, in finite and infinite dimensions, and for bounded and unbounded linear operators.
We refer the reader to \cite{EHO,AN} for detailed studies.

The (complete) exponential positivity property can be rephrased in a more traditional matrix theory using (complete) cross-positivity.

\begin{thm}[\protect{\cite[Theorem 3]{SV70}}]
\label{thm:SV}
A $\ast$-linear map $A:M_n(\RR)\to M_n(\RR)$ is \expp if and only if
it is cross-positive.
\end{thm}

\begin{corollary}
    A $\ast$-linear map $A:M_n(\RR)\to M_n(\RR)$ is \cexpp if and only if it is \ccrp.
\end{corollary}

\begin{proof}
By definition, $A$ is \ccrp if and only if $I_k\otimes A$ is \crp for each $k\in \NN$. By Theorem \ref{thm:SV} this holds if and only if $I_k\otimes A$ is \expp for each $k\in \NN$, i.e., if and only if\looseness=-1
\begin{eqnarray*}
\forall k\in \NN, \forall t\ge 0,\forall X\succeq 0&:&\exp(t(I_k\otimes A))(X)=\sum_{i=0}^{\infty}\frac{1}{i!}t^i(I_k\otimes A)^i(X)\\
&=&\Big(I_k\otimes \sum_{i=0}^{\infty}\frac{1}{i!}(tA)^i\Big)(X)=(I_k\otimes \exp(tA))(X)\succeq 0.
\end{eqnarray*}
However, this is equivalent to complete positivity of $\exp(tA)$ for each $t\ge 0$, i.e., to complete exponential positivity of $A$.
\end{proof}

\subsection{Cross-positive maps and biquadratic biforms}
Let $n\geq 2$ and let $\sa n$ stand for the set of all real symmetric $n\times n$ matrices.
To each linear map $A:\sa {n}\to \sa {n}$ we assign the biquadratic form
$p_A\in\RR[\x,\y]$ as in \eqref{eq:pA}.
Let $I\subseteq\RR[\x,\y]$ be the ideal generated by
$\y^T\x=\sum_{i=1}^n\x_i\y_i$, and let $V(I)$ be the corresponding real variety
\[
	V(I):=\{(x,y)\in\RR^{n}\times\RR^n\mid y^Tx=0\}.
\]
The variety $V(I)$ is an irreducible hypersurface for $n\ge 2$ and the defining polynomial $\y^T\x$ changes sign on $\RR^{2n}$. Hence 
 the ideal $I$ is real radical \cite[Theorem 4.5.1]{BCR98}. Thus
$I$ is the vanishing ideal of $V(I)$, i.e., a polynomial $p\in\RR[\x,\y]$ vanishes on $V(I)$ if and only if $p\in I$.

A sum of a positive map and a map of the form 
\beq\label{eq:null}
A(X)=CX+XC^T \quad\text{for some }C\in M_n(\RR)\text{ and for all }  X\in M_n(\RR)
\eeq
is clearly \crp. The converse is true up to closure, see \cite[Lemma 6 and Theorem 2]{SV70}. It was long conjectured that each \crp map is a sum of a positive map and a map of the form \eqref{eq:null} (see \cite[Section 4]{Damm04} or \cite[p.409]{CFMT11}), but a counterexample was found in \cite{KOST}. Such counterexamples were called exotic \crp maps in \cite{KOST}. On the other hand, an analogous 
counterexample does not exist for \ccrp maps (see \cite[Theorem 3]{Lin76}).

The following is a special case of \cite[Corollary 15]{KOST} and \cite[Theorem 2]{TG13}, but can also be established by a straight-forward calculation.

\begin{lem}\label{lem:CX+XC^T}
For a linear map $A:\sa {n}\to \sa {n}$ we have $p_A\in I$ 
if and only if 
it is of the form \eqref{eq:null} for every $X\in \sa {n}$.
\end{lem}

The following lemma bounds the degrees of the forms needed in the sum of squares representations of biquadratic biforms modulo $I$.

\begin{lemma}\label{lem:sos-bilin-modI}
	Let a biquadratic biform $p\in \RR[\x,\y]$ be of the form 
		\begin{equation}\label{p-sos}
			p=\sum_{i=1}^k p_i^2+q
		\end{equation}
	for some $k\in \NN$, $p_i\in \RR[\x,\y]$ and $q\in I$. Then $p$ is a sum of squares of bilinear forms modulo the ideal $I$.
\end{lemma}

\begin{proof}
	The polynomial $q$ is of the form $q=r(\x,\y)\big(\sum_{i=1}^n \x_{i}\y_i\big)$ where $r(\x,\y)\in \RR[\x,\y]$.
	Let us write $p_{i,j}$ and $r_{j}$ to denote the homogeneous parts of $p_i$ and $r$ of degree $j$.  
	By comparing the degree $0$ parts of both sides of  $\eqref{p-sos}$ we conclude that $p_{i,0}\equiv 0$ for each $i$. 
	Polynomials $p_{i,1}(\x,\y)$ are of the form $p_{i,1}(\x,\y)=\sum_{\ell=1}^n (a_{i,\ell} \x_\ell + 
	b_{i,\ell} \y_{\ell})$ where $a_{i,\ell}\in \RR$, $b_{i,\ell}\in \RR.$ If any of 
	$a_{i,\ell}$ or $b_{i,\ell}$ is nonzero, then $\x_{\ell}^2$ or $\y_{\ell}^2$ should appear in $p$ with a positive coefficient, which is not true. 
	Hence, $p_{i,1}\equiv 0$ for each $i$ and consequently $r_{0}=r_1\equiv 0$.
	By comparing the degree 4 parts of both sides of \eqref{p-sos} we get 
		$p=\sum_{i=1}^k p_{i,2}^2 + r_{2}\big(\sum_{i=1}^n \x_{i}\y_i\big)$,
	where $p_{i,2}(\x,\y)$ and $r_{2}$ are linear combinations of monomials of the form $\x_{j_1}\x_{j_2}$, $\y_{j_1}\y_{j_2}$ and $\x_{j_1}\y_{j_2}$ for some $j_1,j_2\in \{1,\ldots,n\}$.
	Since $p$ is a biform of bidegree $(2,2)$, we conclude that only  monomials of the form $\x_{\ell_1}\y_{\ell_2}$ appear nontrivially in $p_{i,2}$ and $r_2$.
This proves the lemma.\looseness=-1
\end{proof}

We define the map $\Psi:(\x,\alpha)\mapsto(\x,\y)$ given by
	\begin{equation}\label{eq:VI}
	\begin{split}
		y_1	&=	\alpha_1 x_2,\\
		y_i	&=	\alpha_i x_{i+1}-\alpha_{i-1}x_{i-1} \quad\text{ for }i=2,\ldots,n-1,\\
		y_n	&=	-\alpha_{n-1}x_{n-1},
	\end{split}
	\end{equation}
	where
$\alpha=(\alpha_1,\ldots,\alpha_{n-1})$ is a tuple of commuting 
variables. 

Note that the image $\Psi(\RR^{2n-1})$ of $\Psi$ is dense in $V(I)$
in the usual Euclidean topology. This follows by noticing that every point in $V(I)$ 
can be approximated arbitrarily well by points with nonzero $\x_i$-coordinates, which are in 
$\Psi(\RR^{2n-1})$ since expressing $\alpha_i$ from the system \eqref{eq:VI} above is then well-defined.

Under the map $\Psi$ 
the biquadratic form $p_A\in\RR[\x,\y]$ of \eqref{eq:pA} 
corresponds to\looseness=-1
\begin{equation} \label{eq:qA}
q_A(\x,\alpha)=p_A\big(\Psi(\x,\alpha)\big)\in \RR[\x,\alpha],
\end{equation} 
which is 
a  form quartic in $\x$ and quadratic in $\alpha$.

\begin{proposition}\label{prop:obviousPQcp}
For a $\ast$-linear map 
$A:M_n(\RR)\to M_n(\RR)$ the following are equivalent:
\begin{enumerate}[\rm (i)]
\item
$A$ is \crp;
\item
$p_A\geq0$ on $V(I)$;
\item
$q_A\geq0$ on $\RR^{2n-1}.$
\end{enumerate}
\end{proposition}

\begin{proof}
The equivalence between (ii) and (iii) follows from the fact that $\Psi(\RR^{2n-1})$
is dense in $V(I)$ in the Euclidean topology.

(i)$\Rightarrow$(ii) Given $(x,y)\in V(I)$,
\[
	\langle xx^T,yy^T\rangle=\Tr(yy^Txx^T)=\Tr\big(y (y^Tx)x^T\big)=0.
\]
Hence
$	p_A(x,y)=\langle A(xx^T),yy^T\rangle\geq0
$
by assumption.

(i)$\Leftarrow$(ii) Assume $p_A$ is nonnegative on $V(I).$ Given $U,V\in\psd n$ with $\langle U,V\rangle=0$, write $U=\sum u_iu_i^T$ and $V=\sum v_iv_i^T$. 
As the scalar product of two positive semidefinite matrices is nonnegative, we deduce $\langle u_i,v_j\rangle=0$ for all $i,j$.
The assumption now implies $p_A(u_i,v_j)\geq0$. Then
\[
	 \langle A(U),V\rangle=\sum_{i,j} \langle A(u_iu_i^T),v_jv_j^T\rangle=\sum_{i,j} p_A(u_i,v_j)\geq0.
	\qedhere
\]
\end{proof}

We next give the counterpart of Proposition \ref{prop:obviousPQcp} for \ccrp maps.

\begin{proposition}\label{prop:obviousPQccp}
Let $A:\sa {n}\to \sa {n}$ be a linear map. The following are equivalent:
\begin{enumerate}[\rm (i)]
\item\label{part1-1805-1447}
$A$ extends to some \ccrp map $\widetilde{A}:M_n(\RR)\to M_n(\RR)$;
\item\label{part2-1805-1447}
$p_A$ is a sum of squares modulo $I$;
\item $q_A$ is a sum of squares.
\end{enumerate}
\end{proposition}

In the proof of the proposition we exploit Newton polytopes to restrict possible terms appearing in a sum of squares representation of $q_A$.

Let $\ur:=(r_1,\ldots,r_n)\in \ZZ^{n}_+$ stand for a tuple of nonnegative integers,
$\x^\ur$ for the monomial $\x_1^{r_1}\cdots \x_n^{r_n}$ and
$\conv(E)\subseteq \RR^n$ for the convex hull of the set $E\subseteq \RR^n$.
Recall that the \textbf{Newton polytope} $N(p)$ of a polynomial $p(\x)=\sum_{\ur} c_{\ur}\x^{\ur}$, 
where $c_\ur\in \RR\setminus\{0\}$, is the convex hull of the exponent vectors of the monomials
appearing nontrivially in $p$, i.e., 
	$$N(p)=\conv\big(\big\{\ur\colon \x^{\ur}\text{ has a nonzero coefficient in } p \big\}\big)\subseteq \RR^n.$$

\begin{proof}[Proof of Proposition \ref{prop:obviousPQccp}]
(i)$\Rightarrow$(ii): 
By \cite[Theorem 3]{Lin76}, 
$\widetilde{A}(X)=\widetilde\Phi(X)+CX+XC^T$ for some completely positive map $\widetilde\Phi$ and some $C\in M_n(\RR)$.
Using \cite[Proposition 3.1]{KMSZ19} for the restriction 
$\widetilde\Phi|_{\sa {n}}$ of $\widetilde\Phi$ to $\sa {n}$ and 
Lemma \ref{lem:CX+XC^T} for $X\mapsto CX+XC^T$, it follows that
$p_{\widetilde A}=p_A=\sum_{i=1}^kp_i^2+q$ for some bilinear forms $p_i$ and some biquadratic form $q\in I$, i.e., $p_A$ is a sum of squares modulo $I$ by Lemma \ref{lem:sos-bilin-modI}. 

(ii)$\Rightarrow$(i): Using \cite[Proposition 3.1]{KMSZ19} and Lemma \ref{lem:CX+XC^T}, 
$A(X)=\Phi(X)+CX+XC^T$ for some completely positive map $\Phi:\sa {n}\to \sa {n}$ and some $C\in M_n(\RR)$. Invoking Arveson's extension theorem \cite[Theorem 7.5]{Pau02}, there exists a completely positive extension $\widetilde{\Phi}:M_n(\RR)\to M_n(\RR)$ of $\Phi$, whence $\widetilde A(X)=\widetilde \Phi(X)+CX+XC^T$ is a \ccrp extension of $A$.

(ii)$\Rightarrow$(iii) is obvious, so we prove (iii)$\Rightarrow$(ii). 
First note the multi-homogeneity of $q_A$ implies that $q_A$ is a sum of squares of biforms that are quadratic in $\x$ and linear in $\alpha$. Write\looseness=-1
\begin{equation}\label{qA-2005-1904}
	q_A(x,\alpha)=\sum_{\ell=1}^{m}  q^{(\ell)}(\x,\alpha)^2,
\end{equation}
where 
	$\displaystyle  q^{(\ell)}(\x,\alpha)=
		\sum_{i=1}^{n-1}\sum_{1\leq j\leq k\leq n} c_{jk}^{(\ell,i)} \alpha_i\x_j\x_k$ 
for some $c_{jk}^{(\ell,i)}\in \RR$. 
It follows by definition that $q_A$ is a linear combination of the terms of the following forms:
\begin{itemize}
	\item $(\alpha_1 \x_2)^2 \; \x_j\x_k$, 
	\item $(\alpha_1 \x_2)(\alpha_{i}\x_{i+1}-\alpha_{i-1}\x_{i-1}) \x_j\x_k$,
	\item $\alpha_1\alpha_{n-1}\x_2\x_{n-1}\x_j\x_k$,
	\item $(\alpha_{i}\x_{i+1}-\alpha_{i-1}\x_{i-1})(\alpha_{\ell}\x_{\ell+1}-\alpha_{\ell-1}\x_{\ell-1})\x_j\x_k$,
	\item $(\alpha_{n-1} \x_{n-1})(\alpha_{i}\x_{i+1}-\alpha_{i-1}\x_{i-1}) \x_j\x_k$, 
	\item $(\alpha_{n-1} \x_{n-1})^2 \;\x_j\x_k$, 
\end{itemize}
where $i,\ell=2,\ldots,n-1$, $j,k=1,\ldots,n$.
By \cite[Theorem 1]{Rez78}, we have the inclusions $N(q^{(\ell)})\subseteq \frac{1}{2}N(q_A)$ of Newton polytopes,  which implies that
each $q^{(\ell)}$ is a linear combination of the monomials
\begin{equation}\label{mon-250521-0837}
	\alpha_1 \x_2 \x_j,\qquad \alpha_i \x_{i+1} \x_{j},\qquad \alpha_{i-1} \x_{i-1} \x_{j},\qquad \alpha_{n-1} \x_{n-1} \x_j,
\end{equation}
where $i=2,\ldots,n-1$ and $j=1,\ldots,n.$\\

\noindent \textbf{Claim.} Each
	$q^{(\ell)}(\x,\alpha)$
	can be expressed as a polynomial in the polynomials
		$\alpha_1\x_2$, $\alpha_2\x_3-\alpha_1\x_1$, $\ldots$, $\alpha_{n-1}\x_n-\alpha_{n-2}\x_{n-2}$,
		$\alpha_{n-1}\x_{n-1}$, $\x_1,\ldots,\x_n.$\\

\noindent\textit{Proof of Claim.}
We consider how each of the monomials in \eqref{mon-250521-0837} can appear in $q^{(\ell)}(\x,\alpha)$.
For $j=1,\ldots,n$, the monomials 
	$$\alpha_1 \x_2 \x_j=(\alpha_1 \x_2)\x_j\qquad \text{and} \qquad \alpha_{n-1} \x_{n-1} \x_j=(\alpha_{n-1}\x_{n-1})\x_j$$ 
can clearly by expressed as the claim suggests.
The formula 
	$$\alpha_i \x_i \x_{i+1}=\sum_{s=2}^{i}(\alpha_{s} \x_{s+1}-\alpha_{s-1}\x_{s-1})\x_s+(\alpha_1\x_2)\x_1$$
implies the same holds also for the monomials $\alpha_i \x_i \x_{i+1}$, $i=2,\ldots,n-1$.

For $i=2,\ldots,n-1$, it remains to consider the monomials 
\begin{equation}\label{rem-mon-250521-0927}
	\alpha_{i} \x_{i+1}\x_j, \; j\neq i,\qquad \text{and}\qquad \alpha_{i-1} \x_{i-1}\x_j, \;j\neq i.
\end{equation}
For $s=1,\ldots,n-2$ we define the vectors
\begin{align*}
	\widehat{\alpha}_s
		&=(\underbrace{0,\ldots,0}_{\substack{s-1\\ \text{ zeroes}}},\alpha_s,\alpha_{s+1},\underbrace{0,\ldots, 0}_{\substack{n-s-2\\ \text{ zeroes}}}),\\
	\widehat{\x}_s
		&=(\x_1,\ldots,\x_{s},0,\x_{s+2},\ldots,\x_n),\\
	\widehat{\y}_{s}
		&=
	(\underbrace{0,\ldots, 0}_{s \text{ zeroes}},
		\alpha_{s+1}\x_{s+2}-\alpha_s \x_s , 
		\underbrace{0, \ldots ,0}_{\substack{n-s-1\\\text{ zeroes}}}).
\end{align*}
If any of the monomials from \eqref{rem-mon-250521-0927} occurs in $q^{(\ell)}(\x,\alpha)$, then it also occurs in the polynomial $q^{(\ell)}(\widehat \x_{i-1},\widehat \alpha_{i-1})$ with the same coefficient.
By definition, 
\begin{equation*}
	q_A(\widehat{x}_{i-1},\widehat{\alpha}_{i-1})
		=\widehat{\y}_{i-1} A\left((\widehat{\x}_{i-1})^T\;\widehat{\x}_{i-1} \right) \widehat{\y}_{i-1}^T
		=(\alpha_{i}\x_{i+1}-\alpha_{i-1}\x_{i-1})^2 A\left((\widehat{\x}_{i-1})^T\; \widehat{\x}_{i-1} \right)_{ii}.
\end{equation*}
Hence, each 
	$q^{(\ell)}(\widehat{\x}_{i-1},\widehat{\alpha}_{i-1})$
vanishes on $V(\alpha_{i}\x_{i+1}-\alpha_{i-1}\x_{i-1})$. 
Since $\alpha_{i}\x_{i+1}-\alpha_{i-1}\x_{i-1}$ is irreducible in $\RR[\x,\alpha]$ and it changes sign on $\RR^{2n-1}$, it follows by \cite[Theorem 4.5.1]{BCR98} that
\begin{equation}\label{eq1-2005-1759}
	q^{(\ell)}(\widehat{\x}_{i-1},\widehat{\alpha}_{i-1})
	=(\alpha_{i}\x_{i+1}-\alpha_{i-1}\x_{i-1})p_{i-1}(\widehat{\x}_{i-1}),
\end{equation}
where $p_{i-1}$ is a linear form in $\widehat{\x}_{i-1}$. 
Now \eqref{eq1-2005-1759} implies that the monomials from \eqref{rem-mon-250521-0927}
can appear nontrivially in 
	$q^{(\ell)}(\widehat{\x}_{i-1},\widehat{\alpha}_{i-1})$
only from the scalar multiple of the term
	$$(\alpha_{i}\x_{i+1}-\alpha_{i-1}\x_{i-1})\x_j,$$
which concludes the proof of the claim.\hfill \ensuremath{\Box}\\

Using the Claim and \eqref{qA-2005-1904} it follows that $p_A(\x,\y)$ agrees on 
a dense subset of $V(I)$ and 
by continuity on the whole $V(I)$
with a sum of squares polynomial, 
which we denote by	$r(\x,\y)$. 
Since $p_A-r$ vanishes on $V(I)$, the polynomial $\y^T\x$ is irreducible in $\mathbb{R}[\x,\y]$ 
and its sign changes on $\mathbb{R}^{2n-1}$, it follows that $p_A-r\in I$ by \cite[Theorem 4.5.1]{BCR98}.
\end{proof}

\section{Nichtnegativstellens\"atze for the case $n=3$}\label{sec:31}

In this section we will establish, in the case $n=3$, some certificates of global nonnegativity for the form $q_A(\x,\alpha)$
of \eqref{eq:qA} or nonnegativity for $p_A(\x,\y)$ of \eqref{eq:pA} on $V(I)$. By Proposition \ref{prop:obviousPQcp}, this yields certificates
for a $\ast$-linear map $A:M_{3}(\RR)\to M_{3}(\RR)$ to be \crp.\looseness=-1

\begin{remark}
In the case $n=2$, 
	$$q_A(\x,\alpha)=\alpha_1^2 \cdot p_A((\x_1,\x_2),(\x_2,-\x_1))$$ 
and since 
$p_A((\x_1,\x_2),(\x_2,-\x_1))$ is a quartic form, it follows by \cite{Hil1} that $q_A$ is nonnegative if and only if
it is a sum of squares.
\end{remark}

For a matrix polynomial $A(\x)\in M_{m}(\RR[\x])$ we denote by $\Tr(A(\x))$ 
its \textbf{trace}, i.e., the sum of the diagonal entries.
For a ring $R$ we denote by $\sum M_m(R)^2$ the set of all finite sums of the expressions of the form
$G^T G$,
where $G\in M_m(R)$. 
Every element of $\sum M_m(R)^2$ is a
\textbf{sum of squares (sos) matrix polynomial}.
We say a symmetric matrix polynomial  $A(\x)\in M_{m}(\RR[\x])_{\textrm{sym}}$ is 
\textbf{positive semidefinite (psd) in $x\in \RR^n$} if $v^T A(x) v\geq 0$ for every $v\in \RR^m$, 
and write $A(x)\succeq 0$. We call $A(\x)\in M_{m}(\RR[\x])_{\textrm{sym}}$ \textbf{psd}
if it is psd in every $x\in \RR^n$.\looseness=-1

In this paragraph we connect, for every $n\in \NN$, 
global nonnegativity of $q_A(\x,\alpha)$ with positive semidefiniteness of a certain 
matrix polynomial. 
We denote by $\RR[\x]_{\hom}$ the set of homogeneous real polynomials in
$\x$. Since $q_A(\x,\alpha)$ is a quadratic form in $\alpha$ 
with coefficients from $\RR[\x]_{\hom}$, we can associate to it a symmetric matrix polynomial 
$Q_A\in M_{n-1}(\RR[\x]_{\hom})$ such that\looseness=-1 
\begin{equation}\label{eq:QA}
\alpha^T Q_A(\x)\alpha=q_A(\x,\alpha).
\end{equation} 

\begin{proposition}\label{prop-130921}
		The polynomial $q_A(\x,\alpha)$ is globally nonnegative if and only if $Q_A(\x)$ is positive
	semidefinite for all $\x$. 
\end{proposition}

\begin{proof}
	The statement follows by the equality \eqref{eq:QA}.
\end{proof}

\begin{remark}
Note that in the case $n=3$, Proposition \ref{prop-130921}
implies that
the parameterization \eqref{eq:VI} leads to the reduction of the problem of certifying cross-positivity of the map $A$ to certifying positivity of the $2\times 2$ matrix polynomial $Q_A$. Under the assumption that $Q_A$ does not vanish in any point of $\RR^3$ we establish such a certificate in Theorem \ref{Nsatz-with-tr} below.
\end{remark}

Let $A(\x)\in M_m(\RR[\x]_{\hom})$ be a matrix polynomial.
We call $x\in\RR^n$ a \textbf{zero} of $A(\x)$, if $A(x)$ is a zero matrix. 
A zero $x\in \RR^n$ of $A(\x)$ is \textbf{nontrivial}, if $x\neq 0$.
The following theorem is the first main result of this section. It is a certificate for $Q_A$ without nontrivial zeroes
to be psd in case $n=3$. 

\begin{theorem}\label{Nsatz-with-tr}
		Let 
	$Q(\x) \in M_2(\RR[\x_1,\x_2,\x_3]_{\hom})$ be a
	$2\times 2$ symmetric matrix polynomial over $\RR[\x]_{\hom}$, i.e., ${Q(\x)}^T=Q(\x)$.
	The following statements are equivalent:
	\begin{enumerate}[\rm (i)]
		\item\label{part1} $Q(\x)$ is positive semidefinite and does not have 
			nontrivial  real zeroes.
		\item\label{part2} $\Tr Q$ is strictly positive on 
			$\RR^3\setminus\{0\}$ and $\det Q$ is nonnegative on $\RR^3\setminus\{0\}$.
		\item\label{part3}  $\Tr Q$ is strictly positive on 
			$\RR^3\setminus\{0\}$ and there exists $N\in \NN$ such that 
			$\Tr(Q)^{N}\cdot \det Q$ is a sum of squares of forms.
		\item\label{part4} $\Tr Q$ is strictly positive on 
			$\RR^3\setminus\{0\}$ and there exists $N\in \NN$ such that 
			$$(\x_1^2+\x_2^2+\x_3^2)^{N}\cdot Q\in \sum M_2(\RR[\x])^2.$$
	\end{enumerate}
	Moreover, if all entries of $Q(\x)$ are of the same degree and $Q(\x)$ does not have nontrivial complex zeroes, 
	then \ref{part1}-\ref{part4} imply that:
	\begin{enumerate}[\rm (i)]
	\setcounter{enumi}{4}
		\item\label{part6} 
			If $J\subseteq \RR[\x]$ is the ideal in $\RR[\x]$
			generated by the polynomial $1-\x_1^2-\x_2^2-\x_3^2$, then
				$$Q\in \sum M_2(\RR[\x])^2+M_2(J).$$
	\end{enumerate}
\end{theorem}

\renewcommand\qedsymbol{$\square$}

\begin{proof}[Proof of the equivalences
	$\ref{part1}\Leftrightarrow \ref{part2}\Leftrightarrow\ref{part3}$ of 
	Theorem \ref{Nsatz-with-tr}] 
	Since the trace and the determinant of a matrix are the sum and the product 
	of the eigenvalues, respectively, the equivalence between \ref{part1}
	and \ref{part2} is easy to see. 
	The nontrivial implication $(\Rightarrow)$ in the equivalence 
	$\ref{part2}\Leftrightarrow\ref{part3}$	follows
	by \cite[Corollary 3.12]{Sch1}.
\end{proof}

\renewcommand\qedsymbol{$\blacksquare$}

We equip the set of matrix polynomials $M_m(\CC[\x])$ with the
conjugate transpose \textbf{involution} 
$\ast$  and write $M_m(\CC[\x])_{\textrm{her}}$ for the 
subset of $\textbf{hermitian}$ matrix polynomials, i.e., 
$F\in M_m(\CC[\x])$ with $F^\ast=F$.
In the proof of $\ref{part1}\Rightarrow \ref{part4}$ of Theorem \ref{Nsatz-with-tr} we will make use of the following factorization lemma.

\begin{lemma} \label{h-2-1-lema}
	For 
	$Q=\left[\begin{array}{cc} a & b \\ b^\ast & c\end{array}\right]
		\in M_2(\CC[\x])_{\textrm{her}}$ 
	the following equalities hold:
		\begin{eqnarray}
		a^4 Q
		&=&
			\left[\begin{array}{cc} a & 0 \\ b^\ast & a \end{array}\right]
			\left[\begin{array}{cc} a^3 & 0 \\ 0 & a(ac- bb^\ast)\end{array}\right]
			\left[\begin{array}{cc} a & b \\ 0 & a \end{array}\right],\label{factor1}\\
			\left[\begin{array}{cc} a^3 & 0 \\ 0 & a(ac-bb^\ast)\end{array}\right]
		&=&
			\left[\begin{array}{cc} a & 0 \\ -b^\ast & a\end{array}\right]
			Q
			\left[\begin{array}{cc} a & -b \\ 0 & a \end{array}\right].
		\label{factor2}
		\end{eqnarray}
\end{lemma}

\begin{proof}
	Easy computation.
\end{proof}

\begin{proof}[Proof of the equivalence
	$\ref{part1}\Leftrightarrow \ref{part4}$ of 
	Theorem \ref{Nsatz-with-tr}] 	
	The nontrivial implication is $(\Rightarrow)$.
	We write $Q=\left[\begin{array}{cc} a & b \\ b & c\end{array}\right]$.
	It is easy to check that
		$$Q =V\left[\begin{array}{cc} \Tr(Q) & i(a-c)+2b \\ i(c-a)+2b & \Tr(Q) 
			\end{array}\right]V^\ast,$$
	where 
		$V=\frac{1}{2}\left[\begin{array}{cc} 1 & i \\ i & 1\end{array}\right]$.
	By \eqref{factor1},  
		\begin{equation}\label{factor3}
			\Tr(Q)^4 \;Q =
			\widetilde{V}
			\left[\begin{array}{cc} \Tr(Q)^3 & 0 \\ 0 & \Tr(Q)(\Tr(Q)^2- dd^\ast)
			\end{array}\right]
			\widetilde{V}^\ast,
					\end{equation}
	where 
	$\widetilde V=V\left[\begin{array}{cc} \Tr(Q) & 0 \\ d^{\ast} & \Tr(Q) \end{array}\right]$
	and $d:=i(a-c)+2b$. A computation shows
		$\Tr(Q)^2- dd^\ast=4\det Q$.
	Since the left hand side of \eqref{factor3} belongs to $M_2(\RR[x]_{\hom})$,
	the right hand side equals 
		$$V_1
		\underbrace{\left[\begin{array}{cc} \Tr(Q)^3 & 0 \\ 0 & 4\Tr Q \det Q		
		\end{array}\right]}_{=:D}V_1^T+
		V_2
		\left[\begin{array}{cc} \Tr(Q)^3 & 0 \\ 0 &4 \Tr Q \det Q		
		\end{array}\right]V_2^T$$
	where $V_1, V_2\in M_2(\RR[x])$ are the real and imaginary parts of $\widetilde V$. 
	Now \ref{part4} follows by \cite[Corollary 3.12]{Sch1} since 
	there exists $N\in \NN$ large enough such that
	each form on the diagonal of $D$ multiplied by $(\x_1^2+\x_2^2+\x_3^2)^N$ is a sum
	of squares of forms.
\end{proof}

\begin{remark} \label{rem1}
 	If $Q$ in Theorem \ref{Nsatz-with-tr} is quartic (for example, $Q=Q_A$), 
	then $\Tr Q$ is a ternary quartic. 	
	Thus it is a sum of three squares by \cite{Hil1}.
	So in that case the exponent $N$ in \ref{part4} of 
	Theorem \ref{Nsatz-with-tr} depends only on
	$\det Q$ which is of degree 8.
	By \cite{Hil2} there is a positive form  $q$ of degree $4$ 
	such that $q\det Q$ is a sum of 
	squares of three forms.  Moreover, 
	$q^2 \det Q$ is a sum of squares of four forms \cite{Land}.
	See also \cite[p.\ 2830]{Rez05}.
\end{remark}

It remains to prove the implication $\ref{part1}\Rightarrow\ref{part6}$ 
in Theorem \ref{Nsatz-with-tr}. 
We will use Scheiderer's local-global principle \cite{Sch1}. For this aim we first prove 
the following proposition.

\begin{proposition} \label{pomozna}
	Assume in the notation of Theorem \ref{Nsatz-with-tr} that statement
	\ref{part1} holds and $Q(\x)$ does not have nontrivial complex zeroes.
 	Then for every 
	$x_0\in \CC^3\setminus\{0\}$ 
	there exists a polynomial $h\in \RR[\x]$ such that 
	$h(x_0)\ne 0$ and 
		\begin{equation*} 
			h^2Q\in \sum M_2(\RR[\x])^2+M_2(J).
		\end{equation*}
\end{proposition}

\begin{proof}[Proof of Proposition \ref{pomozna}]
	Let us write 
		$Q(\x)=\left[ 
			\begin{array}{cc}  a(\x) & b(\x) \\ b(\x) & c(\x)
			 \end{array} \right]$
	and choose $x_0\in \CC^3\setminus\{0\}.$
Since $Q$ is without nontrivial complex zeros, one of the following cases applies:
\begin{enumerate}
		\item\label{Case1} $a(x_0)\neq 0$.
		\item\label{Case2} $a(x_0)=0$ and $c(x_0)\neq 0$. 
		\item\label{Case3} $a(x_0)=c(x_0)=0$ and $b(x_0)\neq 0$. 
\end{enumerate}
\medskip

\noindent \textbf{Claim.} There exists an orthogonal matrix $U\in M_2(\RR)$ such that,
denoting $UQU^T=\left[ 
			\begin{array}{cc}  \widetilde a(\x) & \widetilde b(\x) \\ \widetilde b(\x) & \widetilde c(\x)
			 \end{array} \right]$, 
 $\widetilde a(x_0)\neq 0$.\\
\medskip

\noindent \textit{Proof of Claim.}
If we are in Case \eqref{Case1}, then we can take the identity matrix for $U$. If we are in Case \eqref{Case2},
then we take a permutation matrix for $U$. Finally, in Case \eqref{Case3}, we define 
$U=\frac{1}{\sqrt{2}}\left[\begin{array}{cc} 1 & 1 \\ 1 & -1\end{array}\right]$
and note that $\widetilde a(\x)=\frac{1}{2}(a(\x)+2b(\x)+c(\x))$. Hence, $\widetilde a(x_0)\neq 0$.\hfill \ensuremath{\Box}\\

By \eqref{factor1},   
	\begin{eqnarray*}
		\widetilde{a}^4\;  Q &=& 
		U^T\left[\begin{array}{cc} 
		\widetilde{a}& 0\\
		\widetilde{b} & 
		\widetilde{a}
		\end{array}\right]
		\left[\begin{array}{cc} 
		d_1 & 0\\
		0 & d_2
		\end{array}\right]
		\left[\begin{array}{cc} 
		\widetilde{a} & \widetilde{b}\\
		0 & \widetilde{a}
		\end{array}\right] U,
		\end{eqnarray*}
	where 
		\begin{equation*}
			d_1 =\widetilde{a}^3\in \RR[\x]\quad\text{and}\quad
			d_2 =	\widetilde{a}\left(\widetilde{a}\widetilde{c}
							- \widetilde{b}^2\right)\in 
							\RR[x].
		\end{equation*}
By \eqref{factor2}, 
	\begin{eqnarray*}
		\left[\begin{array}{cc} 
		d_1 & 0\\
		0 & d_2
		\end{array}\right] &=&
		\left[\begin{array}{cc} 
		\widetilde{a}& 0\\
		-\widetilde{b} & 
		\widetilde{a}
		\end{array}\right] U
		Q U^T
		\left[\begin{array}{cc} 
		\widetilde{a} & -\widetilde{b}\\
		0 & \widetilde{a}
		\end{array}\right].
	\end{eqnarray*}
		It follows that $d_1\geq 0$, $d_2\geq 0$ on $\RR^3$. 
	By \cite[Theorem 3.2]{Sch1}, $d_1$ and $d_2$ belong to $\sum \RR[\x]^2 + J$.
	This concludes the proof.
\end{proof}
	
\begin{proof}[Proof of the implication
	$\ref{part1}\Rightarrow \ref{part6}$ of 
	Theorem \ref{Nsatz-with-tr}] 	
	Let $R:=\RR[\x]/J$ be a quotient ring and let  $\Phi:R\rightarrow C(V(J),\RR)$ be the 
	natural map, i.e.,
	$\Phi(\check f)=f|_{V(J)},$ where
    $f\in \RR[\x]$, $\check f=f+J$
    and
    the variety $V(J)$ is the set
	$\left\{x\in \RR^3\colon \sum_{i=1}^3 x_i^2=1\right\}.$
	Let
		$$L:=\left\langle \check{h}^2\in \RR[\x]/J\colon h^2 Q\in \sum 
			M_2(\RR[\x])^2+ M_2(J)\right\rangle$$
	be an ideal in $\RR[\x]/J$.
	If $L$ were a proper ideal, then all its elements would have a common zero 
	$x_0\in \{x\in\CC^3\colon \sum_{i=1}^3 x_i^2=1\}$.
	By Proposition \ref{pomozna}, there exists $h\in \RR[\x]$ such that
	$h(x_0)\neq 0$ and $\check h^2\in L$.
	Hence $L$ is not a proper ideal and thus $L=\RR[\x]/J$.  
	In particular, there exist $\check h_1^2,\ldots ,\check h_k^2\in L$ such that $1+J\in \langle \check h_1^2,\ldots ,\check h_k^2\rangle$.
	By \cite[Proposition 2.7]{Sch1}, there exist $s_1,\ldots,s_k \in \RR[\x]$ with 
	$s_j>0$ on $V(J)$ 
 such that
	$\sum_{j=1}^k  s_j h_j^2\in 1+J$. 
	By \cite[Theorem 3.2]{Sch1}, $s_j\in \sum \RR[\x]^2+ J$.
	Hence
		$$\sum_{j=1}^k s_jh_j^2 (Q+M_2(J))=Q+M_2(J)\in  \sum M_2(\RR[\x])^2+ M_2(J),$$
	which concludes the proof.
\end{proof}

The following lemma, which holds in every dimension $n$, gives a special sufficient condition for a biquadratic form $p_A$ that is nonnegative on $V(I)$ to be a sum of a globally nonnegative biquadratic form and an element of $I$. 
Using the language of \cite{KOST}, a \crp map $A$ satisfying this condition is not exotic.
The lemma will be used in Corollary \ref{3x3e_1->0} to establish the second
main result of this section: a certificate for nonnegativity of $p_A$ on $V(I)$
in the case $n=3$ when $Q_A$ has a nontrivial real zero.

\begin{lemma}\label{e_1->0}
Let $A\colon M_n(\RR)\to M_n(\RR)$ be a cross-positive map and assume that there exists a nonzero vector $x_0\in \mathbb{R}^n$ such that 
	\begin{equation}\label{eq-030921-0748}
		\big(I_n-\frac{1}{||x_0||^2}x_0x_0^T\big)A(x_0x_0^T)\big(I_n-\frac{1}{||x_0||^2}x_0x_0^T\big)=0_n.
	\end{equation}
Then there exists $C\in M_n(\mathbb{R})$ such that the map $X\mapsto A(X)-CX-XC^T$ is positive.
\end{lemma}

For each $i=1,\ldots ,n$ let $e_i$ be the $i$-th element of the standard basis of $\RR^n$, i.e., the vector with 1 in the $i$-th component and 0 elsewhere.
We denote by $E_{ij}:=e_ie_j^T$ the standard $n\times n$ matrix units.

\begin{remark}\label{rem-030921-1458}
 In the proof of Lemma \ref{e_1->0} and Corollary \ref{3x3e_1->0} below we will use the following action of 
$\mathrm{GL}_n$ on the set of cross-positive linear maps $A:M_n(\RR)\to M_n(\RR)$:
	$$(g\cdot A)(X)=gA(g^{-1}Xg^{-T})g^{T}.$$
\end{remark}

\begin{proof}[Proof of Lemma \ref{e_1->0}]
By Remark \ref{rem-030921-1458} we can assume that $x_0=e_1$.
	Then \eqref{eq-030921-0748} means that $A(E_{11})$ is of the following form:
	$$A(E_{11})=\left[\begin{array}{cc} \ast &\ast\\
						\ast & 0_{n-1}
						\end{array}\right].$$
The idea of the proof consists of the following steps:
\begin{enumerate}
\item\label{030921-1138} There exists a matrix $C\in M_n(\RR)$ such that the map 
    $B:M_n(\RR)\to M_n(\RR)$, 
defined by 
	$B(X):=A(X)-CX-XC^T$, satisfies the following conditions:
	\begin{align}
	B(E_{11})&=0_n \label{to-prove1-030921},\\
	B(E_{1i}+E_{i1})e_1&=0 \quad\text{for }i=2,\ldots,n.\label{to-prove2-030921}
	\end{align}
\item\label{030921-1458} Using \eqref{to-prove1-030921} and cross-positivity of $B$ it follows that
	\begin{equation}\label{to-prove3-030921}
		B(E_{1i}+E_{i1})e_j=0\quad \text{for }i,j=2,\ldots,n.
	\end{equation}
\item\label{030921-1500} $y^TB(xx^T)y\geq 0$ for every $x,y\in \RR^n$. 
\end{enumerate}

To prove \eqref{030921-1138} first write 
	$C=\left[\begin{array}{ccc} \mathbf{c}_1& \cdots & \mathbf{c}_n\end{array}\right]$
in column form, where $\mathbf{c}_j$ are the columns of $C$ and note that
	\begin{equation}\label{030921-1139}
		CE_{ij}=\left[\begin{array}{ccc} 0_{n\times (j-1)} & \mathbf{c}_i & 0_{n\times (n-j)}\end{array}\right],
	\end{equation}
where $0_{k\times \ell}$ stands for the $k\times \ell$ zero matrix. 
Using \eqref{030921-1139} note that the condition \eqref{to-prove1-030921} determines $\mathbf{c}_1$, 
while conditions in \eqref{to-prove2-030921} determine columns $\mathbf{c}_i$, $i=2,\ldots,n$, i.e., 
\begin{align*}
	\mathbf{c}_1
	&=A(E_{11})e_1-\frac{1}{2}e_1^TA(E_{11})e_1\cdot e_1,\\
	\mathbf{c}_i
	&=A(E_{1i}+E_{i1})e_1-\frac{1}{2}e_1^TA(E_{1i}+E_{i1})e_1\cdot e_1-\frac{1}{2}e_1^TA(E_{11})e_1\cdot e_i\quad \text{for }i=2,\ldots,n.
\end{align*}

By \eqref{to-prove1-030921} it follows that $p_B(e_1,y)=0$ (where $p_B$ is as in \eqref{eq:pA}) for every $y\in \RR^n$.
In particular, $\frac{\partial p_B}{\partial y_i}(e_1,y)=0$ for every $y\in \RR^n$ and every $i=1,\ldots,n$.
Since $p_B(x,y)\ge 0$ on $V(I)$, for each $y\perp e_1$ there exists a Lagrange multiplier $\lambda (y)\in \mathbb{R}$ such that
\begin{equation}\label{030921-1450}
	\grad p_B(e_1,y)=\lambda(y)\grad h(e_1,y),
\end{equation}
where $h(\x,\y)=\y^T\x$.
In particular,
	$$0=\frac{\partial p_B}{\partial y_1}(e_1,y)=\lambda(y)\frac{\partial h}{\partial y_1}(e_1,y)=\lambda(y),$$
and using $\lambda(y)=0$ in \eqref{030921-1450} implies that
\begin{equation}\label{030921-1456}
	0=\frac{\partial p_B}{\partial x_i}(e_1,y)=y^TB(e_1e_i^T+e_ie_1^T)y\quad \text{for every }i=2,\ldots,n.
\end{equation}
Since $y$ is any vector orthogonal to $e_1$, this proves \eqref{030921-1458}. 

It remains to check \eqref{030921-1500}. We write $x=\lambda e_1+v$ and $y=\mu e_1+w$ for some $v,w\perp e_1$ and some $\lambda ,\mu \in \mathbb{R}$. Using \eqref{to-prove1-030921}--\eqref{to-prove3-030921} we see that $B(e_1e_i^T+e_ie_1^T)=0$ for each $i=1,\ldots ,n$ and thus $B(xx^T)=B(vv^T)=B((\nu e_1+v)(\nu e_1+v)^T)$ for each $\nu \in \mathbb{R}$. If $\mu \ne 0$, then $y^T(v-\frac{w^Tv}{\mu}e_1)=0$, therefore cross-positivity of $B$ implies
$$0\le y^TB\left(\left(v-\frac{w^Tv}{\mu}e_1\right)\left( v-\frac{w^Tv}{\mu}e_1\right)^T\right)y=y^TB(xx^T)y.$$
On the other hand, if $\mu =0$, then a sequence $\{z_k\}_{k\in \mathbb{N}}$, where $z_k:=\frac{1}{k}e_1+w$, satisfies
$z_k^Te_1\ne 0$ for each $k\in \mathbb{N}$ and $y=\lim _{k\to \infty}z_k$.  
By the above,
$z_k^TB(xx^T)z_k\ge 0$ for each $k\in \mathbb{N}$, therefore
$$y^TB(xx^T)y=\lim _{k\to \infty}z_k^TB(xx^T)z_k\ge 0.$$
This concludes the proof of Lemma \ref{e_1->0}.
\end{proof}

\begin{remark}
As a consequence of Lemma \ref{e_1->0} it follows that testing cross-positivity of a linear map is a NP-hard problem,
since this is true for testing positivity of a biquadratic form, see, e.g., \cite[Theorem 2.2]{LNQY09}. Indeed, let $A\colon M_n(\RR)\to M_n(\RR)$ be an arbitrary $\ast$-linear map. Let $D\colon M_{n+1}(\RR)\to M_n(\RR)$ be the map that deletes the first row and column and let $B\colon M_{n+1}(\RR)\to M_{n+1}(\RR)$ be defined by $B(X)=\left[ \begin{array}{cc}0 & 0 \\ 0 & A(D(X))\end{array}\right]$. It is clear that the map $A$ is positive if and only if $B$ is. However, $B(E_{11})=0$ and $B(E_{1i}+E_{i1})=0$ for $i=2,\ldots ,n+1$, hence the proof of Lemma \ref{e_1->0} implies that $B$ is positive if and only if it is \crp. 
Thus the problem of testing whether $B$ is \crp
is from a computational complexity viewpoint at least as hard as checking whether $A$ is positive, implying that testing 
cross-positivity of a map is NP-hard.
\end{remark}

\begin{cor}\label{3x3e_1->0}
Let $n=3$ and let $p\in \mathbb{R}[\x,\y]$ be a biquadratic form which is nonnegative on the real variety $V(I)$. Assume that there exist a nonzero vector $v_0\in \mathbb{R}^3$ and two linearly independent vectors $w_1,w_2\perp v_0$ such that $p(v_0,w_1)=p(v_0,w_2)=0$ or that there exist a nonzero vector $w_0\in \mathbb{R}^3$ and two linearly independent vectors $v_1,v_2\perp w_0$ such that $p(v_1,w_0)=p(v_2,w_0)=0$. Then $p\in \sum \mathbb{R}[\x,\y]^2+I$.
\end{cor}
\begin{proof}
By symmetry we can assume that there exist $v_0\in \mathbb{R}^3\backslash \{0\}$ and linearly independent vectors $w_1,w_2\perp v_0$ such that $p(v_0,w_1)=p(v_0,w_2)=0$. As $p$ is nonnegative on $V(I)$ it is equal to $p_A$ for some \crp map $A\colon M_3(\RR)\to M_3(\RR)$. 
By Remark \ref{rem-030921-1458} we may assume that $v_0=e_1$.
By the assumption of the corollary the quadratic form $(\lambda ,\mu )\mapsto p_A(e_1,\lambda w_1+\mu w_2)$ is positive semidefinite with zero coefficients at $\lambda ^2$ and at $\mu ^2$. Consequently, $p_A(e_1,\lambda w_1+\mu w_2)=0$ for all $\lambda ,\mu \in \mathbb{R}$, or equivalently, $w^TA(e_1e_1^T)w=0$ for each $w\perp e_1$. By Lemma \ref{e_1->0} there exists $C\in M_3(\mathbb{R})$ such that the map $B\colon M_3(\RR)\to M_3(\RR)$, defined by $B(X)=A(X)-CX-XC^T$, is positive. The biquadratic form $p_B$ satisfies the assumptions of \cite[Lemma 4.2]{Qua15}, so it is a sum of squares of bilinear forms. As $p_A-p_B\in I$ 
by the construction of $B$,
 this proves the corollary.\looseness=-1
\end{proof}

\section{Blekherman type volume estimates}

In this section we quantify the gap between \crp and \ccrp maps by extending the estimates on the volumes of compact sections of the cones of nonnegative 
biforms established in \cite{KMSZ19} to nonnegative biforms on the variety $V(I)$. 
The proofs are analogous to those in \cite{KMSZ19} and are inspired by \cite{Ble06,BB05}.\looseness=-1

Let $n\geq 3$ and $\RR[\x,\y]_{k_1, k_2}$ be the subspace of 
\textbf{biforms of bidegree} $(k_1,k_2)$, i.e., polynomials from $\RR[\x,\y]$ which are homogeneous of degree $k_1$ in 
$\x=(\x_1,\ldots,\x_n)$ and of degree $k_2$ in $\y=(\y_1,\ldots,\y_n)$.
Let 
	$$\mathcal Q:=\RR[\x,\y]_{2, 2}/(I\cap \RR[\x,\y]_{2, 2})$$ 
be the quotient space.  We write
\begin{eqnarray*}
	\Pos_\cQ^{(n)}	&:=& \left\{ p\in \cQ\colon p(x,y)\geq 0 \quad \text{for all }(x,y)\in V(I)\right\},\\
	\Sq_\cQ^{(n)}	&:=& \left\{ p\in \cQ\colon p-\sum_{i=1}^k p_i^2\in I\quad \text{for some }k\in \NN \text{ and }p_i\in \RR[\x,\y]_{1,1}\right\},
\end{eqnarray*}	
for the cone of polynomials  nonnegative on $V(I)$ and the cone of sums of squares on $V(I)$, respectively. Lemma \ref{lem:sos-bilin-modI} states that if a biform $p\in \RR[\x,\y]_{2,2}$ is a sum of squares in $\RR[\x,\y]/I$, then it is a sum of squares in $\cQ$.

We will estimate the gap between the cones $\Pos_\cQ^{(n)}$ and $\Sq_\cQ^{(n)}$ by comparing 
the volumes of suitably chosen compact sections of these cones.
First we have to carefully introduce an appropriate measure on the set $(\sym^{n-1}\times \sym^{n-1})\cap V(I)$
with respect to which we will integrate elements from $\cQ.$ This is the content of the next subsection.

\subsection{Definition of integration}

We define
	$$T:=(\sym^{n-1}\times \sym^{n-1})\cap V(I)$$
and equip it with the subspace topology.
Let $C(T)$ denote the vector space of all continuous functions on $T$.
The special orthogonal group $\SO(n)$ acts on the vector space $C(T)$ by rotating the coordinates, i.e., for $g\in \SO(n)$ and $f\in C(T)$, 
	\begin{equation}\label{SO(n)-action}
		g\cdot f(\x,\y):=f(g^{-1}\x,g^{-1}\y).
	\end{equation}
Choose a point $w:=(x,y)\in T$ and define a map 
	$\phi_w:\SO(n)\to T$ by $\phi_w(g)=gw=(gx,gy).$
Observe that since $n\geq 3$ the map is surjective 
and its kernel 
	$$\ker(\phi_w)=\{g\in \SO(n)\colon gx=x,\; gy=y\}$$ 
is homeomorphic to $\SO(n-2).$ 
We denote by $\hat \phi_w:\SO(n)/\ker(\phi_w)\to T$ the induced map of $\phi_w$ on $\SO(n)/\ker(\phi_w)$, which is the set of left cosets of $\ker(\phi_w)$ in $\SO(n)$.
Let $\widehat{\sigma}$ be the Haar measure on $\SO(n)$. We equip the quotient space
$\SO(n)/\ker(\phi_w)$ with the positive normalized $\SO(n)$-invariant measure 
$\sigma$ induced by $\widehat{\sigma}$ which exists and is unique. (See \cite[Theorem 1 on p.\ 138]{Nac65} and use the fact that 
compact groups are unimodular for uniqueness.)\looseness=-1

\begin{proposition}\label{existence-of-a-measure}
	The pushforward $(\hat \phi_w)_{\ast}(\sigma)$ 
	of $\sigma$ to $T$ is an $\SO(n)$-invariant measure.
\end{proposition}

\begin{proof}
Let $\Delta$ be a Borel subset of $T$ and $g\in \SO(n)$.
Then 
	$$(\hat \phi_w)_{\ast}(\sigma)(g\Delta)=
	\sigma((\hat \phi_w)^{-1}(g\Delta))=
	\sigma(g(\hat \phi_w)^{-1}(\Delta))=
	\sigma((\hat \phi_w)^{-1}(\Delta))=
	(\hat \phi_w)_{\ast}(\sigma)(\Delta),
	$$
where we used $\SO(n)$-invariance of $\sigma$ for the third equality and the following calculation
for the second one:
\begin{align*}
	(\hat \phi_w)^{-1}(g\Delta)
	&=\{g'\in \SO(n)/\ker(\phi_w)\colon g'w\in g\Delta \}
	=\{g'\in \SO(n)/\ker(\phi_w)\colon g^{-1}g'w\in \Delta \}\\
	&=g\{g''\in \SO(n)/\ker(\phi_w)\colon g''w\in \Delta \}
	=g(\hat{\phi}_w)^{-1}(\Delta).
	\qedhere
\end{align*}
\end{proof}

\begin{proposition}
	There exists a unique normalized $\SO(n)$-invariant measure on $T$.
\end{proposition}

\begin{proof}
	We already established the existence of a measure in 
	Proposition \ref{existence-of-a-measure}. It remains to prove the
	uniqueness. Let us assume to the contrary
	that $\mu_1$ and $\mu_2$ are two different normalized
	$\SO(n)$-invariant measures on $T$. Then
	$(\phi^{-1}_w)_\ast(\mu_1)$ and $(\phi^{-1}_w)_{\ast}(\mu_2)$
	are two different normalized
	$\SO(n)$-invariant measures on $\SO(n)/\ker(\phi_w)$. 
	But this contradicts the uniqueness of $\sigma$.
\end{proof}

From now on we will denote the measure 
$(\hat{\phi}_w)_{\ast}(\sigma)$
by $\sigma$.

\begin{remark}
\label{Stiefel-manifold}
In \cite{Mui82,Chi03} the set $T$ is known as the \textbf{Stiefel manifold} $V_{2,n}(\RR)$ of all $2$--frames in $\RR^n$,
i.e., the sets of all pairs of orthonormal vectors in $\RR^n$. Equivalently, $V_{2,n}(\RR)$ is the set of all real $n\times 2$ matrices $X$ such that $X^TX$
is the $2\times 2$ identity matrix.
Regarding $T$ as a manifold, it can also be equipped with the uniform normalized measure with respect to the action of the orthogonal group $\SO(n)$ \cite[\S 1.4.3]{Chi03}. This measure coincides with the measure $\sigma$ introduced above.
\end{remark}

\subsection{Estimates}
Now that we defined the measure on $T$, we can construct the appropriate sections
of $\Pos_\cQ$ and $\Sq_\cQ$ and present the volume estimates for these sections 
(see Theorems \ref{psd-intro}, \ref{squares-intro} below).

The $L^p$ norm  of a biform $f\in \cQ$ on $T$ is given by
	\begin{equation*}\label{measure}
		\left\| f \right\|_p^p= \int_{T} |f|^p\ \dd\sigma,
	\end{equation*} 
while the supremum norm is
	$$\|f\|_{\infty}:=\max_{(x,y)\in T} |f(x,y)|.$$

\noindent Let $\mathcal{H}^{(n)}_{\cQ}$ be the hyperplane
 of biforms from $\cQ$ of average 1 on $T$, i.e.,
	$$\mathcal H^{(n)}_{\cQ}=\left\{ f\in \cQ\colon \int_T f\; \dd\sigma=1 \right\}.$$
Let $\left(\Pos^{(n)}_\cQ\right)'$ and $\left(\Sq^{(n)}_{\cQ}\right)'$
be the sections
of the cones $\Pos^{(n)}_{\cQ}$ and $\Sq^{(n)}_{\cQ}$,
	\begin{eqnarray*}
		\left(\Pos^{(n)}_{\cQ}\right)' &=& \Pos^{(n)}_{\cQ}\;\bigcap\; \cH^{(n)}_{\cQ},\\
		\left(\Sq^{(n)}_{\cQ}\right)'	&=& \Sq^{(n)}_{\cQ}\;\bigcap\; \cH^{(n)}_{\cQ}.
	\end{eqnarray*}
Thus
$\left(\Pos^{(n)}_{\cQ}\right)'$ and $\left(\Sq^{(n)}_{\cQ}\right)'$ are convex and compact full-dimensional sets in the finite dimensional hyperplane $\cH^{(n)}_{\cQ}$.
For technical reasons we translate these sections by subtracting the polynomial
$(\sum_{i=1}^n x_i^2)(\sum_{i=1}^n y_i^2)$, i.e.,\\
\[
\begin{split}
	\widetilde{\Pos}^{(n)}_{\cQ}
			&=	\left\{ f\in \cQ\colon  f+(\sum_{i=1}^n x_i^2)(\sum_{i=1}^n y_i^2)
				\in \left(\Pos^{(n)}_{\cQ}\right)'\right\},\\
	\widetilde{\Sq}^{(n)}_{\cQ} 
			&= \left\{ f\in \cQ\colon  f+(\sum_{i=1}^n x_i^2)(\sum_{i=1}^n y_i^2)
				\in \left(\Sq^{(n)}_{\cQ}\right)' \right\}.
\end{split}
\]
Let $\cM:=\cM^{(n)}_{\cQ}$  be the hyperplane  of
biforms from $\cQ$ with average 0 on $T$,
	\begin{equation} \label{M-def}
		\cM=\left\{ f\in \cQ\colon \int_T f\; \dd\sigma=0 \right\}. 
	\end{equation}
Since $\sigma$ is normalized, 
\begin{equation*}
	\int_{T} \Big(\sum_{i=1}^n x_i^2\Big)
	\Big(\sum_{i=1}^n y_i^2\Big) \;\dd\sigma=
	1,
\end{equation*}
and hence
	$$\widetilde{\Pos}^{(n)}_{\cQ}\subseteq \cM
		\quad\text{and}\quad
		\widetilde{\Sq}^{(n)}_{\cQ}\subseteq \cM.$$
The natural $L^2$ inner product in $\cQ$ is defined by 
\begin{equation}
\label{L2-inner-product}
\langle f,g\rangle=\int_{T}fg\;\dd\sigma.
\end{equation}
With this inner product $\cM$ is a Hilbert subspace of $\cQ$ of dimension $D_{\cM}=\binom{n+1}{2}^2-n^2-1$
and so it is isomorphic to $\RR^{D_\cM}$ as a Hilbert space.
Let $S_{\cM}$, $B_\cM$ be the unit sphere and the unit ball in $\cM$, respectively. 
Let $\psi:\RR^{D_\cM}\to \cM$ be a unitary isomorphism and $\psi_{\ast}\mu$ the pushforward of the
Lebesgue measure $\mu$ on $\RR^{D_{\cM}}$ to $\cM$, i.e., $\psi_\ast\mu(E):=\mu(\psi^{-1}(E))$ for every
Borel measurable set $E\subseteq \cM$.

\begin{lemma} \label{unique-pushforward}
	The measure of a Borel set $E\subseteq \cM$ does not depend on the choice of the unitary isomorphism $\psi$, i.e.,
	if $\psi_1:\RR^{D_\cM}\to \cM$ and $\psi_2:\RR^{D_\cM}\to \cM$ are unitary isomorphisms, then
	$(\psi_1)_\ast\mu(E)=(\psi_2)_\ast\mu(E)$.
\end{lemma}

	The proof of Lemma \ref{unique-pushforward} is the same as the proof of \cite[Lemma 1.4]{KMSZ19}.\\

We are now ready to compare the volumes of the 
sections defined above. 
The lower bound for the volume of the section of nonnegative biforms from $\cQ$ is as follows:

\begin{theorem}\label{psd-intro}
	For $n\in \NN$, 
		$$\frac{3^{3}\cdot 10^{-\frac{20}{9}}}{\sqrt{n}}\leq 
		\left(\frac{\Vol \widetilde{\Pos}^{(n)}_{\cQ}}{\Vol B_{\cM}}\right)^{\frac{1}{D_{\cM}}}.$$
\end{theorem}

Next we give the upper bound for the volume of the section of sums of squares biforms from $\cQ$:

\begin{theorem} \label{squares-intro}
	For integers $n\geq 3$,
		$$\left(\frac{\Vol \widetilde{\Sq}^{(n)}_{\cQ}}{\Vol B_\cM}\right)^{\frac{1}{D_{\cM}}}
			\leq 2^{3}\cdot 3\cdot 6^{\frac{1}{2}}\cdot \frac{1}{n}.$$
\end{theorem}

Combining the previous two theorems we obtain:

\begin{corollary}\label{cor:ratio}
For integers $n\geq 3$,
		$$	\left(\frac{\Vol \widetilde{\Sq}^{(n)}_{\cQ}}{\Vol \widetilde{\Pos}^{(n)}_{\cQ}}\right)^{\frac{1}{D_{\cM}}}
			\leq 
		\frac
  {2^{5}\cdot 2^{\frac{1}{2}}\cdot 5^2 \cdot 10^{\frac{2}{9}} }
  {3^{\frac{3}{2}}\cdot \sqrt{n}}.$$
\end{corollary}

In the language of cross-positive and completely cross-positive maps, Corollary \ref{cor:ratio}
can be stated in the following form.

\begin{corollary}	\label{verjetnost-intro}
	For every $n\in \NN$ the probability that a 
	cross-positive map
		$\Phi:M_n(\RR) \to M_n(\RR)$
	is completely cross-positive, is bounded above by
		$$ p_{n}<
		\left(
  \frac{2^{5}\cdot 2^{\frac{1}{2}} \cdot 5^2\cdot 10^{\frac{2}{9}}}{3^{\frac{3}{2}}\cdot \sqrt{n}}\right)^{D_\cM}.$$
	In particular, 
	\(\lim\limits_{n\to\infty}p_{n}=0.\)
\end{corollary}

\noindent Here, the probability $p_{n}$ is defined as the ratio between the volumes of the sections 
$\widetilde{\Sq}_{\cQ}^{(n)}$ and $\widetilde{\Pos}_{\cQ}^{(n)}$ in $\cM$.

\begin{remark}
    \begin{enumerate}
        \item 
        The correspondence \eqref{eq:pA}
        between $\ast$-linear maps 
        $A:M_n(\RR)\to M_n(\RR)$ and biquadratic biforms $p_A$
        is bijective only when maps are restricted to symmetric matrices $\sym_n(\RR)$.
        Since nonnegativity of $p_A$ on $V(I)$ is equivalent to $A$ being cross-positive (see Proposition \ref{prop:obviousPQcp}), while $p_A$ being a sum of squares modulo $I$ is equivalent to the existence of some completely cross-positive extension 
        $\widetilde A:M_n(\RR)\to M_n(\RR)$ of the restriction of $A$ to $S_n(\RR)$ (see Proposition \ref{prop:obviousPQccp}), $p_n$ is clearly an upper bound for the probability that a cross-positive map is completely cross-positive.
        \item If we want to compare the sizes of two cones $K\subseteq L\subseteq \RR^n$ in a fixed metric, then the most unbiased choice of a compact set $C$ to compare the sizes of $K\cap C$ and $L\cap C$ is the unit ball $B$ of this metric. In our case, the metric is the $L^2$ norm, coming from the inner product \eqref{L2-inner-product}. In this norm, the condition $f\in B$ is given by a quadratic inequality in the coefficients of $f$ and therefore sharp lower and upper bounds on $K\cap B$ (resp.\ $L\cap B$) following the same asymptotics
 are difficult to establish. Replacing the unit ball $B$ with a hyperplane whose normal is some vector from the unit sphere leads to more manageable conditions. The choice of the hyperplane is not arbitrary, since its position can have a large impact on the size difference of the intersections, e.g.,\ if the normal is almost perpendicular to some ray on the boundary of the larger cone $L$, then the difference in size can be very large, even if the smaller cone is not significantly smaller.
However, if there is a vector in the interior of both cones, which is fixed by all symmetries for each cone, then the orthogonal complement of this vector is a fair choice of the hyperplane to capture size difference between the cones.
In our case, the polynomial
$(\sum_{i=1}^n x_i^2)
	(\sum_{i=1}^n y_i^2)$
is a fixed point for the action of the orthogonal group $O(n)$
on $\RR[\x,\y]_{2,2}$,
defined by
$O\cdot p(\x,\y)=p(O^{-1}\x,O^{-1}\y)$.
Note that the ideal $I$ is invariant under this action and therefore the action extends naturally to the action on $\cQ$. It is clear that the sets
$\Pos_{\cQ}^{(n)}$ and
$\Sq_{\cQ}^{(n)}$ are invariant under this action and therefore comparing their sizes by intersecting them with $\cH_{\cQ}^{(n)}$ is an appropriate choice.
        \item Blekherman \cite[Theorem 6.1]{Ble06} established
volume bounds for sum of squares forms.
Our proofs of Theorems \ref{psd-intro} and 
\ref{squares-intro}
freely borrows his ideas.
An important ingredient in the proof of Theorem \ref{squares-intro} is also a new version of the Reverse H\"older inequality, which we prove in Section \ref{revHolder} below. 
        \item 
In \cite{Ble06} Blekherman proved that for a fixed degree bigger than 2 the ratio between the volume radii of compact sections of the cones of sum of squares forms and nonnegative forms goes to 0, as the number of variables goes to infinity. 
Corollary \ref{cor:ratio} is an analog of his result for sum of squares biquadratic forms and nonnegative biquadratic forms on Stiefel manifolds $V_{2,n}(\RR).$ (See Remark \ref{Stiefel-manifold}.)
    \end{enumerate}
\end{remark}

Let $V$ be a real vector space. Recall that, for a convex body $\cK$ with the origin in its interior,
the \textbf{gauge} $G_{\cK}$ is defined by
	$$G_{\cK}:V\to\RR,\quad
		G_{\cK}(p)=\inf\left\{\lambda>0\colon p\in \lambda\cdot \cK\right\}.$$

\begin{proof}[Proof of Theorem \ref{psd-intro}]
We denote $\cK=\widetilde{\Pos}^{(n)}_{\cQ}$. 
As in \cite[\S 2.1.1]{KMSZ19}
we establish that
	\begin{equation*}\label{inequality}
		\left( \frac{\Vol \cK}{\Vol B_{\cM}} \right)^{\frac{1}{D_{\cM}}}\geq
			\left(\int_{S_{\cM}} \left\|f\right\|_\infty \dd\widetilde\mu\right)^{-1}
	\end{equation*}
where $\widetilde\mu$ is a rotation invariant probability measure on $S_{\mathcal{M}}$.
The proof of the inequality in Theorem \ref{psd-intro} now reduces to proving the following claim.
	\\

\noindent \textbf{Claim:}
	$\ds\int_{S_{\cM}}\left\|f\right\|_\infty \dd\widetilde\mu \leq 
	3^{-3}\cdot 10^{\frac{20}{9}}\cdot n^{\frac{1}{2}}$.\\

To prove this claim we will use \cite[Corollary 2]{Bar02}.
Let
	$(\RR^n\otimes \RR^n)^{\otimes 2}$ 
be the $2$-nd tensor power of $\RR^n\otimes \RR^n$ .
Let  $e_1,e_2\in \RR^n$ be standard unit vectors
and let $w$ be the tensor
	$$w:=(e_1 \otimes e_2)^{\otimes 2}\in (\RR^n\otimes \RR^n)^{\otimes 2}.$$
The group $\SO(n)$ acts on	$(\RR^n\otimes \RR^n)^{\otimes 2}$ by the natural diagonal action, i.e., 
for $g\in \SO(n)$ and all $x_i\in \RR^n$, 
	$$g(x_1\otimes \cdots \otimes x_{4})=gx_1\otimes \cdots \otimes gx_{4}$$
and extend by linearity. 
We also define
	$$v:=w-q,\quad\text{where }q=\int_{g\in \SO(n)} gw\; \dd \widehat\sigma (g),$$
and we integrate with respect to the Haar measure $\widehat \sigma$ on $\SO(n)$.
As in \cite[Example 1.2]{BB05}, we proceed as follows:
\begin{enumerate}[label={\rm(\arabic*)}]
	\item We identify the vector space of biforms from $\cQ$
		with the vector space $V_1$ of the restrictions of linear functionals
			$\ell:(\RR^n\otimes \RR^n)^{\otimes 2}\to \RR$
		to the orbit 
	$$\SO(n)w=\{(x\otimes y)^{\otimes 2}\colon \left\|x\right\|=\left\|y\right\|=1,\;y^Tx=0	
	\}.$$
	Note also that $\SO(n)(e_1\otimes e_2)=T$.
	\item We identify the vector space of biforms from $\cM$
		with the vector space $V_2$ of the restrictions of linear functionals
			$\ell:(\RR^n\otimes \RR^n)^{\otimes 2}\to \RR$
		to $B=\SO(n)v.$
	\item We introduce an inner product on $V_2$
		by defining
			$$\langle \ell_1,\ell_2\rangle:=\int_{g\in \SO(n)} \ell_1(g v)\cdot \ell_2(gv)\;\dd \widehat\sigma(g).$$
		This inner product also induces the dual inner product on the dual space $V_2^\ast\cong V_2$ which we
		also denote by $\langle\cdot,\cdot\rangle$.
\end{enumerate}

By  \cite[Corollary 2]{Bar02},
	$$\left\|f\right\|_\infty\leq (D_k)^{\frac{1}{2k}} \cdot \left\|f\right\|_{2k},$$
where
	$D_k=\dim\Span\{gw^{\otimes k}\colon g\in \SO(n)\}.$ Clearly,
\[\begin{split}
	D_k \leq \dim\Span\{g e_1^{\otimes 2k}\colon g\in \SO(n)\}\cdot \dim\Span\{g e_2^{\otimes 2k}\colon g\in \SO(n)\}
		= \binom{2 k+n-1}{2k}^2,
\end{split}
\]
where the equality follows as in \cite[p.\ 404]{Bar02}. 
If $n$ is odd, we let $2k_0=9(n-1)$. Otherwise take
$2k_0=9n$ to get
	$$D_{k_0}^{\frac{1}{2k_0}}\leq \binom{\frac{20}{9}k_0}{2k_0}^{\frac{1}{k_0}}.$$
	Since $2k_0=9\ell_0$ for some $\ell_0\in \NN$,  
		$$D_{k_0}^{\frac{1}{2k_0}}\leq \binom{10\ell_0}{9\ell_0}^{\frac{2}{9\ell_0}}\leq \left(\frac{10}{9}\cdot 10^{\frac{1}{9}}\right)^2,$$
	where we used \cite[Lemma 2.2]{KMSZ19} in the last inequality. 

To prove the Claim  it remains to estimate the average $L^{2k_0}$ norm, i.e.,
		\begin{equation}\label{def-A}
			A=\int_{S_{\cM}} \left\| f\right\|_{2k_0}\; \dd\widetilde\mu = \int_{S_{\cM}} \left(\int_{T}
				f^{2k_0}\; \dd\sigma\right)^{\frac{1}{2k_0}} \dd\widetilde\mu .
		\end{equation}
	Notice that
		\begin{equation}\label{identifikacija}
			\int_{S_{\cM}} \left(\int_{T}
				f^{2k_0}\; \dd\sigma \right)^{\frac{1}{2k_0}}\dd\widetilde\mu=
			\int_{c\in S_{V^\ast_2}} \left(\int_{g\in\SO(n)} \langle c,gv\rangle^{2k_0}\dd \widehat\sigma(g)\right)^{\frac{1}{2k_0}}\dd \breve\sigma(c),
		\end{equation}
	where $S_{V^\ast_2}$ is the unit sphere in $V^\ast_2$ endowed with the rotation invariant probability measure $\breve\sigma$.
	Combining \eqref{def-A},
 \eqref{identifikacija} we obtain
		$$A=
			\int_{c\in S_{V^\ast_2}} \left(\int_{g\in \SO(n)} \langle c,gv\rangle^{2k_0}\dd \widehat\sigma(g)\right)^{\frac{1}{2k_0}}
			\dd\breve\sigma(c)\leq
			\sqrt{\frac{2k_0 \langle v,v\rangle}{D_\cM}}=\sqrt{2k_0},$$
	where we used \cite[Lemma 3.5]{BB05} for the inequality and
	\cite[Remark p.\ 62]{BB05} for the last equality.
	This equality proves the Claim and establishes the lower bound in Theorem \ref{psd-intro}.
\end{proof}

\begin{proof}[Proof of Theorem \ref{squares-intro}]	
We write $\widetilde{\Sq}=\widetilde{\Sq}^{(n)}_{\cQ}$ for brevity.
	We define the support function
		$L_{\widetilde{\Sq}}$ of $\widetilde{\Sq}$ by
		$$L_{\widetilde{\Sq}}:\cM\to \RR,\quad
			L_{\widetilde{\Sq}}(f)=\max_{g\in \widetilde{\Sq}} \left\langle f,g\right\rangle.$$
	Let $S_\cU$ be the unit sphere in $\cU:=\RR[\x,\y]_{1, 1}/(I\cap \RR[\x,\y]_{1, 1})$ equipped with the 
	$L^2$ norm,
	and let $\|\;\|_{\sq}$ be the norm on $\cQ$ defined by
		$$\|f\|_{\sq}=\max_{g\in S_{\cU}} |\langle f,g^2\rangle|.$$
	As in \cite[\S 2.3.1]{KMSZ19}, it follows that
		$$\left(\frac{\Vol \widetilde{\Sq}}{\Vol B_\cM}\right)^{\frac{1}{D_{\cM}}}\leq
			\int_{S_{\cM}} \left\|f\right\|_{\sq} \dd\widetilde\mu.$$
	To prove the inequality of Theorem \ref{squares-intro} it now suffices to prove the following claim.\\

	\noindent \textbf{Claim:} $\ds\int_{S_{\cM}} \left\|f\right\|_{\sq} \dd\widetilde\mu 
		\leq 2^{3}\cdot 3 \cdot \sqrt{6}\cdot \frac{1}{n}.$\\

	For $f\in \cQ$ let $H_f$ be the quadratic form on $\cU$ defined by
		$$H_f(g)=\langle f, g^2\rangle\quad\text{for }g\in \cU.$$
	Note that
		$$\|f\|_{\sq}=\|H_{f}\|_{\infty}.$$
	Here $\|H_{f}\|_{\infty}$ stands for the supremum norm of 	$H_f$ on the unit sphere $S_\cU$.\\

	Let $\widehat\mu$ be the $\SO(n)$-invariant probability measure on $S_\cU$.
	The $L^{2p}$ norm of $H_f$ for a positive integer
	$p$ is defined by
		$$\|H_f\|_{2p}:=\left(\int_{S_{\cU}} H_f^{2p}(g)\dd\widehat\mu\right)^{\frac{1}{2p}}.$$
As in \cite[p.\ 3343--3344]{KMSZ19} (for $k_1=k_2=1$), it follows that 
$$
\int_{S_{\mathcal{M}}}\|H_{f}\|_{\infty}\dd\widetilde{\mu}
\leq 
2\sqrt{3}\int_{S_{\mathcal{M}}}\|H_{f}\|_{2D_\cU}\dd\widetilde{\mu}
\leq 
2\sqrt{3} \cdot \max_{g\in S_{\cU}}\|g^2\|_2 \cdot \sqrt{\frac{2D_{\cU}}{D_{\cM}}}
\leq 
2\sqrt{3}\cdot 6 \cdot\sqrt{\frac{2D_{\cU}}{D_{\cM}}},
$$
where the last inequality follows by 
$\|g^2\|_2=\|g\|_4^{2}$
and
Proposition \ref{bask-prod} below.
	To prove the Claim it remains to establish
		\begin{equation}\label{ostanek}
			\sqrt{\frac{2D_{\cU}}{D_{\cM}}} \leq 2^{\frac{3}{2}} n^{-1}.
		\end{equation}
	The dimensions $D_{\cU}$, $D_\cM$ are easily verified to be
		\begin{eqnarray*}
			D_{\cU} 
			&=& \dim \RR[\x,\y]_{1,1}-1=n^2-1,\\
			D_{\cM} 
			&=& \dim \RR[\x,\y]_{2,2}- \dim(\RR[\x,\y]_{2,2}\cap I)-1
			= \Big(\frac{n(n+1)}{2}\Big)^2-n^2-1.
		\end{eqnarray*}

	\noindent Observe that
		 \begin{eqnarray*}
		 \frac{2D_{\cU}}{D_{\cM} } 
		&=& \frac{2^3 (n^2-1)}{n^2(n+1)^2-4n^2-4} = 
			\frac{2^3 (n^2-1)}{(n^2-1)(n+1)^2-3(n^2-1)+2n-6} \\
		&\underbrace{\le}_{n>2}& \frac{2^3}{n^2+2n-2} \le \frac{2^3}{n^2},
		\end{eqnarray*}
	which proves \eqref{ostanek}.
\end{proof}


\subsection{Reverse H\"older inequality}
\label{revHolder}

We write $I_{1,1}=I\cap \RR[\x,\y]_{1, 1}$.
A bilinear form $g\in \RR[\x,\y]_{1, 1}/I_{1,1}$
is \textbf{symmetric} (resp.\ \textbf{skew-symmetric}) 
if it of the form
$g(\x,\y)=\x^TA\y+I_{1,1}$ 
for some symmetric (resp.\ skew-symmetric) matrix $A\in \RR^{n\times n}$.

\begin{proposition}
\label{bask-prod}
	For a bilinear biform 
        $g\in \RR[\x,\y]_{1, 1}/I_{1,1}$,
\begin{align}
\label{reverse-holder}
    \left(  
        \int_T g^4 \; \dd\sigma 
    \right)^{\frac{1}{4}}=
    \left\| g\right\|_4
    \leq 
    \sqrt{6}\left\|g\right\|_2 
    =
    \sqrt{6}
    \left( 
        \int_T g^2\; \dd\sigma
    \right)^{\frac{1}{2}}.
\end{align}
If $g$ is symmetric, then we can take
$\sqrt{3}$ instead of $\sqrt{6}$ in \eqref{reverse-holder} above,
while if $g$ is skew-symmetric,
$\sqrt{6}$ can be replaced by $\sqrt[4]{6}$. Moreover, the constants $\sqrt{3}$ (resp.\ $\sqrt[4]{6}$) are asymptotically sharp as $n\to\infty$.
\end{proposition} 

We point out an important fact about the inequality \eqref{reverse-holder},
which is crucial for Corollary \ref{cor:ratio}. Namely, 
the constant $C$ in $\|g\|_4\leq C\|g\|_2$ can be chosen to be  independent of the number of variables $n$.

The proof of Proposition \ref{bask-prod} will be done separately for the symmetric (Section \ref{symmetric-case}) and skew-symmetric case (Section \ref{skew-symmetric-case}), while the general case (Section \ref{general-case}) follows from the fact that every bilinear form $g$ can be written as a sum of a symmetric form $g_s$ and a skew-symmetric form $g_a$, together with the observation that $g_s$ is perpendicular to $g_a$ in the $L^2$ inner product.
For the proof we first need to compute the values of the integrals of monomials of bidegree $(2,2)$ and some monomials of bidegree $(4,4)$ with respect to $\sigma$, which is the content of Section \ref{computations}. 
Using these computations, \eqref{reverse-holder} becomes an inequality in the coefficients of $g_a$ (resp.\ $g_s$). We prove that this inequality holds.

\begin{remark}
    In \cite{Duo87}
    a version of the Reverse 
    H\"older inequality 
    with respect to
    the Lebesgue measure on
    the unit sphere and
    polynomials of any degree 
    is established. 
    Lemma 2.9 of
    \cite{KMSZ19}
    extends this result 
    to the product measure
    of two Lebesgue measures 
    on unit spheres.
    However, in the proof of Theorem
    \ref{squares-intro}
    we cannot use this extension because the measure $\sigma$ 
    is not a product measure.
    Therefore we have to establish the
    Reverse H\"older inequality
    we need in our setting.
\end{remark}

\subsubsection{Computations needed for the proof of Proposition \ref{bask-prod}}
\label{computations}
Let us introduce new variables
\begin{align*}
    \z_i
        &=\x_i\y_i,\quad i=1,\ldots,n,\\
    \z_{ij}
        &=\x_i\y_j,\quad i,j=1,\ldots,n,\\
    \vv_{ij}
        &=\z_{ij}+\z_{ji},\quad i,j=1,\ldots,n,\\
    \w_{ij}
        &=\z_{ij}-\z_{ji},\quad i,j=1,\ldots,n.
\end{align*}

\begin{lemma}
\label{integrals}
Let $n\geq 3$.
The following identities hold:
\begin{align*}
    I_1
    &=\int_T \z_i^2 \; \dd\sigma
    =  \frac{1}{n(n+2)}
    \quad
    \text{for }i=1,\ldots,n.\\
    I_2
    &=\int_T \z_i\z_j \; \dd\sigma
    =  
    -\frac{1}{n-1}I_1
    =    
    -\frac{1}{(n-1)n(n+2)}\quad
    \text{for }i,j=1,\ldots,n,\; i\neq j,\\[0.3em]
    I_3
    &=\int_T \z_{ij}^2 \; \dd\sigma
    =  
    \frac{n+1}{n-1}I_1
    =
    \frac{n+1}{(n-1)n(n+2)}
    \quad
    \text{for }i,j=1,\ldots,n,\; i\neq j,\\[0.3em]
    I_4
    &=\int_T \vv_{ij}^2 \; \dd\sigma
    =  
    2\frac{n}{n-1}I_1
    =
    \frac{2}{(n-1)(n+2)}
    \quad
    \text{for }i,j=1,\ldots,n,\; i\neq j,\\[0.3em]
    I_5
    &=\int_T \w_{ij}^2 \; \dd\sigma
    =  
    2\frac{n+2}{n-1}I_1
    =
    \frac{2}{(n-1)n}    
    \quad
    \text{for }i,j=1,\ldots,n,\; i\neq j,\\[0.3em]
    I_6&=
    \int_T \z_{ij}\z_{kl} \; \dd\sigma
    =
    0\quad
        \text{if at least one of  } i,j,k,l
        \text{ occurs an odd number of times}
        ,\\[0.3em]
    I_7
    &=\int_T \vv_{ij}\vv_{kl} \; \dd\sigma
    =  
    0    
    \quad
    \text{if at least one of } i,j,k,l
    \text{ occurs an odd number of times,}\\[0.3em]
    I_8
    &=\int_T \w_{ij}\w_{kl} \; \dd\sigma
    =  
    0    
    \quad
    \text{if at least one of } i,j,k,l
    \text{ occurs an odd number of times,}\\[0.3em]  
    J_1
    &=\int_T \z_i^4 \; \dd\sigma
    = 
        \displaystyle\frac{9}{n(n+2)(n+4)(n+6)}
    \quad
    \text{for }i=1,\ldots,n,
    \\[0.3em]
    J_2
    &=\int_T \z_i^3\z_j \; \dd\sigma
    =
    -\frac{1}{n-1}J_1
    =
        \displaystyle-\frac{9}{(n-1)n(n+2)(n+4)(n+6)}\quad
    \text{for }i,j=1,\ldots,n,\; i\neq j,\\[0.3em]
    J_3
    &=\int_T \z_i^2\z_j^2 \; \dd\sigma
    =
        \displaystyle\frac{n^2+4n+15}{(n-1)n(n+1)(n+2)(n+4)(n+6)},\quad
    \text{for }i,j=1,\ldots,n,\; i\neq j,
    \\[0.3em]
    J_4
    &=\int_T \z_i^2\z_j\z_k \; \dd\sigma
    =
        \displaystyle-\frac{n-3}{(n-1)n(n+1)(n+2)(n+4)(n+6)}\\[0.3em]
    &\pushright{
    \text{for }i,j,k,l=1,\ldots,n,\; 
    i,j,k
    \text{ pairwise different}
    ,}
    \\[0.3em]
    J_5
    &=\int_T \z_i\z_j\z_k\z_l \; \dd\sigma
    =\frac{3}{(n-1)n(n+1)(n+2)(n+4)(n+6)}\\[0.3em]
    &\pushright{   
    \text{for }
        n\geq 4
    \text{ and }
        i,j,k=1,\ldots,n,\; 
        i,j,k,l
        \text{ pairwise different},}
        \\[0.3em]
    J_6
    &=\int_T \w_{ij}^4 \; \dd\sigma
    =\frac{24}{(n-1)n(n+1)(n+2)}
    \quad
    \text{for }
        i,j=1\ldots,n,\; 
        i\neq j,\\[0.3em]
    J_{7}
    &
    =\int_T \w_{ij}^2\w_{kl}^2 \; \dd\sigma
    =\frac{1}{6}J_6
    =\frac{4}{(n-1)n(n+1)(n+2)}\\
    &\pushright{
    \text{for }
        n\geq 4
    \text{ and }
        i,j,k,l \text{ pairwise different},}
        \\[0.3em]
    J_{8}
    &=\int_T \z_{ij}\z_{kl}\z_{op}\z_{rs} \; \dd\sigma
    =  
    0,
    \\[0.3em]
    &\pushright{
    \text{if at least one of } i,j,k,l,o,p,r,s
    \text{ occurs an odd number of times,}}\\[0.3em] 
    J_{9}
    &=\int_T \w_{ij}\w_{kl}\w_{op}\w_{rs} \; \dd\sigma
    =  
    0,   
    \\
    &\pushright{
    \text{if at least one of } i,j,k,l,o,p,r,s
    \text{ occurs an odd number of times.}}
\end{align*}
\end{lemma}

In the proof of Lemma \ref{integrals}
we will use the following technical lemma.
Recall that $n!!=\displaystyle\prod_{k=0}^{\lceil\frac{n}{2}\rceil -1} (n-2k)$
stands for the double factorial of $n\in \NN$.

\begin{lemma}
\label{integrals-help}
For $i\in \NN$ and $j\in \NN\cup \{0\}$ the following equalities hold:
\begin{align*}
    A_{i,2j}
&   :=\frac
    {\int_0^\pi\sin^{i+2j}(\phi)\;\dd\phi}
    {\int_0^\pi \sin^{i}(\phi)\;\dd\phi}
    =\frac{(i+2j-1)!!\; i!!}{(i-1)!!\; (i+2j)!!},\\[0.3em]
B_{i,2j}
&   :=\frac
    {\int_0^\pi \cos^2(\phi)\sin^{i+2j}(\phi)\;\dd\phi}
    {\int_0^\pi \sin^{i}(\phi)\;\dd\phi}
    =\frac{(i+2j-1)!!\; i!!}{(i-1)!!\;(i+2j+2))!!},\\[0.3em]
C_{i,2j}
&   :=\frac
    {\int_0^\pi \cos^4(\phi)\sin^{i+2j}(\phi)\;\dd\phi}
    {\int_0^\pi \sin^{i}(\phi)\;\dd\phi}
=3\frac{(i+2j-1)!!\; i!!}{(i-1)!!\; (i+2j+4))!!}.
\end{align*}
\end{lemma}

\begin{proof}
We have
\begin{align*}
    A_{i,2j}
    &=\frac
    {\int_0^\pi\sin^{i+2j}(\phi)\;\dd\phi}
    {\int_0^\pi \sin^{i}(\phi)\;\dd\phi}
    =
    \frac
    {B\big(\frac{i+2j+1}{2},\frac{1}{2}\big)}
    {B\big(\frac{i+1}{2},\frac{1}{2}\big)}
    =
    \frac
    {\Gamma\big(\frac{i+2j+1}{2}\big)\Gamma\big(\frac{i+2}{2}\big)}
    {\Gamma\big(\frac{i+1}{2}\big)\Gamma\big(\frac{i+2j+2}{2}\big)}
    \\
    &=
    \frac
    {
    \big(i+1+2(j-1)\big)
    \big(i+1+2(j-2)\big)
    \cdots
    \big(i+1\big)
    }
    {
    \big(i+2+2(j-1)\big)
    \big(i+2+2(j-2)\big)
    \cdots
    \big(i+2\big)}\\
    &=
    \frac{(i+2j-1)!!\; i!!}{(i-1)!!\; (i+2j)!!}.
\end{align*}
The proofs for $B_{i,2j}$ and $C_{i,2j}$
are similar.
\end{proof}

Now we are ready to prove Lemma \ref{integrals}.

\begin{proof}[Proof of Lemma \ref{integrals}]
We write 
\begin{align*}
    \PHI
    &=(\phi_1,\phi_2,\ldots,\phi_{n-1}),
    \quad
    \phi_1,\ldots,\phi_{n-2}\in [0,\pi],\; \phi_{n-1}\in [0,2\pi]\\
    \PSI
    &=(\psi_1,\psi_2,\ldots,\psi_{n-2}),
    \quad
    \psi_1,\ldots,\psi_{n-3}\in [0,\pi],\;
    \psi_{n-2}\in [0,2\pi].
\end{align*}
Let
$$
R_n^j(\phi)
=
\begin{pmatrix}
    I_{j-1} & 0 & 0 & 0\\
    0 & \cos(\phi) & -\sin(\phi) & 0\\
    0 & \sin(\phi) & \cos(\phi) & 0 \\
    0 & 0 & 0 & I_{n-j-1}
\end{pmatrix},\quad 1\leq j\leq n-1,
$$
be a Givens rotation, where $I_k$ stands for the $k\times k$ identity matrix, and
\begin{align*}
H_{n}^{1}(\PHI)
    &=
    R_{n}^{n-1}(\phi_{n-1})
    R_{n}^{n-2}(\phi_{n-2})
    \cdots
    R_{n}^{1}(\phi_{1}),\\
H_{n}^{2}(\PSI)
    &=
    R_{n}^{n-1}(\psi_{n-2})
    R_{n}^{n-2}(\psi_{n-3})
    \cdots
    R_{n}^{2}(\psi_{1}).
\end{align*}
Then $H_{n}^{1}(\PHI)$ is 
    (see, e.g., the formula for $L_1(\theta)$ in \cite[p.\ 3--4]{Tum65}
    or use induction on $n$)
\begin{tiny}
\begin{equation*}
    \begin{pmatrix}
        \cos(\phi_1) 
            & -\sin(\phi_1) 
                & 0 
                    & \cdots
                        & 0
                            \\
        \sin(\phi_1)\cos(\phi_2) 
            & \cos(\phi_1)\cos(\phi_2)
                & -\sin(\phi_2)
                   & \cdots
                        & 0
                            \\
        \vdots 
            & \vdots 
                & \vdots
                    & \cdots    
                        & \vdots
                            \\
        \Big(\displaystyle\prod_{j=1}^{i-1}\sin(\phi_j)\Big)
        \cos(\phi_i)
            &
                \cos(\phi_1)
                \Big(\displaystyle\prod_{j=2}^{i-1}\sin(\phi_j)\Big)
                \cos(\phi_i)
                &
                    \cos(\phi_2)
                    \Big(\displaystyle\prod_{j=3}^{i-1}\sin(\phi_j)\Big)
                    \cos(\phi_i)
                    &   \cdots  
                        & 0
                            \\
        \vdots 
            & \vdots
                & \vdots
                    & \cdots
                        & \vdots
                            \\
        \Big(\displaystyle\prod_{j=1}^{n-2}\sin(\phi_j)\Big)
        \cos(\phi_{n-1})
            &
                \cos(\phi_1)
                \Big(\displaystyle\prod_{j=2}^{n-2}\sin(\phi_j)\Big)
                \cos(\phi_{n-1})
                &
                    \cos(\phi_2)
                    \Big(\displaystyle\prod_{j=3}^{n-2}\sin(\phi_j)\Big)
                    \cos(\phi_{n-1})
                    & \cdots
                        & -\sin(\phi_{n-1})
                            \\
        \Big(\displaystyle\prod_{j=1}^{n-2}\sin(\phi_j)\Big)
        \sin(\phi_{n-1})
            &
                \cos(\phi_1)
                \Big(\displaystyle\prod_{j=2}^{n-2}\sin(\phi_j)\Big)
                \sin(\phi_{n-1})
                &
                    \cos(\phi_2)
                    \Big(\displaystyle\prod_{j=3}^{n-2}\sin(\phi_j)\Big)
                    \sin(\phi_{n-1})
                    & \cdots    
                        & \cos(\phi_{n-1})
    \end{pmatrix}
\end{equation*}
\end{tiny}
and 
$$
    H_n^2(\PSI)
    =
    \begin{pmatrix}
        1 & 0\\
        0 & H_{n-1}^1(\PSI)
    \end{pmatrix}.
$$
The set $T=V_{2,n}(\RR)$ (see Remark \ref{Stiefel-manifold}) 
can be parametrized by 
(see, e.g., 
    \cite[p.\ 48--49]{Chi03} 
or 
    \cite[\S2]{Tum65}) 
\begin{align*}
(\PHI,\PSI)\mapsto 
(x(\PHI),y(\PHI,\PSI))
&=
\begin{pmatrix}
    x_1(\PHI) & y_1(\PHI,\PSI)\\
    x_2(\PHI) & y_2(\PHI,\PSI)\\
    \vdots & \vdots \\
    x_n(\PHI) & y_n(\PHI,\PSI)
\end{pmatrix}\\
&=\text{the first two columns of }
H_n^1(\PHI)H_{n}^2(\PSI),
\end{align*}
where
$(\PHI,\PSI)\in 
\big([0,\pi]^{n-2}\times [0,2\pi]\big)
\times
\big([0,\pi]^{n-3}\times [0,2\pi]\big).
$
We define 
\begin{align*}
    \underline{\int_n}
    &=
    \underbrace{\int_{0}^{\pi}\cdots\int_{0}^{\pi}}_{n-2}\int_{0}^{2\pi}
    \underbrace{\int_{0}^{\pi}\cdots\int_{0}^{\pi}}_{n-3}\int_{0}^{2\pi},\\
    \dd\PHI
    &=\dd\phi_1\dd\phi_2\ldots\dd\phi_{n-1},\quad
    \dd\PSI
    =\dd\psi_1\dd\psi_2\ldots\dd\psi_{n-2}.
\end{align*}
We have
    $$
        \int_T g(x,y)\; \dd\sigma= 
        \underline{\int_n}
        g(x(\PHI),y(\PHI,\PSI)) V_n(\PHI,\PSI)\;      
        \dd\PHI\dd\PSI,
    $$
where by \cite[Theorem 2.1]{Chi90} (taking $V=x(\PHI)$, $G(V)=H_n^1(\PHI) \text{ without the first column}$ and
$Z= \text{the first column of }H_{n-1}^1(\PSI)$) and \cite{Blu60},
    $$
    V_n(\PHI,\PSI)=
    \frac{1}{S_n}
    \prod_{i=1}^{n-2}
    \sin(\phi_i)^{n-1-i}
    \cdot 
    \frac{1}{S_{n-1}}
    \prod_{i=1}^{n-3}
    \sin(\psi_i)^{n-2-i},
    $$
with
\begin{align*}
        S_n
        &=
\underbrace{\int_{0}^{\pi}\cdots\int_{0}^{\pi}}_{n-2}
\int_{0}^{2\pi}
    \prod_{i=1}^{n-2}
    \sin(\phi_i)^{n-1-i}
        \;\dd\PHI,\\
            S_{n-1}
            &=
\underbrace{\int_{0}^{\pi}\cdots\int_{0}^{\pi}}_{n-3}
\int_{0}^{2\pi}
    \prod_{i=1}^{n-3}
    \sin(\psi_i)^{n-2-i}
        \;\dd\PSI.    
    \end{align*}
By the invariance of the integral with respect to the 
change of indices 
we can assume without loss of generality 
that $i,j,k,l\in \{1,2,3,4\}$ in all equalities of Lemma
\ref{integrals}.
Due to the difference in parameterizations of 
some of the coordinates 
    $x_i(\PHI)$, $y_i(\PHI,\PSI)$, $i=1,2,3,4$,
we separate cases
$n\geq 6$, $n=5$, $n=4$ and $n=3$
in the rest of the proof.
\\

\noindent \textbf{Case 1:} $n\geq 6$.
The coordinates 
    $x_i(\PHI)$, $y_i(\PHI,\PSI)$, $i=1,2,3,4$,
are the following:
\begin{align*}
    x_1(\PHI)
    &=\cos(\phi_1),\\
    x_2(\PHI)
    &=\sin(\phi_1)\cos(\phi_2),\\
    x_3(\PHI)
    &=\sin(\phi_1)\sin(\phi_2)\cos(\phi_3),\\
    x_4(\PHI)
    &=\sin(\phi_1)\sin(\phi_2)\sin(\phi_3)\cos(\phi_4),\\
y_1(\PHI,\PSI)
    &=
    -\sin(\phi_1)\cos(\psi_1),\\
y_2(\PHI,\PSI)
    &=
    \cos(\phi_1)\cos(\phi_2)\cos(\psi_1)
        -\sin(\phi_2)\sin(\psi_1)\cos(\psi_2),\\
y_3(\PHI,\PSI)
    &=
    \cos(\phi_1)\sin(\phi_2)\cos(\phi_3)\cos(\psi_1)
    +
    \cos(\phi_2)\cos(\phi_3)\sin(\psi_1)\cos(\psi_2)\\
    &\hspace{1cm}
    -
    \sin(\phi_3)\sin(\psi_1)\sin(\psi_2)\cos(\psi_3),\\
y_4(\PHI,\PSI)
    &=
    \cos(\phi_1)\sin(\phi_2)\sin(\phi_3)\cos(\phi_4)\cos(\psi_1)
    +
    \cos(\phi_2)\sin(\phi_3)\cos(\phi_4)\sin(\psi_1)\cos(\psi_2)\\
    &
    +
    \cos(\phi_3)\cos(\phi_4)\sin(\psi_1)\sin(\psi_2)\cos(\psi_3)
    -
    \sin(\phi_4)\sin(\psi_1)\sin(\psi_2)\sin(\psi_3)\cos(\psi_4).
\end{align*}
In the computations below we will need the following identities in the notation of Lemma \ref{integrals-help}:
\begin{align*}
A_{n-4,2}
&
=\frac{(n-3)!!\; (n-4)!!}{(n-5)!!\; (n-2)!!}
=\frac{n-3}{n-2},\\
A_{n-4,4}
&
=\frac{(n-1)!!\; (n-4)!!}{(n-5)!!\; n!!}
=\frac{(n-3)(n-1)}{(n-2)n},\\
A_{n-2,6}
&
=\frac{(n+3)!!\; (n-2)!!}{(n-3)!!\; (n+4)!!}
=\frac{(n-1)(n+1)(n+3)}{n(n+2)(n+4)},\\
A_{n-2,8}
&
=\frac{(n+5)!!\; (n-2)!!}{(n-3)!!\; (n+6)!!}
=\frac{(n-1)(n+1)(n+3)(n+5)}{n(n+2)(n+4)(n+6)},\\
B_{n-5,0}
&
=\frac{(n-6)!!\; (n-5)!!}{(n-6)!!\; (n-3)!!}
=\frac{1}{n-3},\\
B_{n-4,0}
&
=\frac{(n-5)!!\; (n-4)!!}{(n-5)!!\; (n-2)!!}
=\frac{1}{n-2},\\
B_{n-4,2}
&
=\frac{(n-3)!!\; (n-4)!!}{(n-5)!!\; n!!}
=\frac{n-3}{(n-2)n},\\
B_{n-3,0}
&
=\frac{(n-4)!!\; (n-3)!!}{(n-4)!!\; (n-1)!!}
=\frac{1}{n-1},\\
B_{n-3,2}
&
=\frac{(n-2)!!\; (n-3)!!}{(n-4)!!\; (n+1)!!}
=\frac{n-2}{(n-1)(n+1)},\\
B_{n-3,4}
&
=\frac{n!!\; (n-3)!!}{(n-4)!!\; (n+3)!!}
=\frac{(n-2)n}{(n-1)(n+1)(n+3)},\\
B_{n-2,2}
&
=\frac{(n-1)!!\; (n-2)!!}{(n-3)!!\; (n+2)!!}
=\frac{n-1}{n(n+2)},\\
B_{n-2,4}
&
=\frac{(n+1)!!\; (n-2)!!}{(n-3)!!\; (n+4)!!}
=\frac{(n-1)(n+1)}{n(n+2)(n+4)},\\
B_{n-2,6}
&
=\frac{(n+3)!!\; (n-2)!!}{(n-3)!!\; (n+6)!!}
=\frac{(n-1)(n+1)(n+3)}{n(n+2)(n+4)(n+6)},\\
C_{n-3,0}
&
=3\frac{(n-4)!!\; (n-3)!!}{(n-4)!!\; (n+1)!!}
=\frac{3}{(n-1)(n+1)},\\
C_{n-3,2}
&
=3\frac{(n-2)!!\; (n-3)!!}{(n-4)!!\; (n+3)!!}
=\frac{3(n-2)}{(n-1)(n+1)(n+3)},\\
C_{n-2,4}
&
=3\frac{(n+1)!!\; (n-2)!!}{(n-3)!!\; (n+6)!!}
=\frac{3(n-1)(n+1)}{n(n+2)(n+4)(n+6)}.
\end{align*}
Now we are ready to prove the identities of 
Lemma \ref{integrals}. In the computations below we include only terms with nonzero integrals, i.e., terms where in none of the factors
$\cos(\phi_i)^k$ or $\cos(\psi_i)^k$ 
the exponent $k$ is odd.

\begin{align*}
    I_1
    &=\int_T \z_1^2\; \dd\sigma
    =
    \underline{\int_n} 
    \cos(\phi_1)^2\sin(\phi_1)^{2}\cos(\psi_1)^2 
    V_n(\PHI,\PSI)\;      \dd\PHI\dd\PSI\\
    &
    =
    B_{n-2,2}B_{n-3,0}
    =
    \frac{n-1}{n(n+2)}\frac{1}{n-1}
    =
    \frac{1}{n(n+2)},\\
    I_2
    &=
    \int_T \z_1\z_2 \; \dd\sigma
    =-
    \underline{\int_n} \cos(\phi_1)^2\sin(\phi_1)^2\cos(\phi_2)^2\cos(\psi_1)^2 V_n(\PHI,\PSI)\;      \dd\PHI\dd\PSI    
    \\
    &
    =
    -B_{n-2,2}\big(B_{n-3,0}\big)^2
    =
    -
    \frac{n-1}{n(n+2)}\frac{1}{(n-1)^2}
    =
    -\frac{1}{(n-1)n(n+2)},\\
    J_1
    &=\int_T \z_1^4\; \dd\sigma
    =
    \underline{\int_n} 
    \cos(\phi_1)^4\sin(\phi_1)^4\cos(\psi_1)^4 
    V_n(\PHI,\PSI)\;      \dd\PHI\dd\PSI   
    \\
    &
    =
    C_{n-2,4}C_{n-3,0}
    =
    \frac{3(n-1)(n+1)}{n(n+2)(n+4)(n+6)}
    \frac{3}{(n-1)(n+1)}
    =
        \frac{9}{n(n+2)(n+4)(n+6)},\\
J_2&= \int_T \z_1^3\z_2 \; \dd\sigma
    =-
    \underline{\int_n} 
    \cos(\phi_1)^4\sin(\phi_1)^4
    \cos(\phi_2)^2
    \cos(\psi_1)^4
    V_n(\PHI,\PSI)      \; \dd\PHI\dd\PSI   \\
    &=
    -C_{n-2,4}B_{n-3,0}C_{n-3,0}
    =
    -
    \frac{3(n-1)(n+1)}{n(n+2)(n+4)(n+6)}
    \frac{1}{n-1}
    \frac{3}{(n-1)(n+1)}\\
    &=
    -\frac{9}{(n-1)n(n+2)(n+4)(n+6)},\\
J_3&=\int_T \z_1^2\z_2^2 \; \dd\sigma
    =
    \underline{\int_n} 
    \cos(\phi_1)^4\sin(\phi_1)^4
    \cos(\phi_2)^4
    \cos(\psi_1)^4
    V_n(\PHI,\PSI)     \; \dd\PHI\dd\PSI\\
&+ \underline{\int_n}
    \cos(\phi_1)^2\sin(\phi_1)^4
    \cos(\phi_2)^2\sin(\phi_2)^2
    \cos(\psi_1)^2\sin(\psi_1)^2
    \cos(\psi_2)^2
    V_n(\PHI,\PSI)     \; \dd\PHI\dd\PSI\\
    &
=
C_{n-2,4}\big(C_{n-3,0}\big)^2
+
B_{n-2,4}\big(B_{n-3,2}\big)^2B_{n-4,0}\\
&=
\frac{3(n-1)(n+1)}{n(n+2)(n+4)(n+6)}\frac{9}{(n-1)^2(n+1)^2}
+
\frac{(n-1)(n+1)}{n(n+2)(n+4)}
\frac{(n-2)^2}{(n-1)^2(n+1)^2}
\frac{1}{n-2}\\
        &
        =
        \frac{n^2+4n+15}{(n-1)n(n+1)(n+2)(n+4)(n+6)},\\
J_4&=\int_T \z_1^2\z_2\z_3\; \dd\sigma
=   
    \underline{\int_n}
    \cos(\phi_1)^4\sin(\phi_1)^4
    \cos(\phi_2)^2\sin(\phi_2)^2
    \cos(\phi_3)^2
    \cos(\psi_1)^4\cdot
    \\
    &\hspace{1cm}
    \cdot
    V_n(\PHI,\PSI)    \;  \dd\PHI\dd\PSI
    -
    \underline{\int_n}
    \cos(\phi_1)^2\sin(\phi_1)^4
    \cos(\phi_2)^2\sin(\phi_2)^2
    \cos(\phi_3)^2
    \cos(\psi_1)^2\sin(\psi_1)^2\cdot\\
    &\hspace{1cm}
    \cdot
    \cos(\psi_2)^2
    V_n(\PHI,\PSI)    \;  \dd\PHI\dd\PSI
    \\
    &=
    C_{n-2,4}B_{n-3,2}B_{n-4,0}C_{n-3,0}
    -
    B_{n-2,4}\big(B_{n-3,2}\big)^2\big(B_{n-4,0}\big)^2\\
    &=
    \frac{3(n-1)(n+1)}{n(n+2)(n+4)(n+6)}
    \frac{n-2}{(n-1)(n+1)}
    \frac{1}{n-2}
    \frac{3}{(n-1)(n+1)}\\
    &\hspace{2cm}
    -\frac{(n-1)(n+1)}{n(n+2)(n+4)}
    \frac{(n-2)^2}{(n-1)^2(n+1)^2}
    \frac{1}{(n-2)^2}    
    \\
        &=
        -\frac{n-3}{(n-1)n(n+1)(n+2)(n+4)(n+6)},\\
J_5&=\int_T \z_1\z_2\z_3\z_4\; \dd\sigma=
    -
    \underline{\int_n}
    \cos(\phi_1)^4\sin(\phi_1)^4
    \cos(\phi_2)^2\sin(\phi_2)^4
    \cos(\phi_3)^2\sin(\phi_3)^2\cdot\\
    &\hspace{1cm}\cdot
    \cos(\phi_4)^2
    \cos(\psi_1)^4
    V_n(\PHI,\PSI)    \;  \dd\PHI\dd\PSI
    -   
    \underline{\int_n}
    \cos(\phi_1)^2\sin(\phi_1)^4
    \cos(\phi_2)^4\sin(\phi_2)^2\cdot\\
    &\hspace{1cm}\cdot
    \cos(\phi_3)^2\sin(\phi_3)^2
    \cos(\phi_4)^2
    \cos(\psi_1)^2\sin(\psi_1)^2
    \cos(\psi_2)^2
    V_n(\PHI,\PSI)   \;   \dd\PHI\dd\PSI\\
&+
2\underline{\int_n}
    \cos(\phi_1)^2\sin(\phi_1)^4
    \cos(\phi_2)^2\sin(\phi_2)^4
    \cos(\phi_3)^2\sin(\phi_3)^2
    \cos(\phi_4)^2
    \cos(\psi_1)^2\sin(\psi_1)^2\cdot\\
    &\hspace{1cm}\cdot\cos(\psi_2)^2
    V_n(\PHI,\PSI)  \;    \dd\PHI\dd\PSI
+\underline{\int_n}
    \cos(\phi_1)^2\sin(\phi_1)^4
    \cos(\phi_2)^2\sin(\phi_2)^2\cdot\\
    &\hspace{1cm}\cdot
    \cos(\phi_3)^2\sin(\phi_3)^2
    \cos(\phi_4)^2
    \cos(\psi_1)^2\sin(\psi_1)^2    
    \sin(\psi_2)^2
    \cos(\psi_3)^2
    V_n(\PHI,\PSI)  \;    \dd\PHI\dd\PSI\\
    &=
    -C_{n-2,4}B_{n-3,4}B_{n-4,2}B_{n-5,0}C_{n-3,0}
    -B_{n-2,4}C_{n-3,2}B_{n-4,2}B_{n-5,0}B_{n-3,2}B_{n-4,0}\\
    &
    +2B_{n-2,4}B_{n-3,4}B_{n-4,2}B_{n-5,0}B_{n-3,2}B_{n-4,0}
    +B_{n-2,4}\big(B_{n-3,2}\big)^2B_{n-4,2}\big(B_{n-5,0}\big)^2A_{n-4,2}
    \\
    &
    =
    -
    \frac{3(n-1)(n+1)}{n(n+2)(n+4)(n+6)}
    \frac{(n-2)n}{(n-1)(n+1)(n+3)}
    \frac{n-3}{(n-2)n}
    \frac{1}{n-3}
    \frac{3}{(n-1)(n+1)}
    \\
    &\hspace{1cm}-
    \frac{(n-1)(n+1)}{n(n+2)(n+4)}
    \frac{3(n-2)}{(n-1)(n+1)(n+3)}
    \frac{n-3}{(n-2)n}
    \frac{1}{n-3}
    \frac{n-2}{(n-1)(n+1)}
    \frac{1}{n-2}
    \\
    &\hspace{1cm}+
    2
    \frac{(n-1)(n+1)}{n(n+2)(n+4)}
    \frac{(n-2)n}{(n-1)(n+1)(n+3)}
    \frac{n-3}{(n-2)n}
    \frac{1}{n-3}
    \frac{n-2}{(n-1)(n+1)}
    \frac{1}{n-2}
    \\
    &\hspace{1cm}+
    \frac{(n-1)(n+1)}{n(n+2)(n+4)}
    \frac{(n-2)^2}{(n-1)^2(n+1)^2}
    \frac{n-3}{(n-2)n}
    \frac{1}{(n-3)^2}
    \frac{n-3}{n-2}
    \\
        &=
        \frac{3}{(n-1)n(n+1)(n+2)(n+4)(n+6)}.
\end{align*}

The fact 
    $I_3=\frac{n+1}{n-1}I_1$ is a consequence of the
following computation:
\begin{align*}
    1
    &=
    \int_T 
        \Big(\sum_{i=1}^n \x_i^2\Big)
        \Big(\sum_{i=1}^n \y_i^2\Big)
       \; \dd\sigma
    =
    \int_T \Big(\sum_{i=1}^n \z_i^2\Big) 
        \;\dd\sigma
    +
    \int_T \Big(\sum_{i\neq j} \z_{ij}^2\Big)             \; \dd\sigma
    =
    nI_1+n(n-1)I_3.
\end{align*}
Hence, 
    $$1-nI_1=\frac{n+1}{n+2}=n(n-1)I_3,$$
which implies 
    $I_3=\frac{(n+1)}{(n-1)n(n+2)}=\frac{n+1}{n-1}I_1$.

Further,
\begin{align*}
    I_4
    &=
    \int_T (\z_{ij}+\z_{ji})^2 \; \dd\sigma
    =  
    \int_T (\z_{ij}^2+2\z_{ij}\z_{ji}+\z_{ji}^2) \; \dd\sigma
    =  
    \int_T (\z_{ij}^2+2\z_{i}\z_{j}+\z_{ji}^2) \; \dd\sigma\\
    &=
    2(I_3+I_2)
    =2\Big(\frac{n+1}{n-1}-\frac{1}{n-1}\Big)I_1
    =\frac{2n}{n-1}I_1
    =\frac{2}{(n-1)(n+2)},\\
    I_5
    &=\int_T (\z_{ij}-\z_{ji})^2 \; \dd\sigma
    =  
    \int_T (\z_{ij}^2-2\z_{ij}\z_{ji}+\z_{ji}^2) \; \dd\sigma
    =  
    \int_T (\z_{ij}^2-2\z_{i}\z_{j}+\z_{ji}^2) \; \dd\sigma\\
    &=
    2(I_3-I_2)
    =2\Big(\frac{n+1}{n-1}+\frac{1}{n-1}\Big)I_1
    =\frac{2(n+2)}{n-1}I_1
    =\frac{2}{(n-1)n}.
\end{align*}

Next we prove that 
$\int_T \z_{ij}\z_{kl} \; \dd\sigma=0$ 
if at least one of $i,j,k,l$
occurs an odd number of times.
Write $g(\x,\y):=\x_i\y_j\x_k\y_l$ and let 
$i_1$ be the index among $i,j,k,l$ which occurs an odd number of times.
Since
\begin{align*}
\begin{split}
&\quad g(
x_1,\ldots,x_{i_1-1},-x_{i_1},x_{i_1+1},\ldots,x_n,
y_1,\ldots,y_{i_1-1},-y_{i_1},y_{i_1+1},\ldots,y_n)\\
&=
-
g(
x_1,\ldots,x_{i_1-1},x_{i_1},x_{i_1+1},\ldots,x_n,
y_1,\ldots,y_{i_1-1},y_{i_1},y_{i_1+1},\ldots,y_n)
\end{split}
\end{align*}
for every $(x,y)\in T$, it follows that 
$I_6=\int_T \z_{ij}\z_{kl} \; \dd\sigma=0$.
Consequently, $I_7=I_8=0$, since $I_7$ and $I_8$ are both weighted sums of the integrals of the form $I_6$.

Now we prove $J_6=\frac{24}{(n-1)n(n+1)(n+2)}$.
We have
\begin{align*}
\int_T \w_{12}^4\;\dd\sigma
&=\int_T (\x_1\y_2-\y_1\x_2)^4\;\dd\sigma\\
&=\int_T (\x_1^4\y_2^4-4\x_1^3\x_2\y_1\y_2^3+6\x_1^2\x_2^2\y_1^2\y_2^2-4\x_1\x_2^3\y_1^3\y_2+\x_2^4\y_1^4)\;\dd\sigma\\
&=
2\int_T \x_2^4\y_1^4\;\dd\sigma
-
8\int_T \x_1\x_2^3\y_1^3\y_2\;\dd\sigma
+
6\int_T \z_1^2\z_2^2\;\dd\sigma,
\end{align*}
where we used
that 
$
\int_T \x_1^4\y_2^4\;\dd\sigma
=\int_T \x_2^4\y_1^4\;\dd\sigma
$
and
$
\int_T \x_1^3\x_2\y_1\y_2^3\;\dd\sigma
=
\int_T \x_1\x_2^3\y_1^3\y_2\;\dd\sigma
$
due to 
symmetry of the integral value in the indices of the variables $\x, \y$.
We compute:
\begin{align*}
J_{6}^{(1)}
&=\int_T \x_2^4\y_1^4\;\dd\sigma
=
\underline{\int_n}
\sin(\phi_1)^8\cos(\phi_2)^4\cos(\psi_1)^4 
V_n(\PHI,\PSI)
\;\dd\PHI\dd\PSI\\
&=
A_{n-2,8}\big(C_{n-3,0}\big)^2
=
\frac{(n-1)(n+1)(n+3)(n+5)}{n(n+2)(n+4)(n+6)}
\frac{9}{(n-1)^2(n+1)^2}
\\
    &=
        \frac{9(n+3)(n+5)}{(n-1)n(n+1)(n+2)(n+4)(n+6)},\\
J_6^{(2)}
&=\int_T \x_1\x_2^3\y_1^3\y_2\;\dd\sigma
=
-
\underline{\int_n}
\cos(\phi_1)^2\sin(\phi_1)^6
\cos(\phi_2)^4
\cos(\psi_1)^4
V_n(\PHI,\PSI)
\;\dd\PHI\dd\PSI\\
&=
-
B_{n-2,6}\big(C_{n-3,0}\big)^2
=
-
\frac{(n-1)(n+1)(n+3)}{n(n+2)(n+4)(n+6)}
\frac{9}{(n-1)^2(n+1)^2}
\\
    &=
        -\frac{9(n+3)}{(n-1)n(n+1)(n+2)(n+4)(n+6)},
\end{align*}
where in the second integral we included only the term with a nonzero integral.
Hence,
$$
J_6
=2J_6^{(1)}-8J_{6}^{(2)}+6J_3
=\frac{24}{(n-1)n(n+1)(n+2)}.
$$

Next we prove $J_{7}=\frac{4}{(n-1)n(n+1)(n+2)}$.
We have
\begin{align*}
&\int_T \w_{12}^2\w_{34}^2\;\dd\sigma
=\int_T 
(\x_1\y_2-\y_1\x_2)^2
(\x_3\y_4-\y_3\x_4)^2\;\dd\sigma\\
&=\int_T 
\big(
\x_1^2\x_3^2\y_2^2\y_4^2
-2\x_1\x_2\x_3^2\y_1\y_2\y_4^2
+\x_2^2\x_3^2\y_1^2\y_4^2
-2\x_1^2\x_3\x_4\y_2^2\y_3\y_4
+4\x_1\x_2\x_3\x_4\y_1\y_2\y_3\y_4\\
&\hspace{1cm}
-2\x_2^2\x_3\x_4\y_1^2\y_3\y_4
+\x_1^2\x_4^2\y_2^2\y_3^2
-2\x_1\x_2\x_4^2\y_1\y_2\y_3^2
+\x_2^2\x_4^2\y_1^2\y_3^2
\big)\;\dd\sigma\\
&=
4\int_T \x_2^2\x_4^2\y_1^2\y_3^2\;\dd\sigma
-
8\int_T \x_2^2\x_3\x_4\y_1^2\y_3\y_4\;\dd\sigma
+
4\int_T \z_1\z_2\z_3\z_4\;\dd\sigma
,
\end{align*}
where we used
that 
\begin{align*}
&
\int_T \x_1^2\x_3^2\y_2^2\y_4^2\;\dd\sigma
=\int_T \x_2^2\x_3^2\y_1^2\y_4^2\;\dd\sigma
=
\int_T \x_1^2\x_4^2\y_2^2\y_3^2\;\dd\sigma
=\int_T \x_2^2\x_4^2\y_1^2\y_3^2\;\dd\sigma,\\
&
\int_T \x_1\x_2\x_3^2\y_1\y_2\y_4^2\;\dd\sigma
=
\int_T \x_1^2\x_3\x_4\y_2^2\y_3\y_4\;\dd\sigma
=
\int_T \x_2^2\x_3\x_4\y_1^2\y_3\y_4\;\dd\sigma
=
\int_T \x_1\x_2\x_4^2\y_1\y_2\y_3^2\;\dd\sigma
\end{align*}
due to the 
symmetry of the integral value in the indices of the variables $\x,\y$.
We compute:
\begin{align*}
J_{7}^{(1)}
&=\int_T \x_2^2\x_4^2\y_1^2\y_3^2\;\dd\sigma
=
\underline{\int_n}
\cos(\phi_1)^2\sin(\phi_1)^6
\cos(\phi_2)^2\sin(\phi_2)^4
\cos(\phi_3)^2\sin(\phi_3)^2\cdot\\
&\cdot
\cos(\phi_4)^2
\cos(\psi_1)^4 
V_n(\PHI,\PSI)\;\dd\PHI\dd\PSI
+\underline{\int_n}
\sin(\phi_1)^6
\cos(\phi_2)^4\sin(\phi_2)^2
\cos(\phi_3)^2\sin(\phi_3)^2\cdot\\
&\cdot
\cos(\phi_4)^2
\cos(\psi_1)^2\sin(\psi_1)^2
\cos(\psi_2)^2
V_n(\PHI,\PSI)\;\dd\PHI\dd\PSI
+\underline{\int_n}
\sin(\phi_1)^6
\cos(\phi_2)^2\sin(\phi_2)^2\cdot\\
&\cdot
\sin(\phi_3)^4
\cos(\phi_4)^2
\cos(\psi_1)^2\sin(\psi_1)^2
\sin(\psi_2)^2
\cos(\psi_3)^2
V_n(\PHI,\PSI)\;\dd\PHI\dd\PSI\\
&=
B_{n-2,6}B_{n-3,4}B_{n-4,2}B_{n-5,0}C_{n-3,0}
+
A_{n-2,6}C_{n-3,2}B_{n-4,2}B_{n-5,0}B_{n-3,2}B_{n-4,0}\\
&\hspace{1cm}
+
A_{n-2,6}\big(B_{n-3,2}\big)^2A_{n-4,4}\big(B_{n-5,0}\big)^2A_{n-4,2}
\\
&=
\frac{(n-1)(n+1)(n+3)}{n(n+2)(n+4)(n+6)}
\frac{(n-2)n}{(n-1)(n+1)(n+3)}
\frac{n-3}{(n-2)n}
\frac{1}{n-3}
\frac{3}{(n-1)(n+1)}
\\
&+
\frac{(n-1)(n+1)(n+3)}{n(n+2)(n+4)}
\frac{3(n-2)}{(n-1)(n+1)(n+3)}
\frac{n-3}{(n-2)n}
\frac{1}{n-3}
\frac{n-2}{(n-1)(n+1)}
\frac{1}{n-2}
\\
&+
\frac{(n-1)(n+1)(n+3)}{n(n+2)(n+4)}
\frac{(n-2)^2}{(n-1)^2(n+1)^2}
\frac{(n-3)(n-1)}{(n-2)n}
\frac{1}{(n-3)^2}
\frac{n-3}{n-2}
\\
    &
=
\frac{(n+3)(n+5)}{(n-1)n(n+1)(n+2)(n+4)(n+6)},
\\
J_{7}^{(2)}
&=\int_T \x_2^2\x_3\x_4\y_1^2\y_3\y_4\;\dd\sigma
=
\underline{\int_n}
\cos(\phi_1)^2\sin(\phi_1)^6
\cos(\phi_2)^2\sin(\phi_2)^4\cdot\\
&\cdot
\cos(\phi_3)^2\sin(\phi_3)^2
\cos(\phi_4)^2
\cos(\psi_1)^4 
V_n(\PHI,\PSI)\;\dd\PHI\dd\PSI
+\underline{\int_n}
\sin(\phi_1)^6
\cos(\phi_2)^4\sin(\phi_2)^2\cdot\\
&\cdot
\cos(\phi_3)^2\sin(\phi_3)^2
\cos(\phi_4)^2
\cos(\psi_1)^2\sin(\psi_1)^2
\cos(\psi_2)^2
V_n(\PHI,\PSI)\;\dd\PHI\dd\PSI\\
&-\underline{\int_n}
\sin(\phi_1)^6
\cos(\phi_2)^2\sin(\phi_2)^2
\cos(\phi_3)^2\sin(\phi_3)^2
\cos(\phi_4)^2
\cos(\psi_1)^2\sin(\psi_1)^2\\
&\hspace{2cm}\cdot
\sin(\psi_2)^2
\cos(\psi_3)^2
V_n(\PHI,\PSI)\;\dd\PHI\dd\PSI\\
    &=
        \frac{3}{(n-1)n(n+1)(n+2)(n+4)(n+6)}
        +
        \frac{3}{(n-1)n^2(n+1)(n+2)(n+4)}
        \\
&\hspace{1cm}
-
A_{n-2,6}\big(B_{n-3,2}\big)^2B_{n-4,2}B_{n-5,0}A_{n-4,2}B_{n-5,0}\\
    &=
        \frac{3}{(n-1)n(n+1)(n+2)(n+4)(n+6)}
        +
        \frac{3}{(n-1)n^2(n+1)(n+2)(n+4)}
        \\
&
\hspace{1cm}-
\frac{(n-1)(n+1)(n+3)}{n(n+2)(n+4)}
\frac{(n-2)^2}{(n-1)^2(n+1)^2}
\frac{n-3}{(n-2)n}
\frac{1}{n-3}
\frac{n-3}{n-2}
\frac{1}{n-3}
\\
&=
-\frac{n+3}{(n-1)n(n+1)(n+2)(n+4)(n+6)},
\end{align*}
where we included only the terms with nonzero integrals in the computations.
Hence,
$$
J_{7}
=4J_{7}^{(1)}-8J_{7}^{(2)}+4J_5
=\frac{4}{(n-1)n(n+1)(n+2)}.
$$
The argument for $J_{8}=J_{9}=0$
is the same as for $I_6=I_7=I_8=0$ above.\\

\noindent \textbf{Case 2:} $n=5$.
Note that the parameterizations of 
    $x_i(\PHI)$, $i=1,2,3,4$, 
and 
    $y_i(\PHI,\PSI)$, $i=1,2,3$,
are the same as in the case $n\geq 6$,
while
\begin{align*}
y_4(\PHI,\PSI)
    &=
    \cos(\phi_1)\sin(\phi_2)\sin(\phi_3)\cos(\phi_4)\cos(\psi_1)
    +
    \cos(\phi_2)\sin(\phi_3)\cos(\phi_4)\sin(\psi_1)\cos(\psi_2)\\
    &\hspace{1cm}
    +
    \cos(\phi_3)\cos(\phi_4)\sin(\psi_1)\sin(\psi_2)\cos(\psi_3)
    -
    \sin(\phi_4)\sin(\psi_1)\sin(\psi_2)\sin(\psi_3).
\end{align*}
So the computations of the integrals of monomials from Lemma \ref{integrals} containing at most 3 different indices remain the same as in the case $n\geq 6$.
The remaining formulas containing monomials with possibly more than three different indices are 
$I_6,I_7,I_8,J_5,J_7,J_8,J_9$.
The arguments for 
$I_6=I_7=I_8=J_{8}=J_{9}=0$
are the same as in the case $n\geq 6$.
The argument for $J_5$ following the same formula as in the case $n\geq 6$ also when applied to $n=5$  is the following computation:
\begin{align*}
0
&=
\int_T (\z_1+\z_2+\z_3+\z_4+\z_5)^4\; \dd\sigma=
\sum_{i=1}^5 \int_T \z_i^4\;\dd\sigma 
+4\sum_{i\neq j} \int_T \z_i^3\z_j\;\dd\sigma
+6\sum_{i<j} \int_T \z_i^2\z_j^2\;\dd\sigma\\
&+
12\sum_{
    \substack{   
    i,j,k \text{ pairw.}\\ \text{diff.}, j<k 
    }
} \int_T \z_i^2\z_j\z_k\;\dd\sigma
+
24\sum_{
    \substack{   
    i<j<k<l
    }
} \int_T \z_i\z_j\z_k\z_l\;\dd\sigma\\
&=
5J_1
+4\cdot 2\binom{5}{2}J_2
+6\binom{5}{2}J_3
+12\cdot 5\binom{4}{2}J_4
+24\binom{5}{4}J_5.
\end{align*}
Using $J_1,J_2,J_3,J_4$ as stated in Lemma \ref{integrals} for $n=5$, we get 
$J_5=\frac{1}{27720}$, which is also in accordance with the formula in Lemma \ref{integrals} for $n=5$.
 
It remains to do direct computations for the value of $J_{7}$.
In the notation of case $n\geq 6$ we need to compute $J_{7}^{(1)}$ and $J_{7}^{(2)}$:
\begin{align*}
J_{7}^{(1)}
&=\int_T \x_2^2\x_4^2\y_1^2\y_3^2\;\dd\sigma
=
\underline{\int_5}
\cos(\phi_1)^2\sin(\phi_1)^6
\cos(\phi_2)^2\sin(\phi_2)^4
\cos(\phi_3)^2\sin(\phi_3)^2\cdot\\
&\cdot
\cos(\phi_4)^2
\cos(\psi_1)^4
V_5(\PHI,\PSI)
\;\dd\PHI\dd\PSI
+\underline{\int_5}
\sin(\phi_1)^6
\cos(\phi_2)^4\sin(\phi_2)^2
\cos(\phi_3)^2\sin(\phi_3)^2\cdot\\
&\cdot
\cos(\phi_4)^2
\cos(\psi_1)^2\sin(\psi_1)^2
\cos(\psi_2)^2
V_{5}(\PHI,\PSI)
\;\dd\PHI\dd\PSI
+\underline{\int_5}
\sin(\phi_1)^6
\cos(\phi_2)^2\sin(\phi_2)^2\cdot\\
&\cdot
\sin(\phi_3)^4
\cos(\phi_4)^2
\cos(\psi_1)^2\sin(\psi_1)^2
\sin(\psi_2)^2
\cos(\psi_3)^2
V_{5}(\PHI,\PSI)
\;\dd\PHI\dd\PSI
\\
J_{7}^{(2)}
&=\int_T \x_2^2\x_3\x_4\y_1^2\y_3\y_4\;\dd\sigma
=
\underline{\int_5}
\cos(\phi_1)^2\sin(\phi_1)^6
\cos(\phi_2)^2\sin(\phi_2)^4
\cos(\phi_3)^2\sin(\phi_3)^2\cdot\\
&\cdot
\cos(\phi_4)^2
\cos(\psi_1)^4
V_5(\PHI,\PSI)
\;\dd\PHI\dd\PSI
+\underline{\int_5}
\sin(\phi_1)^6
\cos(\phi_2)^4\sin(\phi_2)^2
\cos(\phi_3)^2\sin(\phi_3)^2\cdot\\
&\cdot
\cos(\phi_4)^2
\cos(\psi_1)^2\sin(\psi_1)^2
\cos(\psi_2)^2
V_5(\PHI,\PSI)
\;\dd\PHI\dd\PSI
-\underline{\int_5}
\sin(\phi_1)^6
\cos(\phi_2)^2\sin(\phi_2)^2\cdot\\
&\cdot
\cos(\phi_3)^2\sin(\phi_3)^2
\cos(\phi_4)^2
\cos(\psi_1)^2\sin(\psi_1)^2
\sin(\psi_2)^2
\cos(\psi_3)^2
V_5(\PHI,\PSI)
\;\dd\PHI\dd\PSI.
\end{align*}
In all of the formulas above the difference from the case $n\geq 6$
is that integration intervals for $\phi_4$ and $\psi_3$ are $[0,2\pi]$ instead of $[0,\pi]$.
However, since 
\begin{equation}
\label{replace-integration-interval}
    \frac
    {\int_0^{2\pi}\cos(\phi)^{2i}\sin(\phi)^{2j}\;\dd\phi}
    {\int_0^{2\pi}1\;\dd\phi}
    =
    \frac
    {\int_0^{\pi}\cos(\phi)^{2i}\sin(\phi)^{2j}\;\dd\phi}
    {\int_0^{\pi}1\;\dd\phi}
    \quad \text{for all }i,j\in \NN\cup\{0\},
\end{equation}
we can replace both intervals $[0,2\pi]$ with $[0,\pi]$.
Since the integrands are precisely as in the case $n\geq 6$ with $n=5$,
the formulas for $J_{7}$ from the case $n\geq 6$ hold also when applied to $n=5.$ 
\\

\noindent \textbf{Case 3:} $n=4$.
Note that the parameterizations of 
    $x_i(\PHI)$, $i=1,2,3$, 
and 
    $y_i(\PHI,\PSI)$, $i=1,2$,
are the same as in the case $n\geq 6$,
while
\begin{align*}
    x_4(\PHI)
    &=\sin(\phi_1)\sin(\phi_2)\sin(\phi_3),\\
y_3(\PHI,\psi_1,\psi_2)
&=\cos(\phi_1)\sin(\phi_2)\cos(\phi_3)\cos(\psi_1)
        +
        \cos(\phi_2)\cos(\phi_3)\sin(\psi_1)\cos(\psi_2)\\
        &\hspace{2cm}
        -
        \sin(\phi_3)\sin(\psi_1)\sin(\psi_2),\\
y_4(\PHI,\psi_1,\psi_2)
&=
\cos(\phi_1)\sin(\phi_2)\sin(\phi_3)\cos(\psi_1)
        +
        \cos(\phi_2)\sin(\phi_3)\sin(\psi_1)\cos(\psi_2) \\
        &\hspace{2cm}
        +
        \cos(\phi_3)\sin(\psi_1)\sin(\psi_2).
\end{align*}
So the computations of the integrals of monomials from Lemma \ref{integrals} containing at most 2 different indices remain the same as in the case $n\geq 6$.
The remaining formulas containing monomials with possibly more than two different indices
are $I_6, I_7,I_8, J_4, J_5, J_7,J_8,J_9$.
The arguments for 
$I_6=I_7=I_8=J_{8}=J_{9}=0$ 
are the same as in the case $n\geq 6$.
The argument for 
$J_4$ is direct computation.
We have:
\begin{align*}
&\int_T \z_1^2\z_2\z_3 \dd\sigma
    =
    \underline{\int_4}
    \cos(\phi_1)^4\sin(\phi_1)^{4}
    \cos(\phi_2)^2\sin(\phi_2)^2
    \cos(\phi_3)^2
    \cos(\psi_1)^4
    V_{4}(\PHI,\PSI)\;\dd\PHI\dd\PSI\\
    & 
    -\underline{\int_4}
    \cos(\phi_1)^2\sin(\phi_1)^{4}
    \cos(\phi_2)^2\sin(\phi_2)^2
    \cos(\phi_3)^2
    \cos(\psi_1)^2\sin(\psi_1)^2
    \cos(\psi_2)^2
    V_{4}(\PHI,\PSI)\;\dd\PHI\dd\PSI.
\end{align*}
By \eqref{replace-integration-interval},
the integration intervals for $\phi_3$ and $\psi_2$ 
can be replaced by $[0,\pi]$ instead of $[0,2\pi]$.
Since the integral is precisely as in the case $n\geq 6$ with $n=4$,
the formula for $J_4$ from the case $n\geq 6$
holds also when applied to $n=4.$ 
The argument for $J_5$ following the same formula as in the case $n\geq 6$ for $n=4$ is the following computation:
\begin{align*}
0
&=
\int_T (\z_1+\z_2+\z_3+\z_4)^4\; \dd\sigma=
\sum_{i=1}^4 \int_T \z_i^4\;\dd\sigma 
+4\sum_{i\neq j} \int_T \z_i^3\z_j\;\dd\sigma
+6\sum_{i<j} \int_T \z_i^2\z_j^2\;\dd\sigma\\
&+
12\sum_{
    \substack{   
    i,j,k \text{ pairw.}\\ \text{diff.}, j<k 
    }
} \int_T \z_i^2\z_j\z_k\;\dd\sigma
+
24
 \int_T \z_1\z_2\z_3\z_4\;\dd\sigma\\
&=
4J_1
+4\cdot 2\binom{4}{2}J_2
+6\binom{4}{2}J_3
+12\cdot 4\binom{3}{2}J_4
+24J_5.
\end{align*}
Using $J_1,J_2,J_3,J_4$ as stated in Lemma \ref{integrals} for $n=4$, we get 
$J_5=\frac{1}{9600}$, which is also as stated in Lemma \ref{integrals} for $n=4$. 

It remains to do direct computations for the value of $J_{7}$.
In the notation of case $n\geq 6$ we need to compute $J_{7}^{(1)}$ and $J_{7}^{(2)}$:
\begin{align*}
J_{7}^{(1)}
&=\int_T \x_2^2\x_4^2\y_1^2\y_3^2\;\dd\sigma
=
\underline{\int_4}
\cos(\phi_1)^2\sin(\phi_1)^6
\cos(\phi_2)^2\sin(\phi_2)^4
\cos(\phi_3)^2\sin(\phi_3)^2\cdot\\
&\cdot
\cos(\psi_1)^4
V_4(\PHI,\PSI)
\;\dd\PHI\dd\PSI
+\underline{\int_4}
\sin(\phi_1)^6
\cos(\phi_2)^4\sin(\phi_2)^2
\cos(\phi_3)^2\sin(\phi_3)^2\cdot\\
&\cdot
\cos(\psi_1)^2\sin(\psi_1)^2
\cos(\psi_2)^2
V_{4}(\PHI,\PSI)
\;\dd\PHI\dd\PSI
+\underline{\int_4}
\sin(\phi_1)^6
\cos(\phi_2)^2\sin(\phi_2)^2\cdot\\
&\cdot
\sin(\phi_3)^4
\cos(\psi_1)^2\sin(\psi_1)^2
\sin(\psi_2)^2
V_{4}(\PHI,\PSI)
\;\dd\PHI\dd\PSI
\\
J_{7}^{(2)}
&=\int_T \x_2^2\x_3\x_4\y_1^2\y_3\y_4\;\dd\sigma
=
\underline{\int_4}
\cos(\phi_1)^2\sin(\phi_1)^6
\cos(\phi_2)^2\sin(\phi_2)^4
\cos(\phi_3)^2\sin(\phi_3)^2\cdot\\
&\cdot
\cos(\psi_1)^4
V_4(\PHI,\PSI)
\;\dd\PHI\dd\PSI
+\underline{\int_4}
\sin(\phi_1)^6
\cos(\phi_2)^4\sin(\phi_2)^2
\cos(\phi_3)^2\sin(\phi_3)^2\cdot\\
&\cdot
\cos(\psi_1)^2\sin(\psi_1)^2
\cos(\psi_2)^2
V_4(\PHI,\PSI)
\;\dd\PHI\dd\PSI
-\underline{\int_4}
\sin(\phi_1)^6
\cos(\phi_2)^2\sin(\phi_2)^2\cdot\\
&\cdot
\cos(\phi_3)^2\sin(\phi_3)^2
\cos(\psi_1)^2\sin(\psi_1)^2
\sin(\psi_2)^2
V_4(\PHI,\PSI)
\;\dd\PHI\dd\PSI.
\end{align*}
By \eqref{replace-integration-interval},
the integration intervals for $\phi_3$ and $\psi_2$ 
can be replaced by $[0,\pi]$ instead of $[0,2\pi]$.
The difference in the integrands in comparison to the case $n\geq 6$
is that every summand lacks the $\cos(\phi_4)^2$ term, while some summands lack the $\cos(\psi_3)^2$ term. However, looking at
the computation for $n\geq 6$ both correspond to the term $B_{n-5,0}=\frac{1}{n-3}$.
For $n=4$ this term is equal to 1, so the final formula for $J_{7}$ is the same as in the case $n\geq 6$ when applied to $n=4.$ 
\\

\noindent \textbf{Case 4:} $n=3$.
Note that the parameterizations of 
    $x_1(\PHI)$, $x_2(\PHI)$ 
and
    $y_1(\PHI,\PSI)$
are the same as in the case $n\geq 6$,
while
\begin{align*}
    x_3(\phi_1,\phi_2)
    &=\sin(\phi_1)\sin(\phi_2),\\
    y_2(\phi_1,\phi_2,\psi_1)
    &=
        \cos(\phi_1)\cos(\phi_2)\cos(\psi_1)
        -\sin(\phi_2)\sin(\psi_1),\\
    y_3(\phi_1,\phi_2,\psi_1)
    &=
        \cos(\phi_1)\sin(\phi_2)\cos(\psi_1)
        +
        \cos(\phi_2)\sin(\psi_1).
\end{align*}
So the computations of the integrals of monomials from Lemma \ref{integrals}
containing one different index remain the same as in the case $n\geq 6$, i.e., 
$I_1$ and $J_1$ hold for $n=3$.
The arguments for $I_6=I_7=I_8=J_{8}=J_{9}=0$ are the same as in the case $n\geq 6$.
Assuming $I_2=-\frac{1}{2}I_1$, the arguments proving formulas $I_3,I_4,I_5$ are the same as in the case $n\geq 6$.
It remains to establish the formulas for $I_2,J_2,J_3,J_4,J_6$
by direct computations:
\begin{align*}
    I_2
    &=\int_T \z_1\z_2 \;\dd\sigma
    =
    -
    \underline{\int_3}
    \cos(\phi_1)^2\sin(\phi_1)^2\cos(\phi_2)^2\cos(\psi_1)^2 
    V_{3}(\phi_1)\;\dd\phi_1\dd\phi_2\dd\psi_1,\\
    J_2
    &=
    \int_T \z_1^3\z_2 \;d\sigma
    =
    -
    \underline{\int_3}
    \cos(\phi_1)^4\sin(\phi_1)^4
    \cos(\phi_2)^2
    \cos(\psi_1)^4
    V_{3}(\phi_1)\;\dd\phi_1\dd\phi_2\dd\psi_1
    \\
    J_3
    &=
    \int_T \z_1^2\z_2^2 \;\dd\sigma
    =
    \underline{\int_3}
    \cos(\phi_1)^4\sin(\phi_1)^4
    \cos(\phi_2)^4
    \cos(\psi_1)^4
    V_{3}(\phi_1)\;\dd\phi_1\dd\phi_2\dd\psi_1\\
    &+
    \underline{\int_3}
    \cos(\phi_1)^2\sin(\phi_1)^4
    \cos(\phi_2)^2\sin(\phi_2)^2
    \cos(\psi_1)^2\sin(\psi_1)^2
    V_{3}(\phi_1)\;\dd\phi_1\dd\phi_2\dd\psi_1
    \\
    J_4
    &=\int_T \z_1^2\z_2\z_3 d\sigma
    =
    \underline{\int_3}
    \cos(\phi_1)^4\sin(\phi_1)^4
    \cos(\phi_2)^2\sin(\phi_2)^2
    \cos(\psi_1)^4\cdot\\
    &\cdot
    V_{3}(\phi_1)\;\dd\phi_1\dd\phi_2\dd\psi_1-
    \underline{\int_3}
    \cos(\phi_1)^2\sin(\phi_1)^4
    \cos(\phi_2)^2\sin(\phi_2)^2
    \cos(\psi_1)^2\sin(\phi_1)^2\cdot\\
    &\cdot
    V_{3}(\phi_1)\;\dd\phi_1\dd\phi_2\dd\psi_1\\
    J_{6}^{(1)}
&=\int_T \x_2^4\y_1^4\;\dd\sigma
=
\underline{\int_3}
\sin(\phi_1)^8\cos(\phi_2)^4\cos(\psi_1)^4 
V_{3}(\phi_1)\;\dd\phi_1\dd\phi_2\dd\psi_1\\
J_6^{(2)}
&=\int_T \x_1\x_2^3\y_1^3\y_2\;\dd\sigma=
-
\underline{\int_3}
\cos(\phi_1)^2\sin(\phi_1)^6
\cos(\phi_2)^4
\cos(\psi_1)^4 V_{3}(\phi_1)\;\dd\phi_1\dd\phi_2\dd\psi_1,
\end{align*}
where we included only terms with nonzero integrals in the 
computations above.
By \eqref{replace-integration-interval},
the integration intervals for $\phi_2$ and $\psi_1$ 
can be replaced by $[0,\pi]$ instead of $[0,2\pi]$.
The integrands of $I_2$, $J_2$, $J_6^{(1)}$, $J_6^{(2)}$ are the same as in the case $n\geq 6$ and hence also the corresponding formulas when applied to $n=3$.
The difference in the integrands of $J_3$ and $J_4$
in comparison with the case $n\geq 6$
is that some summands lack at least one of the terms
$\cos(\phi_3)^2$ or $\cos(\psi_2)^2$.
Looking at
the computation for $n\geq 6$ these terms 
correspond to the factor $B_{n-4,0}=\frac{1}{n-2}$.
For $n=3$ this term is equal to 1.
So the final formulas for $J_3$ and $J_4$ 
are the same as in the case $n\geq 6$ when applied to $n=3$. Hence, also the formula for $J_6$ is the same when applied to $n=3$ by the argument as in the case $n\geq 6$.
\end{proof}

\subsubsection{Proof of Proposition \ref{bask-prod} for a symmetric bilinear form $g$}
\label{symmetric-case}
By the action
    $g(\x,\y)=g(U^{-1}\x,U^{-1}\y)$, $U\in SO(n)$, 
we can assume without loss of generality that $g$ is of the form
$$
g(\z)=
d_1 \z_1+d_2\z_2+\ldots+d_n\z_n,
\quad d_i\in \RR.
$$
Raising both sides of \eqref{reverse-holder} to the power of 4, we have to prove that
\begin{align}
\label{ineq-to-prove}
    \int_T\left(\sum_i d_i\z_i\right)^4\; \dd\sigma 
    \leq
    9\left(\int_T\left(\sum_i d_i\z_i\right)^2\; \dd\sigma\right)^2.
\end{align}
We can assume that
$\|g\|_2=1$:
\begin{align}
\label{2-norm-is-1}
\left(\sum_i d_i^2-\frac{2}{n-1}\sum_{i< j}d_id_j\right)I_1=1,
\end{align}
where we used that 
    $I_2=-\frac{1}{n-1}I_1$ (see Lemma \ref{integrals}).
Squaring \eqref{2-norm-is-1} we also have
\begin{align}
\label{2-norm-is-1-v2}
\begin{split}
&\sum_{i}d_i^4 
    -
    \frac{1}{n-1}
    \sum_{i\neq j}4d_i^3d_j 
    +
    2\left(1+\frac{2}{(n-1)^2}\right)
    \sum_{i<j} d_i^2d_j^2\\
    &\hspace{1cm}
    -
    \frac{4}{n-1}
    \left(
    1-\frac{2}{n-1}
    \right)
    \sum_{\substack{i,j,k\\ \text{pairwise}\\\text{different},\\j<k}}
        d_i^2d_jd_k
    +
    \frac{1}{(n-1)^2}
    \sum_{i<j<k<l}
        24d_id_jd_kd_l
        =\frac{1}{I_1^2}.
\end{split}
\end{align}
Using  $J_2=-\frac{1}{n-1}J_1$ (see Lemma \ref{integrals}), \eqref{2-norm-is-1} and \eqref{2-norm-is-1-v2} in
\eqref{ineq-to-prove}, the latter is equivalent to:
\begin{align*}
   0
   &\leq
   9
   -
   \left(\sum_{i}d_i^4 
    -
    \frac{\sum_{i\neq j}4d_i^3d_j}{n-1}
    \right)
    J_1
    -6
    \sum_{i<j} d_i^2d_j^2
    J_3
    -
    12\sum_{\substack{i,j,k\\ \text{pairwise}\\\text{different},\\j<k}}
        d_i^2d_jd_k J_4\\
    &\hspace{5cm}
    -
    \sum_{i<j<k<l}
        24d_id_jd_kd_l 
        \cdot
    \left\{
    \begin{array}{rr}
         J_5&   \text{if }n\geq 4\\[0.3em]
         0,&    \text{if }n=3
    \end{array}
    \right.\\
   &=
   9-\frac{J_1}{I_1^2}
    +
    \sum_{i<j} d_i^2d_j^2
    \Big(-6J_3+2J_1+\frac{4}{(n-1)^2}J_1\Big)\\
    &\hspace{2cm}
    +
    \sum_{\substack{i,j,k\\ \text{pairwise}\\\text{different},\\j<k}}
        d_i^2d_jd_k 
    \Big(-12
    J_4
    -\frac{4}{n-1}J_1
    +\frac{8}{(n-1)^2}J_1
    \Big)\\
    &\hspace{2cm}
    +
    \sum_{i<j<k<l}
        24d_id_jd_kd_l 
    \cdot
    \left\{
    \begin{array}{rr}
         -J_5+
    \frac{1}{(n-1)^2}J_1,
    &   \text{if }n\geq 4\\[0.3em]
         0,&    \text{if }n=3
    \end{array}
    \right.
    \\
    &=
   9
    -\frac{J_1}{I_1^2}+\frac{12}{(n-1)^2n(n+1)(n+4)(n+6)}\\
   &\hspace{2cm}
    \Big((n-3)\Big(\sum_{i<j} d_i^2d_j^2
   (n-2)
   -
    2\sum_{\substack{i,j,k\\ \text{pairwise}\\\text{different},\\j<k}}
        d_i^2d_jd_k 
    \Big)+
        12\sum_{i<j<k<l}d_id_jd_kd_l\Big).
\end{align*}
Since 
$$
\frac{J_1}{I_1^2}
=
\frac{9n^2(n+2)^2}{n(n+2)(n+4)(n+6)}
=
\frac{9n(n+2)}{(n+4)(n+6)}
\leq 9,
$$
it suffices to prove that 
\begin{align}
\label{ineq-to-prove-v3}
    (n-3)\Big(\sum_{i<j} d_i^2d_j^2
   (n-2)
   -
    2\sum_{\substack{i,j,k\\ \text{pairwise}\\\text{different},\\j<k}}
        d_i^2d_jd_k 
    \Big)+
        12\sum_{i<j<k<l}d_id_jd_kd_l
    \geq 0.
\end{align}
We will use induction on $n$ to verify \eqref{ineq-to-prove-v3}, starting with $n=3$. Clearly, for $n=3$ both terms are 0
and we have equality in \eqref{ineq-to-prove-v3}.
Let us assume \eqref{ineq-to-prove-v3} holds for some $n$ and prove it for $n+1$. 
Note that in all inequalities \eqref{ineq-to-prove}--\eqref{ineq-to-prove-v3}, the validity for one tuple $(d_1,\ldots,d_n)$ implies the validity for 
every tuple $(d_1+a,\ldots,d_n+a)$, where $a\in \RR$.
The reason for this is that $a(\z_1+\ldots+\z_n)\equiv 0$ on $T$ and hence 
$\int_T(\sum_i (d_i+a)z_i)^l\; \dd\sigma
=\int_T(\sum_i d_iz_i)^l\; \dd\sigma
$
for every $l\in \NN$.
Using this and the symmetry on indices of coefficients we can assume that
\begin{equation}
    \label{nonneg-assumpt-v2}
    d_1\geq d_2\geq \ldots \geq d_{n}\geq d_{n+1}=0.
\end{equation}
Using \eqref{nonneg-assumpt-v2} in \eqref{ineq-to-prove-v3}
it remains to prove
\begin{align}
\label{ineq-to-prove-v4}
    (n-2)\Big(\sum_{i<j\leq n} d_i^2d_j^2
   (n-1)
   -
    2\sum_{\substack{i,j,k\leq n\\ \text{pairwise}\\\text{different},\\j<k}}
        d_i^2d_jd_k 
    \Big)+
        12\sum_{i<j<k<l\leq n}d_id_jd_kd_l
    \geq 0.
\end{align}
We can rewrite \eqref{ineq-to-prove-v4} into
\begin{align}
\label{ineq-to-prove-v5}
\begin{split}
    &(n-3)\Big(\sum_{i<j\leq n} d_i^2d_j^2
   (n-2)
   -
    2\sum_{\substack{i,j,k\leq n\\ \text{pairwise}\\\text{different},\\j<k}}
        d_i^2d_jd_k 
    \Big)+
        12\sum_{i<j<k<l\leq n}d_id_jd_kd_l\\
    &\hspace{2cm}
    +
    2\Big(\sum_{i<j\leq n} d_i^2d_j^2
   (n-2)
   -
    \sum_{\substack{i,j,k\leq n\\ \text{pairwise}\\\text{different},\\j<k}}
        d_i^2d_jd_k 
    \Big)
    \geq 0.
\end{split}
\end{align}
In \eqref{ineq-to-prove-v5}, the first summand 
is nonnegative by the induction hypothesis, the second summand 
by \eqref{nonneg-assumpt-v2} and the third summand
by Muirhead's inequality \cite{Mui03}. Namely,
let
$(s_1,s_{2},s_3,\ldots,s_n)=(2,2,0,\ldots,0)$
and
$(t_1,t_{2},t_3,t_4,\ldots,t_n)=(2,1,1,0,\ldots,0)$.
Note 
$\sum_{i=1}^{n}s_i=\sum_{i=1}^n t_i$
and
$\sum_{i=1}^{k}s_i\geq \sum_{i=1}^k t_i$ for $k=1,\ldots,n-1$.
Since $d_i\geq 0$ for $i=1,\ldots,n$, \cite{Mui03} implies 
\begin{equation}
    \label{ineq:Muirhead}
    \sum_{\sigma\in S_n} d_1^{s_{\sigma(1)}}d_2^{s_{\sigma(2)}}\cdots d_n^{s_{\sigma(n)}}
    \geq 
    \sum_{\sigma\in S_n} d_1^{t_{\sigma(1)}}d_2^{t_{\sigma(2)}}\cdots d_n^{t_{\sigma(n)}},
\end{equation}
where $S_n$ stands for the symmetric group on $\{1,2,\ldots,n\}$.
Note \eqref{ineq:Muirhead} is equivalent to
\begin{equation*}
    \sum_{i<j\leq n}  2(n-2)!\cdot d_i^2d_j^2 
    \geq 
    \sum_{\substack{i,j,k\leq n\\ \text{pairwise}\\\text{different},\\j<k}}
        2(n-3)!\cdot d_i^2d_jd_k,
\end{equation*}
which implies the nonnegativity of the third summand in \eqref{ineq-to-prove-v5}.

It remains to prove the moreover part of Proposition \ref{bask-prod}
for symmetric bilinear forms,
i.e., the constant $\sqrt{3}$ in the inequality 
$\|g\|_4\leq \sqrt{3} \|g\|_2$ is asymptotically sharp as $n\to\infty$.
For $g(\z)=\frac{1}{\sqrt{I_1}}\z_1$ note
$$\|g\|_4=
\sqrt[4]{\frac{J_1}{I_1^2}}=
\sqrt[4]{
\frac{9n(n+2)}{(n+4)(n+6)}}\|g\|_2.
$$
As $n\to \infty$, we deduce that 
$\|g\|_4\to \sqrt{3}\|g\|_2$.
\qed


\subsubsection{Proof of Proposition \ref{bask-prod} for a skew-symmetric bilinear form $g$}
\label{skew-symmetric-case}
By the action
    $g(\x,\y)=g(U^{-1}\x,U^{-1}\y)$, $U\in SO(n)$, 
we can assume without loss of generality (see \cite[Corollary 2.5.11]{HJ13} that $g$ is of the form
$$
g(\w)=
a_{12}\w_{12}+a_{34}\w_{34}+\ldots+
a_{2\lfloor \frac{n}{2} \rfloor-1,2\lfloor \frac{n}{2} \rfloor}
\w_{2\lfloor \frac{n}{2} \rfloor-1,2\lfloor \frac{n}{2} \rfloor},
\quad 
a_{i,i+1}\in \RR.
$$
Raising both sides of \eqref{reverse-holder} to the power of 4, we have to prove that
\begin{align}
\label{ineq-to-prove-skew}
    \int_T
        \Big(\sum_{1\leq i\leq \lfloor \frac{n}{2} \rfloor} 
            a_{2i-1,2i} \w_{2i-1,2i}\Big)^4 \;\dd\sigma 
    \leq
    6
    \left(
    \int_T
        \Big(
        \sum_{1\leq i\leq \lfloor \frac{n}{2} \rfloor} 
            a_{2i-1,2i} \w_{2i-1,2i}
        \Big)^2 \;\dd\sigma
    \right)^2.
\end{align}
We can assume that
$\|g\|_2=1$:
\begin{align}
\label{2-norm-is-1-skew}
\begin{split}
&\int_T 
\Big( 
    \sum_{1\leq i,j\leq \lfloor \frac{n}{2} \rfloor}
        a_{2i-1,2i}a_{2j-1,2j} \w_{2i-1,2i}\w_{2j-1,2j}
\Big)
\;\dd\sigma\\
&=
\Big(\sum_{1\leq i\leq \lfloor \frac{n}{2} \rfloor} a_{2i-1,2i}^2\Big)I_5
+
2\Big(
\sum_{
    1\leq i<j\leq \lfloor \frac{n}{2} \rfloor,
    }
    a_{2i-1,2i}a_{2j-1,2j}
\Big)I_8\\
&=
\Big(\sum_{1\leq i\leq \lfloor \frac{n}{2} \rfloor} a_{2i-1,2i}^2\Big)I_5
=
1,
\end{split}
\end{align}
where we used Lemma \ref{integrals}.
Squaring \eqref{2-norm-is-1-skew} we also have
\begin{align}
\label{2-norm-is-1-v2-skew}
&
\sum_{1\leq i\leq \lfloor \frac{n}{2} \rfloor} a_{2i-1,2i}^4
+
2\sum_{1\leq i<j\leq \lfloor \frac{n}{2} \rfloor} a_{2i-1,2i}^2a_{2j-1,2j}^2
        =\frac{1}{I_5^2}.
\end{align}
Computing the left-hand side of \eqref{ineq-to-prove-skew} using Lemma \ref{integrals} we get
\begin{align}
    \label{LHS-ineq-skew}
    \begin{split}
    &\int_T 
    \Big(
    \sum_{1\leq i,j,k,l\leq \lfloor \frac{n}{2} \rfloor}
        a_{2i-1,2i}a_{2j-1,2j}a_{2k-1,2k}a_{2l-1,2l}
        \w_{2i-1,2i}\w_{2j-1,2j}\w_{2k-1,2k}\w_{2l-1,2l}
    \Big)\; \dd\sigma\\
    &=
    \sum_{1\leq i\leq \lfloor \frac{n}{2} \rfloor}
    a_{2i-1,2i}^4 J_6
    +
    6
    \sum_{
        1\leq i<j\leq \lfloor \frac{n}{2} \rfloor
        }
        a_{2i-1,2i}^2 a_{2j-1,2j}^2 J_{7}\\
    &=
    \Big(\sum_{1\leq i\leq \lfloor \frac{n}{2} \rfloor}
    a_{2i-1,2i}^4 
    +
    \sum_{
        1\leq i<j\leq \lfloor \frac{n}{2} \rfloor
        }
        a_{2i-1,2i}^2 a_{2j-1,2j}^2\Big) J_{6}.
    \end{split},
\end{align}
In the computation above we used the fact that all integrals 
$$\int_T \w_{2i-1,2i}\w_{2j-1,2j}\w_{2k-1,2k}\w_{2l-1,2l}\;\dd\sigma,$$
where at least one index appears an odd number of times, are equal to 0.
Using 
\eqref{2-norm-is-1-skew}, \eqref{2-norm-is-1-v2-skew} and \eqref{LHS-ineq-skew}
in
\eqref{ineq-to-prove-skew}, the latter is equivalent to:
\begin{align}
\label{ineq-to-prove-v2}
\begin{split}
   0
   &\leq
   6
   -
   \frac{J_6}{I_5^2}
   +
       \sum_{
        1\leq i<j\leq \lfloor \frac{n}{2} \rfloor
        }
        a_{2i-1,2i}^2 a_{2j-1,2j}^2
    J_6.
\end{split}
\end{align}
Since 
$$
\frac{J_6}{I_5^2}
=
\frac{24(n-1)^2n^2}{4(n-1)n(n+1)(n+2)}
=
\frac{6(n-1)n}{(n+1)(n+2)}
\leq 6,
$$
the inequality \eqref{ineq-to-prove-v2} clearly holds.

It remains to prove the moreover part of Proposition \ref{bask-prod}
for skew-symmetric bilinear forms,
i.e., the constant $\sqrt[4]{6}$ in the inequality 
$\|g\|_4\leq \sqrt[4]{6} \|g\|_2$ is asymptotically sharp as $n\to\infty$.
For $g(\w)=\frac{1}{\sqrt{I_5}}\w_{12}$, note
$$\|g\|_4=
\sqrt[4]{\frac{J_6}{I_5^2}}=
\sqrt[4]{
\frac{6(n-1)n}{(n+1)(n+2)}}\|g\|_2.
$$
As $n\to \infty$, we deduce that 
$\|g\|_4\to \sqrt[4]{6}\|g\|_2$.
\qed


\subsubsection{Proof of Proposition \ref{bask-prod} for a general bilinear form $g$}
\label{general-case}
    Let $g(\x,\y)=\x^T A\y+I_{1,1}\in \RR[\x,\y]_{1,1}/I_{1,1}$ be a bilinear form,
    where $A\in \RR^{n\times n}$. 
    We can write
    $$
        g(\x,\y)
        =
        \underbrace{\Big(\frac{\x^T(A+A^T)\y}{2}+I_{1,1}\Big)}_{g_s(\x,\y)}
        +
        \underbrace{\Big(\frac{\x^T(A-A^T)\y}{2}+I_{1,1}\Big)}_{g_a(\x,\y)},
    $$
    where 
        $g_s(\x,\y)$ and $g_a(\x,\y)$ 
    are symmetric and skew-symmetric bilinear forms, respectively.
    We can write $g_s$ and $g_a$ as follows:
    \begin{align*}
    g_s(\z)
    &=
    \underbrace{\sum_{i=1}^n a_i \z_{i}}_{g_{s,1}}
    +
    \underbrace{\sum_{1\leq i<j\leq n} b_{ij} (\z_{ij}+\z_{ji})}_{g_{s,2}},\quad
    a_i\in \RR, b_{ij} \in \RR,\\
    g_a(\z)
    &=
    \sum_{1\leq i<j\leq n} c_{ij} (\z_{ij}-\z_{ji}),\quad
    c_{ij} \in \RR.
    \end{align*}
    
    \noindent \textbf{Claim.} 
    $\langle g_s,g_a\rangle:=\int_T g_s g_a \;\dd\sigma
    =0.$ \\

    \noindent\textit{Proof of Claim.}
    Since $g_s=g_{s,1}+g_{s,2}$, it suffices to
    prove that 
    $$
    \langle g_{s,1},g_a\rangle
    =
    \langle g_{s,2},g_{a}\rangle
    =
    0.
    $$
    The fact that 
    $\langle g_{s,1},g_{a}\rangle=0$
    follows by observing that
    $\langle g_{s,1},g_{a}\rangle$
    is a weighted sum of inner products of the form
    $$
    \langle \z_i,(\z_{jk}- \z_{kj})\rangle
    =
    \int_T \z_{ii}\z_{jk} \dd\sigma - \int_T \z_{ii}\z_{kj} \dd\sigma.
    $$
    But the values of both terms are equal to $0$ by Lemma \ref{integrals}, since at least one of the indices $i,j,k$
    occurs an odd number of times.
    
    The fact that 
    $\langle g_{s,2},g_{a}\rangle=0$
    follows by noticing that 
    $\langle g_{s,2},g_{a}\rangle$ is a weighted sum  of inner products of the form
    \begin{align*}
    \langle (\z_{ij}+\z_{ji}),(\z_{kl}-\z_{lk})\rangle
    &=
    \langle \z_{ij},\z_{kl}\rangle
    -\langle \z_{ij},\z_{lk}\rangle
    +\langle \z_{ji},\z_{kl}\rangle
    -\langle \z_{ji},\z_{lk}\rangle\\
    &=
    \int_T \z_{ij}\z_{kl}\;\dd\sigma
    -
    \int_T \z_{ij}\z_{lk}\;\dd\sigma
    +
    \int_T \z_{ji}\z_{kl}\;\dd\sigma
    -
    \int_T \z_{ji}\z_{lk}\;\dd\sigma.
    \end{align*}
    If $(i,j)\neq (k,l)$, then in 
    $\z_{ij}\z_{kl}$, $\z_{ij}\z_{lk}$, $\z_{ji}\z_{kl}$, $\z_{ji}\z_{lk}$
    at least one of the indices $i,j,k,l$
    occurs an odd number of times and hence the 
    corresponding integral is equal to 0 by 
    Lemma \ref{integrals}.
    If $(i,j)=(k,l)$, then
    $\z_{ij}\z_{kl}=\z_{ij}^2$,
    $\z_{ij}\z_{lk}=\z_{ij}\z_{ji}$,
    $\z_{ji}\z_{kl}=\z_{ji}\z_{ij}$,
    $\z_{ji}\z_{lk}=\z_{ji}^2$,
    and hence
    \begin{align*}
    \langle (\z_{ij}+\z_{ji}),(\z_{ij}-\z_{ji})\rangle
    &=
    \int_T \z_{ij}^2\;\dd\sigma
    -
    \int_T \z_{ji}^2\;\dd\sigma
    =
    0.
    \end{align*}
    This proves the Claim.\hfill \ensuremath{\Box}\\ 

    We have
    \begin{align*}
    \|g\|_4
    &=
    \|g_{s}+g_a\|_4
    \leq
    \|g_{s}\|_4+\|g_a\|_4
    \leq \sqrt{3}\|g_s\|_2+\sqrt[4]{6}\|g_a\|_2\\
    &\leq \sqrt{3}\big(\|g_s\|_2+\|g_a\|_2\big)
    \leq \sqrt{6}\|g\|_2,
    \end{align*}
    where 
    in the first inequality we used the triangle inequality for $\|\cdot\|_4$,
    in the second the statement of Proposition \ref{bask-prod} for symmetric (resp.\ skew-symmetric) bilinear forms, in the third $\sqrt[4]{6}<\sqrt{3}$ and in the last 
    $\|g_s\|_2+\|g_a\|_2\leq \sqrt{2}\|g\|_2.$ 
    Indeed, by the Claim 
    \begin{equation}
        \label{eq-1709}
        \|g\|_2=\sqrt{\|g_a\|_2^2+\|g_{s}\|_2^2}.
    \end{equation}
    Further,
    \begin{align}
        \label{ineq-1709}
        \big(\|g_a\|_2+\|g_{s}\|_2\big)^2   
        &=
        \|g_a\|_2^2+2\|g_a\|\|g_s\|+\|g_{s}\|_2^2
        \leq 
        2\big(\|g_a\|_2^2+\|g_s\|_2^2\big),
    \end{align}
    where we used inequality between the arithmetic mean and the geometric mean.
    Using \eqref{eq-1709} in \eqref{ineq-1709} gives
    $\|g_s\|_2+\|g_a\|_2\leq \sqrt{2}\|g\|_2.$ 
\qed


\section{Algorithms and Examples}\label{sec:algo}

Each biquadratic form $f\in \RR[\x,\y]_{2, 2}$ that is
modulo $I$ a nonnegative polynomial but         not a sum of squares yields
  an example of a ``proper'' \crp map $A:M_n(\RR)\to M_n(\RR)$,
  cf.~Proposition 
  \ref{prop:obviousPQccp}.
In this section we specialize
the Blekher\-man-Smith-Velasco { procedure} (\cite[Procedure 3.3]{BSV}; see
also
\cite{KMSZ19} for a specialization in the context of positive linear maps)
to produce many such examples from random input data. Throughout this
section we fix $n\geq2$.

Observe that  biquadratic forms are in bijective correspondence
with quadratic forms on the Segre variety (cf.~\cite[Example 5.6]{BSV} or \cite[Lemma 3.13]{Vel15}). Let $\mathbb{P}^{n-1}$ denote the complex $(n-1)$-dimensional projective space and let
$                \sigma_{n}:\PP^{n-1}\times \PP^{n-1} \to \PP^{n^2-1},$ $
([x_1:\ldots:x_n], [y_1:\ldots:y_n]) \mapsto [x_1y_1:x_1y_2:\ldots:
x_1y_n:\ldots:x_ny_n]$
be the Segre embedding. Its image $\sigma _{n}(\PP^{n-1}\times
\PP^{n-1})=V_{\mathbb{C}}(I_{n})$ is the complex zero locus of the ideal
$I_{n}\subseteq \RR[\z]:=\RR[\z_{11},\z_{12},\ldots ,\z_{1n},\ldots
,\z_{nn}]$ generated by all $2\times 2$ minors of the matrix
$\left(\z_{ij}\right)_{i,j}$.
The complexification $I_n^{\mathbb{C}}\subseteq \CC[\z]$ of the ideal $I_{n}$ is radical \cite[p.\ 98]{Har92} and thus
consists of all polynomials vanishing on $\sigma_{n}(\PP^{n-1}\times
\PP^{n-1})$. It is also well known that $\sigma_{n}(\PP^{n-1}\times
\PP^{n-1})$ is smooth
  \cite[p.\ 184-185]{Har92} and of degree $\binom{2n-2}{n-1}$ 
\cite[p.\ 233]{Har92}.

Let $J\subseteq \RR[\z]$ be the ideal generated by $\sum_{i=1}^n \z_{ii}$, let $J_n=I_n+J$, and let $J^{\CC}$ and $J_n^{\CC}$ be complexifications of $J$ and $J_n$ in $\CC[\z]$.

\begin{lem}
The ideal $J_n^\CC$ is radical.
\end{lem}

\begin{proof}
We first show that $J_n^{\CC}/I_n^{\CC}$ is a radical ideal in $\CC[\z]/I_n^{\CC}$. Let
$f\in \CC[\z]/I_n^{\CC}$ satisfy $f^2\in J_n^{\CC}/I_n^{\CC}$. Since $I_n^{\CC}$ is radical
ideal and $V_{\CC}(I_n)$ is the image of the Segre embedding, the Segre
embedding induces an injective homomorphism between coordinate rings
$\tilde\sigma_n^\#:\CC[\z]/I_n^{\CC}\to\CC[\x,\y]$ sending $\z_{ij}+I_n^{\CC}\mapsto
\x_i\y_j$. Clearly, $\tilde\sigma_n^\#(J_n^{\CC}/I_n^{\CC})\subseteq I^{\CC}$, so
$\left(\tilde\sigma_n^\#(f)\right)^2=\tilde\sigma_n^\#(f^2)\in I^{\CC}$.
Since $I^{\CC}$ is a prime ideal in $\CC[\x,\y]$, it follows that
$\tilde\sigma_n^\#(f)\in I^{\CC}$. Let
\begin{equation}\label{f=g.tr(z)}
\tilde\sigma_n^\#(f)=g\cdot \sum_{i=1}^n\x_i\y_i
\end{equation}
for some $g\in \CC[\x,\y]$. Since $\tilde\sigma_n^\#(f)$ lies in the
image of $\tilde\sigma_n^\#$, each of its monomials is of bidegree
$(d,d)$ for some $d$ (which depends on the monomial). Comparing the
monomials in (\ref{f=g.tr(z)}) we see that the same holds for $g$, i.e.,
$g\in \mathrm{im}\, \tilde\sigma_n^\#$. Let $h\in \CC[\z]/I_n^{\CC}$ satisfy
$\tilde\sigma_n^\#(h)=g$. Then (\ref{f=g.tr(z)}) implies
$\tilde\sigma_n^\#\left(f-h\cdot (\sum_{i=1}^n\z_{ii}+I_n^{\CC})\right)=0$,
and injectivity of $\tilde\sigma_n^\#$ implies that $f=h\cdot
(\sum_i\z_{ii}+I_n^{\CC})\in J_n^{\CC}/I_n^{\CC}$. So, $J_n^{\CC}/I_n^{\CC}$ is radical.

Finally, $\CC[\z]/I_n^{\CC}$ is a domain, and $J_n^{\CC}/I_n^{\CC}$ is a radical ideal.
It follows that $$\left(\CC[\z]/I_n^{\CC}\right)/\left(J_n^{\CC}/I_n^{\CC}\right)\cong
\CC[\z]/J_n^{\CC}$$ is reduced (without nilpotents), hence $J_n^{\CC}$ is a radical
ideal in $\CC[\z]$.
\end{proof}

Since $J_n^{\CC}$ is the homogeneous ideal of all polynomials that vanish on
$V_{\CC}(J_{n})$, the quotient ring $\CC[\z]/{J_n^{\CC}}$ is the coordinate ring
$\CC[V_{\CC}(J_{n})]$ of the variety $V_{\CC}(J_{n})$. The proof of the above lemma
shows that $\left(\tilde\sigma_n^\#\right)^{-1}(I^{\CC})\subseteq J_n^{\CC}/I_n^{\CC}$,
and the converse inclusion is obvious, therefore there is an induced
injective homomorphism $\sigma_{n}^\#:\CC[\z]/J_n^{\CC}\to \CC[\x,\y]/I^{\CC}$
satisfying $\sigma_{n}^\#(\z_{ij}+J_n^{\CC}) =\x_i\y_j+I^{\CC}$ for $1\le i,j\le
n$. The restriction of this homomorphism to the real quadratic forms is
then a (linear) bijective correspondence between quadratic forms from
$\RR[\z]/{J_{n}}$ and biforms from $\RR[\x,\y]_{2, 2}$ modulo $I$.

Recall from
Lemma \ref{lem:sos-bilin-modI} that a biform
$f\in \RR[\x,\y]_{2, 2}$
is a sum of squares modulo $I$ if and only if it is a sum of squares of biforms
from $\RR[\x,\y]_{1, 1}$ modulo $I$.

\begin{lemma} \label{lem:modI}
                A biform $f\in \RR[\x,\y]_{2, 2}$ of bidegree (2,2)
                is a sum of squares modulo $I$
                if and only if the quadratic form ${\sigma_{n}^\#}^{-1}(f)\in
\RR[\z]/J_{n}$ is a sum of squares.
\end{lemma}

\begin{proof}
        To prove the implication $(\Rightarrow)$ let
                \begin{equation}\label{exp-pt-1}
                        f=\sum_{j=1}^{j_0}  f_j^2+g\cdot \sum_{i=1}^n \x_i\y_i
                \end{equation}
        where $j_0\in \NN$, each $f_j\in\RR[\x,\y]_{1,1}$ and $g\in\RR[\x,\y]_{1,1}$.
	Note that all $f_j$ and $g$ 
        are in the image of $\sigma_{n}^\#$. Hence,
$${\sigma_{n}^\#}^{-1}(f)=\sum_{j=1}^{j_0}{\sigma_{n}^\#}^{-1}(f_j)^2+{\sigma_n^\#}^{-1}(g)\cdot
	\sum_{i=1}^n\z_{ii}$$
        is a sum of squares in $\RR[\z]/J_n$.

        It remains to prove the implication
        $(\Leftarrow)$. Since $f$ is in the image of $\sigma_{n}^\#$, it
follows from
                $${\sigma_{n}^\#}^{-1}(f)=\sum_{j=1}^{j_1} [h_j]^2,$$
        where $j_1\in \NN$ and $[h_i]$ is the equivalence class of $h_j\in
\RR[\z]$ in $\RR[\z]/J_{n}$, that
                $$f=\sum_{j=1}^{j_1} \sigma_{n}^\#([h_j])^2$$
        is a sum of squares in $\RR[\x,\y]/I$
        which proves $(\Leftarrow)$.
\end{proof}

        \begin{prop}\label{prop:degCodim}
        The variety $V_{\CC}(J_n)$ is smooth and is a nondegenerate subvariety of
$V_{\CC}\left(\sum_{i=1}^n\z_{ii}\right)$,
                $\dim V_{\CC}(J_n)=2n-3$, its codimension in the hyperplane
$V_{\CC}\left(\sum_{i=1}^n\z_{ii}\right)$ is $(n-1)^2$,
        and the degree of $V_{\CC}(J_n)$ is $\binom{2n-2}{n-1}$.
        \end{prop}

\begin{proof}

Note that $V_{\CC}(J_n)$ is (the projectivization of) the variety of all $n\times n$ matrices of rank 1
and trace 0. Suppose it is contained in a hyperplane of
$V_{\CC}(\sum_{i=1}^n\z_{ii})$. Then there exists a nonzero traceless matrix $M$
such that $\Tr(xy^TM)=y^TMx=0$ for all $x,y\in \CC^n$ satisfying
$y^Tx=0$. Taking $x=\e_i,y=\e_j$ for arbitrary distinct $i$ and $j$ we
get that $M$ is diagonal. Furthermore, taking $x=\e_i+\e_j,y=\e_i-\e_j$
for distinct $i$ and $j$, we get that $M$ is a scalar matrix. Since $\Tr
M=0$, it follows that $M=0$, which is a contradiction. Therefore
$V_{\CC}(J_n)$ is nondegenerate.\looseness=-1

Next, we compute the Hilbert polynomial for $V_{\CC}(J_n)$. We follow the
proof of the analogous result for the Segre variety in
\cite[p.~234]{Har92}. The space of polynomials of degree $d$ in
$\CC[\z]/J_n^{\CC}$ is isomorphic, via the restriction of the homomorphism
$\sigma_n^\#$, to the space of polynomials of bidegree $(d,d)$ in
$\CC[\x,\y]/I^{\CC}$. Its dimension is therefore
$${n+d-1\choose n-1}^2-{n+d-2\choose n-1}^2=\left(\frac{(d+1)\cdots
(n+d-2)}{(n-1)!}\right)^2(n-1)(n+2d-1).$$
This is a polynomial in $d$ with the leading term
$\frac{2n-2}{((n-1)!)^2}d^{2n-3}$, therefore $\dim V_{\CC}(J_n)=2n-3$ and
$\deg V_{\CC}(J_n)=\frac{(2n-2)(2n-3)!}{((n-1)!)^2}={2n-2\choose n-1}$. As
$V_{\CC}(\sum_{i=1}^n\z_{ii})$ is a hyperplane in $\mathbb{P}^{n^2-1}$, the result on codimension also follows.

It remains to prove smoothness. Note that the group $\GL_n$ acts on the
variety $V_{\CC}(J_n)$ of rank 1 traceless matrices by conjugation. Using the
Jordan normal form we see that the action is transitive, so it suffices
to prove that $\e_1\e_2^T$ is a smooth point of $V_{\CC}(J_n)$. To show this
we use the Jacobian criterion. Let $\Jac(\e_1\e_2^T)$ be the Jacobian
matrix for $V_{\CC}(J_n)$ at $\e_1\e_2^T$. The generators of the ideal $J_n^{\CC}$
are $\z_{ij}\z_{kl}-\z_{il}\z_{kj}$ where $i\ne k$ and $j\ne l$, and
$\sum_{i=1}^n\z_{ii}$. The gradient of $\z_{ij}\z_{kl}-\z_{il}\z_{kj}$ in
$\e_1\e_2^T$ is zero if $i\ne 1$ and $k\ne 1$ or if $j\ne 2$ and $l\ne
2$. On the other hand, the gradient of $\z_{12}\z_{kl}-\z_{1l}\z_{k2}$
in $\e_1\e_2^T$ is $\e_k\e_l^T$ for $k\ne 1$ and $l\ne 2$, and the gradient
of $\sum_{i=1}^n\z_{ii}$ is $\sum_{i=1}^ne_ie_i^T$. Clearly,
$\sum_{i=1}^ne_ie_i^T\in\CC^{n^2}$ is not a linear combination of $e_ke_l^T$ with $k\ne 1$ and $l\ne 2$, so $\Rank
\Jac(\e_1\e_2^T)=(n-1)^2+1$ and $\dim \ker\Jac(\e_1\e_2^T)=2n-2$. It
follows that the (projective) tangent space to $V_{\CC}(J_n)$ at $\e_1\e_2^T$
is $(2n-3)$-dimensional, which shows that $\e_1\e_2^T$ is smooth.
\end{proof}

        \begin{cor}\label{cor:minDeg}
        For $n\geq3$ the variety $V_{\CC}(J_n)$ is not of minimal degree,
        i.e., $\deg V_{\CC}(J_n)>1+\codim V_{\CC}(J_n).$
        \end{cor}

  We write
        \begin{eqnarray*}
                \Pos(V(J_{n})) &=& \left\{f\in \RR[\z]/J_{n}\colon f(z)\geq 0\quad
\text{for all } z \in V(J_{n})
                        \right\},\\
                \Sq(V(J_{n})) &=& \left\{f\in \RR[\z]/J_{n}\colon f=\sum_{i=1}^k f_i^2\quad
\text{for some }k\in\NN\text{ and }
                        f_i \in \RR[\z]/J_{n} \right\},
        \end{eqnarray*}
for the cone of nonnegative polynomials and the cone of sums of squares
from $\RR[\z]/J_{n}$, respectively.

  For $n\ge 3$, \cite[Procedure 3.3]{BSV} yields an explicit
construction of nonnegative quadratic forms from $\RR[\z]/{J_{n}}$ that
are not sums of squares forms
  starting from random input data. We now 
{turn this procedure,
specialized to our context, into a probabilistic 
(Las Vegas)
algorithm.}

\begin{algorithm}\label{algo}\rm
        Let $n\ge 3$, $d=2n-3=\dim V(J_n)$, and $e=(n-1)^2=\codim V(J_n).$
         To obtain a quadratic form in $ \Pos(V(J_{n}))\setminus
\Sq(V(J_{n}))$ proceed as follows:

        \begin{enumerate}[label={\rm Step \arabic*}]
        \item\label{it:1} Construction of linear forms $h_0,\ldots, h_d$.
        \begin{enumerate}[label={\rm Step 1.\arabic*}]
                \item\label{it:1.1} 
                        Choose $e+1$ random pairwise orthogonal $x^{(i)}\in \RR^n$ and
$y^{(i)} \in \RR^n$ and calculate their Kronecker tensor products
                                $z^{(i)}=x^{(i)}\otimes y^{(i)}\in \RR^{n^2}$.
                \item\label{it:1.2} 
                        Choose $d$ random vectors $v_1,\ldots v_d \in \RR^{n^2}$ from the
kernel of the matrix
                        $$\begin{pmatrix} z^{(1)} & \ldots & z^{(e+1)} \end{pmatrix}^\ast,$$
                        and form  the linear forms $$h_j(\z)=v_j^\ast\cdot \z  \in
\RR[\z]\quad \text{for } j=1,\ldots,d.$$
                        If the number of points in the intersection
                                $
                                        \ker(
                                        \begin{pmatrix}
                                        v_1 & \ldots & v_d
                                        \end{pmatrix}^\ast)
                                        \bigcap V(J_{n})
                                $
                        is not equal to $\deg (V_{\CC}(J_{n}))=\binom{2n-2}{n-1}$ or if the points
in the intersection are not in linearly general position, then repeat
\ref{it:1.1}.
                \item\label{it:1.3}
                        Choose a random vector $v_0$ from the kernel of the matrix
                        $$\begin{pmatrix} z^{(1)} & \ldots & z^{(e)} \end{pmatrix}^\ast.$$
                        (Note $z^{(e+1)}$ is omitted.)
                        The corresponding linear form $h_0$ is
                                $$h_0(\z)=v_0^\ast\cdot \z \in \RR[\z].$$
                        If $h_0$ intersects  $h_1$, $\ldots$, $h_d$ in more than
                        $e$ points on $V(J_{n})$, then repeat \ref{it:1.3}.\looseness=-1
\end{enumerate}

\smallskip
                \noindent Let $\mathfrak a$ be the ideal in $\RR[\z]/{J_{n}}$
generated by $h_0,h_1,\ldots,h_{d}$.

                \smallskip
        \item\label{it:2} Construction of a quadratic form $f\in
\left(\RR[\z]/{J_{n}}\right)\setminus \mathfrak a^2$.
        \begin{enumerate}[label={\rm Step 2.\arabic*}]

                \item\label{it:2.1}
                Let $g_1(\z),\ldots ,g_{\binom{n}{2}^2}(\z)$ be the generators of the
ideal $I_{n}$, i.e., the $2\times 2$ minors $\z_{ij}\z_{kl}-\z_{il}\z_{kj}$
for $1\le i<k\le n,1\le j<l\le n$.
Set $g_0=\sum_{i=1}^n \z_{ii}$.
                For each $i=1,\ldots ,e$ compute a basis $\{w_1^{(i)},\ldots
,w_{d+1}^{(i)}\}\subseteq \RR^{n^2}$ of the kernel of the matrix
                $$\begin{pmatrix}
                \nabla g_0(z^{(i)})^* \\ \vdots \\ \nabla g_{\binom{n}{2}^2}(z^{(i)})^*
                \end{pmatrix}.$$
                (Note that this kernel is always $(d+1)$-dimensional, since the
variety $V_{\CC}(J_{n})$ is $d$-dimensional (in $\PP^{n^2-1}$) and smooth.)

                \item\label{it:2.2}
Choose a random vector $v\in \RR^{n^4}$ from
the intersection of the
                kernels of the matrices
                $$\begin{pmatrix}
                z^{(i)} \otimes w_1^{(i)} & \cdots & z^{(i)}\otimes w_{d+1}^{(i)}
                        \end{pmatrix}^*\quad \text{for }i=1,\ldots,e
                $$
                with the kernels of the matrices
                $$\begin{pmatrix}
                \e_i\otimes \e_j-\e_j\otimes \e_i
                \end{pmatrix}^*\quad \text{for } 1\le i<j\le n^2.$$
                (The latter condition ensures $v$ is a symmetric tensor in
$\RR^{n^2}\otimes \RR^{n^2}$. Note also the point $z^{(e+1)}$ is omitted.)

                For $1\leq i,k\leq n$ and $1\leq j,l\leq n$      denote
                $$E_{ijkl}=(\e_i\otimes \e_j)\otimes (\e_k\otimes \e_l)+(\e_k\otimes
\e_l)\otimes (\e_i\otimes \e_j)\in \mathbb{R}^{n^4}.$$
                        If $v$ is in
                                \begin{align*}\Span&\Big(\left\{v_i\otimes v_j+v_j\otimes v_i\colon 0\leq i\leq
j\leq d\right\}\\&\bigcup
                                \left\{E_{ijkl}-E_{ilkj};1\leq i<k\leq n,1\leq j<l\leq n\right\}\\
			&\bigcup\left\{\sum_{i=1}^n((\e_i\otimes \e_i)\otimes (\e_j\otimes \e_k)+(\e_j\otimes \e_k)\otimes (\e_i\otimes \e_i)); 1\le j,k\le n\right\}\Big),\end{align*}
                        then repeat {\rm\ref{it:2.2}}.
                        Otherwise the  quadratic form
                                        $$f(\z)=v^\ast\cdot (\z\otimes\z)\in \RR[\z]/J_{n},$$
                        does not belong to $\mathfrak a^2.$
        \end{enumerate}

        \item\label{it:3} Construction of a quadratic form in
$\mathbb{R}[\z]/J_{n}$
         that is nonnegative  but not sos.\\[1mm]
         Calculate the greatest $\delta_0>0$ such that
                                $\delta_0 f+\sum_{i=0}^dh_i^2$
                        is nonnegative on $V(J_{n})$. Then for every $0<\delta\leq\delta_0$
                        the quadratic form
                                \beq\label{eq:Fdelta}F_\delta=\delta f+\sum_{i=0}^dh_i^2\eeq
                        is nonnegative on $V(J_{n})$ but is not a sum of squares.
        \end{enumerate}
\end{algorithm}

\subsection{Correctness of Algorithm \ref{algo}}
The main ingredient in the proof is the theory of minimal degree
varieties as developed in \cite{BSV}.
Since
$V_{\CC}(J_n)$ is not of minimal degree for $n\geq 3$
by Proposition \ref{prop:degCodim}, $\Sq(V(J_{n}))\subsetneq
\Pos(V(J_{n}))$. Hence results of \cite[Section 3]{BSV} apply; their
Procedure 3.3 adapted to our set-up is Algorithm \ref{algo}. While
\ref{it:1} and \ref{it:3} follow immediately from the corresponding
steps in \cite[Procedure 3.3]{BSV}, we note for \ref{it:2} that
``vanishing to the second order at $z^{(i)}$'' means $f(z^{(i)})=0$ and
$\nabla f(z^{(i)})\in \Span\left\{\nabla g_j(z^{(i)})\colon 0\leq j\leq
\binom{n}{2}^2\right\}$.
Since $f\not \in \mathfrak a^2$,
the quadratic form $\delta f+\sum _{i=0}^dh_i^2$ is never a sum of
squares,  while it is nonnegative on $V(J_{n})$ for sufficiently small
$\delta >0$ by the positive definiteness of the Hessian of $\sum
_{i=0}^dh_i^2$ at its real zeros $z^{(1)},\ldots ,z^{(e)}$, see the
proof of the correctness of Procedure 3.3 in \cite{BSV}. 

\subsection{{Towards an implementation}}\label{ssec:implement}
{Note that
the verification in \ref{it:1.2} is computationally difficult
(but can be performed
using Gr\"obner basis if $n$ is small). However, since
all steps in the algorithm are performed with random data, all the
generic conditions from \cite[Procedure 3.3]{BSV} are satisfied with
probability 1. Hence, 
by omitting the verification in \ref{it:1.2}, 
Algorithm \ref{algo} 
becomes a Monte Carlo algorithm and
yields a correct output with probability 1}. \looseness=-1
\ref{it:1} and  \ref{it:2} are easily implemented as
they only require linear algebra.
On the other hand, \ref{it:3} is computationally difficult; testing
nonnegativity even of low degree polynomials is NP-hard, cf.~\cite{LNQY09}.

Several algorithms are available to check nonnegativity of polynomials.
Those using symbolic methods such as
quantifier elimination or cylindrical algebraic decomposition only work
for small problem sizes. We employ numerical methods based on the
Positivstellensatz \cite{BCR98}. To reduce the number of equality
constraints, we rewrite the quadratic form $F_\delta(\z)$ from
\eqref{eq:Fdelta}
into $\x,\y$ variables, obtaining a biquadratic form
we denote by a slight abuse of notation by $F_\delta(\x,\y)$. 

\begin{prop}\label{prop:psatz}
For $f\in\RR[\x,\y]$ the following are equivalent:
\begin{enumerate}[label={\rm(\roman*)}]
\item\label{it:p1} $f\geq0$ on $V(I)$;
\item\label{it:p2} there exist sum of squares
$\sigma_1,\sigma_2\in\RR[\x,\y]$ such that
\beq\label{eq:psatz}\sigma_1f-\sigma_2\in I \quad\text{ and }\quad
\sigma_1\not\equiv0 \text{ on }V(I).\eeq
\end{enumerate}
\end{prop}

\begin{proof}
Assume \ref{it:p2} holds.
 From \eqref{eq:psatz} it follows that $f\geq0$
on $S=V(I)\setminus V(\sigma_1)$. 
Since $V(I)$ is irreducible,
$S$ is Zariski dense in $V(I)$. Since $S$ is also open in $V(I)$, it is dense in $V(I)$ also in the Euclidean topology and hence \ref{it:p1}
holds. Conversely,
suppose \ref{it:p1} holds. By the Positivstellensatz
(e.g.~\cite[Corollary 4.4.3]{BCR98}), there is $m\in\NN$ and
sums of squares $\sigma_1,\sigma_2$ such that
$f\sigma_1-\sigma_2-f^{2m}\in I$.
Assume $\sigma_1=0$ on $V(I)$, then $\sigma_1\in I$
since $I$ is the vanishing ideal of $V(I)$ (see \S \ref{sec:prelim}), 
whence $\sigma_2+f^{2m}\in I$. Thus, again by
the real radical property of $I$, $f\in I$. In this case we may simply pick
$\sigma_1=1$ and $\sigma_2=0$.
\end{proof}

We apply Proposition \ref{prop:psatz} to $F_\delta$ from
\eqref{eq:Fdelta} to search for a $\delta>0$ making $F_\delta(\x,\y)$
nonnegative on $V(I)$.
Let $\delta>0$ be fixed and suppose
the degree of $\sigma_1$ is $\leq2d$. Then the ideal
constraint in \eqref{eq:psatz} immediately converts
into a linear matrix inequality and thus feeds
into a semidefinite program (SDP) that can be solved with standard
solvers \cite{WSV00}.
(Here homogeneity of $F_\delta$ and $I$ enter.
Both $\sigma_j$ can be assumed to be homogeneous, and
$\deg\sigma_2\leq2d+4$.)
To implement the non-equality
constraint in \eqref{eq:psatz}, we pick a random point $(x_0,y_0)\in
V(I)$ and set $\sigma_1(x_0,y_0)=1$.
Our implementation
uses bisection,
sets $d=1$ and
starts with, say, $\delta=1$. Then solve the described feasibility SDP.
If it has a solution, stop.
If not, replace $\delta$ by $\delta/2$ and try again.
If no solution has been found with $\delta$ greater
than some prescribed tolerance, increase $d$, and reset
$\delta=1$. Then repeat the process.
By Proposition \ref{prop:psatz} and the construction
of $F_\delta$ the algorithm will eventually produce
a certificate of nonnegativity for some $\delta>0$.
We refer to \cite{Abi} for a numerical comparison of
polynomial optimization choices for a similar problem.

As in \cite{KMSZ19} (see also \cite{Abi})
it might be possible to apply standard techniques
\cite{PP08,CKP15} to turn obtained {numerical}
sum of squares
certificates into symbolic ones.

\subsection{Example}
In this section we give an explicit numerical example
of a ``proper'' \crp map $\widetilde \Phi:M_3(\RR)\to M_3(\RR)$
built using our {ad-hoc implementation of} Algorithm \ref{algo} {in Wolfram Mathematica}.
Working with rational random data as per Algorithm \ref{algo} quickly leads to very large denominators with bad conditioning, necessitating working with floating point numbers.
Let
\begin{multline*}
p_\Phi(\x,\y)=
75.356\ \x_1^2 \y_2^2+35.3881\ \x_1^2 \y_3^2-65.2694\ \x_1^2
    \y_2 \y_3+89.2972\ \x_2 \x_1 \y_2^2\\
    -19.9103\ \x_3 \x_1
    \y_2^2 +96.593\ \x_2 \x_1 \y_3^2-47.7404\ \x_3 \x_1
    \y_3^2-80.1036\ \x_2 \x_1 \y_2 \y_3\\ +56.4942\ \x_3 \x_1 \y_2
    \y_3+37.6343\ \x_2^2 \y_1+6.96833\ \x_3^2
    \y_1^2+17.7278\ \x_2 \x_3 \y_1^2\\+38.8145\ \x_2^2
    \y_2^2+23.0293\ \x_3^2 \y_2^2+37.1699\ \x_2 \x_3
    \y_2^2+66.6118\ \x_2^2 \y_3^2\\+22.9845\ \x_3^2
    \y_3^2-66.1642\ \x_2 \x_3 \y_3^2-2.03483\ \x_2^2 \y_1
    \y_2+25.0232\ \x_3^2 \y_1 \y_2\\+35.2335\ \x_2 \x_3 \y_1
    \y_2+1.70127\ \x_2^2 \y_1 \y_3-32.1772\ \x_3^2 \y_1
    \y_3-33.3246\ \x_2 \x_3 \y_1 \y_3\\+9.37496\ \x_2^2 \y_2
    \y_3-41.4656\ \x_3^2 \y_2 \y_3+11.4857\ \x_2 \x_3 \y_2 \y_3.
    \end{multline*}
The corresponding linear map $\Phi:\sa3\to\sa3$
is as follows:
{\small
\[
\Phi(E_{11})=
\begin{bmatrix}
  0. & 0. & 0. \\
  0. & 75.356 & -32.6347 \\
  0. & -32.6347 & 35.3881 \\
  \end{bmatrix},
  \quad
  \Phi(E_{22})=
\begin{bmatrix}
37.6343 & -1.01742 & 0.850636 \\
  -1.01742 & 38.8145 & 4.68748 \\
  0.850636 & 4.68748 & 66.6118
  \end{bmatrix},
  \]
  \[
  \Phi(E_{33})=
\begin{bmatrix}
6.96833 & 12.5116 & -16.0886 \\
  12.5116 & 23.0293 & -20.7328 \\
  -16.0886 & -20.7328 & 22.9845
  \end{bmatrix},
\quad
\Phi(E_{12}+E_{21})=
\begin{bmatrix}
0. & 0. & 0. \\
  0. & 89.2972 & -40.0518 \\
  0. & -40.0518 & 96.593
  \end{bmatrix},
  \]
  \[
\Phi(E_{13}+E_{31})=
  \begin{bmatrix}
  0. & 0. & 0. \\
  0. & -19.9103 & 28.2471 \\
  0. & 28.2471 & -47.7404
  \end{bmatrix},
\quad
  \Phi(E_{23}+E_{32})=
\begin{bmatrix}
     17.7278 & 17.6168 & -16.6623 \\
  17.6168 & 37.1699 & 5.74284 \\
  -16.6623 & 5.74284 & -66.1642
   \end{bmatrix}.
\]
}
The polynomial $p_\Phi$ is nonnegative
on $V(I)$
  but not a sum of squares modulo $I$. Equivalently, 
  an arbitrary $*$-linear extension
  $\widetilde \Phi:M_3(\RR)\to M_3(\RR)$ of $\Phi$ is
a proper \crp map, e.g., $\widetilde \Phi$ is trivial on antisymmetric matrices.
This example was produced using Algorithm \ref{algo}
starting with the points
\[
\begin{bmatrix} x^{(1)} & y^{(1)} \\[1mm]
x^{(2)} & y^{(2)}\\[1mm]
x^{(3)} & y^{(3)}\\[1mm]
x^{(4)} & y^{(4)}\\[1mm]
x^{(5)} & y^{(5)}
\end{bmatrix}=
\left[
\begin{array}{rrr|rrr}
-\frac{3}{2} & 1 & \frac{3}{2} & -\frac{21}{2} &
    -\frac{3}{2} & -\frac{19}{2} \\[1mm]
  \frac{1}{3} & 0 & -3 & -24 & 9 & -\frac{8}{3} \\[1mm]
  1 & -1 & -\frac{2}{3} & \frac{14}{3} & -\frac{2}{3}
    & 8 \\[1mm]
  2 & -1 & \frac{1}{2} & -4 & \frac{9}{2} & 25 \\[1mm]
  -\frac{3}{2} & \frac{3}{2} & -\frac{3}{2} &
    \frac{3}{2} & \frac{3}{2} & 0 \\
    \end{array}
\right],
\]
where each $x^{(i)},y^{(i)}\in\RR^3$.

\section*{Acknowledgments}
The authors thank the anonymous referees and the editors 
for their careful reading and insightful comments that have greatly improved this manuscript.

\end{document}